\documentclass[11pt,a4paper]{article}
\usepackage[T1]{fontenc}
\usepackage{inputenc}
\usepackage{authblk}
\usepackage{indentfirst}
\usepackage{geometry}
\usepackage{amsmath}
\usepackage{amsthm}
\usepackage{graphicx}
\usepackage{mathrsfs}
\usepackage{epstopdf}
\usepackage{mathptmx}
\usepackage{amssymb}
\DeclareMathAlphabet{\mathcal}{OMS}{cmsy}{m}{n}
\usepackage{amsfonts}
\usepackage{upgreek}
\usepackage{bm}
\usepackage{indentfirst}
\usepackage{multirow}
\usepackage{booktabs}
\usepackage[colorlinks,linkcolor=blue,anchorcolor=blue,citecolor=blue,]{hyperref}
\numberwithin{equation}{section}

\allowdisplaybreaks
\usepackage{mdframed}
\newtheorem{theorem}{Theorem}

\newtheorem{remark}{Remark}%
\newtheorem{corollary}[theorem]{Corollary}
\theoremstyle{thmstyletwo}%
\theoremstyle{thmstylethree}%
\newtheorem{lemma}{Lemma}
\theoremstyle{definition} %
\newtheorem{assumption}{Assumption} %
\allowdisplaybreaks[4]
\usepackage{subfigure}  
\usepackage{caption}
\usepackage{hyperref}
\hypersetup{
	colorlinks=true,
	linkcolor=blue,
	filecolor=blue,      
	urlcolor=blue,
	citecolor=blue
}
\usepackage{graphicx}

\title{A time-decoupling scheme for mixed-dimensional poroelastic models with fractures}
\author[a]{Dongchun Tang \thanks{Email: dongchun\_tang@163.com}}
\author[c]{Linshuang He \thanks{Email: helinshuang@cdut.edu.cn}}
\author[b]{Shuyu Sun \thanks{Corresponding author. Email: suns@tongji.edu.cn}}
\author[a]{Minfu Feng \thanks{Email: fmf@scu.edu.cn}}
\affil[a]{School of Mathematics, Sichuan University, Chengdu, China}
\affil[b]{School of Mathematical Sciences, Tongji University, Shanghai, China}
\affil[c]{School of Mathematical Sciences, Chengdu University of Technology, Chengdu, China}

\date{}

\begin{document}
	\maketitle
	\begin{abstract}
	We propose a locking-free decoupling method for a mixed-dimensional poroelasticity model with fractures. By introducing the total pressure, the fractured Biot system is reformulated as a four-field formulation involving the displacement, total pressure, matrix pressure, and fracture pressure. We establish an energy dissipation law for the continuous model, which shows its consistency with the second law of thermodynamics. Based on this formulation, a time-decoupled scheme is developed. At the initial time step, a fully coupled scheme is employed, while for subsequent time steps, the flow problem is solved first, followed by the mechanics problem. A stabilization term is incorporated into the mechanical equation to help reduce the restrictions imposed on the model parameters in the stability analysis. Energy stability is established for the semi-discrete scheme. For the spatial discretization, the displacement and total pressure are approximated by the Taylor--Hood element, while Lagrange finite elements are used for the matrix and fracture pressures. A fully discrete decoupled scheme is then constructed. Energy stability and error estimates are derived for the fully discrete scheme, and the method is shown to be locking-free. Numerical experiments are presented to support the theoretical results.\\
	
	\end{abstract}
	\thispagestyle{empty}
	\section{ Introduction}
	Fluid flow in fractured poroelastic media is governed by coupled multiphysical processes, including fluid seepage and solid deformation. Such coupled models are widely used in geotechnical engineering and reservoir flow analysis \cite{Jha2014,TARIQ2023127677}. 
	Fractures often act as mechanically weak regions and significantly influence flow behavior \cite{ALMANI2024117253}. Fluid injection into the subsurface alters the pore pressure field, which in turn modifies the in-situ stress state \cite{doi:10.1137/18M1203754}.
	In linear settings, the coupling between fluid flow and geomechanics is described by the Biot model. This model characterizes the porous skeleton using linear elasticity, where the fluid pressure contributes to the effective stress acting on the solid matrix.
	
	However, the presence of fractures requires special treatment. In mathematical models, fractures are commonly represented as two-dimensional surfaces embedded in a three-dimensional porous matrix, leading to a mixed-dimensional formulation \cite{Berge2020,BONALDI202140,BONALDI20211741,boon2023mixed,girault2015}. Fractures are typically thin and elongated. As a result, tangential flow along the fracture surface dominates, while transverse flow is negligible \cite{KUMAR2020109138}. Therefore, flow within fractures is usually described by the Darcy's law defined on the fracture surface. This setting introduces multiple coupled physical processes across domains of different dimensions \cite{M2005}. The resulting mathematical model consists of the Biot equations for flow and geomechanics in the porous matrix, coupled to Darcy flow along the fracture surface via suitable interface conditions \cite{alboin2002}.
	
	Extensive research has been devoted to the numerical simulation of the classical crack-free Biot system. Murad and Loula et al. \cite{CG-MuradLoula-1992-CMAME} presented a coupled finite element formulation for the two-field Biot model, where the primary variables are the solid displacement and the fluid pressure. It also establishes a theoretical framework for the mixed finite element (MFE) methods. Despite these developments, several numerical challenges remain. When the Lamé coefficient tends to infinity (equivalently, Poisson's ratio approaches 0.5), volumetric locking may occur \cite{Poro-Yi-2017-SIAMJNA}. In addition, when the storage coefficient, permeability, and time step simultaneously approach zero, pressure locking can arise \cite{Poro-PhillipsWheeler-2009-CG}.
	To overcome these limitations, a variety of finite element methods have been proposed. For the two-field formulation, discontinuous Galerkin (DG) methods 	\cite{Biot-ChenLuoFeng-2013-AMC,DG-RiviereTanThompson-2017-CMA} and stabilization techniques \cite{HE2025129285,Biot-RodrigoGasparHuZikatanov-2016-CMAME} have been developed to alleviate locking effects.
	In addition, multi-field formulations have been introduced to further improve numerical performance. Typical examples include three- and four-field methods, such as H(div)-conforming formulations \cite{Hdiv-HongKraus-2018-ETNA,Hdiv-KanschatRiviere-2018-JSC,Hdiv-ZengCaiWang-2019-EAJAM} and hybrid discontinuous Galerkin (HDG) methods \cite{H(div)CDG-HeFengGuo-2023-CMA,HDG-KrausLedererLymbery-2021-CMAME} based on fluid flux variables.
   Another class of three-field approaches is based on the introduction of the total pressure as an additional variable, leading to total pressure based mixed finite element methods \cite{TP3-OyarzuaRuizBaier-2016-SIAMJNA,zhao2025optimally} and DG methods \cite{he2024analysis}.
	Among these approaches, the four-field formulation with fluid flux and total pressure has received considerable attention \cite{he2026time,TP4-OyarzuaRhebergenSolano-2021-ESAIMMMNA,TP4-QiSeshajyerWang-2021-ERA}. This model can be interpreted as a combination of a generalized Stokes sub-problem and a mixed Darcy sub-problem. It inherits the advantages of the aforementioned three-field formulations and supports various coupled discretization strategies \cite{TP4-KumarOyarzuaRuiz-2020-ESAIMMMNA}. However, the introduction of additional variables increases the computational complexity.
	Therefore, three-field formulations offer advantages in terms of parameter robustness and reduced number of model parameters. In this work, we adopt a four-field mixed-dimensional formulation based on the total pressure. The resulting system retains the main advantages of multi-field methods, particularly their robustness with respect to parameter variations, while reducing the number of unknowns.
	
	For mixed-dimensional linearized poroelastic models with fractures, Wheeler et al. considered two discretization approaches, namely the continuous Galerkin (CG) method and the MFE method, for the fluid flow equations. 
	They provided a rigorous numerical analysis for the CG method. In contrast, for the MFE method, only numerical experiments were reported, without theoretical analysis  \cite{girault2015}. Girault et al. extended the classical fixed-stress splitting algorithm to fractured models and proved its convergence to a unique weak solution \cite{Girault2016}. In 2019, Wheeler et al. introduced a mixed formulation within the framework of a fixed-stress iterative method, for which stability and error estimates were rigorously established \cite{girault2019mixed}. More recently, Kumar et al. investigated linearized fractured poroelastic models using both single- and multi-rate undrained split iterative schemes, and provided corresponding convergence analyses \cite{f2022}. In a subsequent work, they applied a multi-rate fixed-stress splitting scheme to the linearized fractured Biot model and further established its convergence \cite{ALMANI2024117253}.
	However, the aforementioned methods are mainly based on iterative schemes. Compared with these approaches, decoupled algorithms offer higher computational efficiency and reduced memory requirements \cite{he2026time,Zhao2025}. Bonaldi et al. provided a comprehensive numerical analysis within the gradient discretization framework for a mixed-dimensional poromechanical model, explicitly accounting for discontinuous fluid pressures and frictionless contact at matrix-fracture interfaces \cite{BONALDI2024}. In this work, we adopt a decoupling strategy to solve the coupled system, where the flow equations are solved first, followed by the mechanics equations.
 
	In this work, the proposed model is a four-field mixed-dimensional poroelastic formulation with displacement $\mathbf{w}$, total pressure $\xi$, matrix pressure $p$, and fracture pressure $p_c$ as the primary variables. Building upon this formulation, we propose a time-decoupling strategy that enables efficient computation and effectively overcomes poroelasticity locking. At the initial time step, all variables are solved in a fully coupled form. From the second time step onward, the system is split into two subproblems: A fluid flow problem and a mechanics problem. These are solved sequentially at each time level, with the pressure field computed first, followed by the displacement field. To improve the robustness of the formulation, we incorporate a stabilization term into the mechanical equation. This modification may help alleviate certain restrictions on the model parameters. For instance, the analysis in \cite{almani2016efficient} relies on a condition of the form $G \geq 1/(c_f C^*)$. For spatial discretization, we employ Taylor--Hood elements for the mechanical equations and Lagrange finite elements for the fluid flow equations, leading to a fully discrete decoupled scheme. We further prove optimal convergence and demonstrate that the proposed method is locking-free. Finally, numerical experiments confirm the effectiveness of the proposed approach. In particular, compared with fully coupled schemes, the decoupled formulation achieves significant computational speedup. 
	
	The paper is organized as follows: 
	In Section \ref{s2}, we introduces the domain setup, present the mathematical model, and derive its variational formulation. The energy stability of the continuous problem is also established. 
	In Section \ref{s3}, we devoted to presenting a semi-discrete time-decoupled scheme. A discrete energy functional is defined, and the stability of the semi-discrete scheme is proved. 
	In Section \ref{s4}, we design a fully discrete scheme using Taylor--Hood elements for the mechanical equations and Lagrange finite elements for the flow equations. Optimal convergence is established, and the scheme is shown to be free of locking. 
	In Section \ref{s5}, we presents several numerical experiments to validate our theoretical analysis.
	\section{Mathematical model and weak formulation}\label{s2}
	\subsection{Domain}
	
	We consider a fractured porous medium $\Omega \in \mathbb{R}^d$, where $d=2$ or $d=3$, which is linear elastic, homogeneous, and isotropic. The medium is assumed to be saturated with a compressible single-phase fluid. The fractures are treated as interfaces that are not necessarily planar, denoted by $\mathcal{C}$. To simplify the modeling process and avoid dealing with curved elements, we assume that both $\partial \Omega$
	and the fracture $\mathcal{C}$ are
	polygonal surfaces. Figure \ref{domain} illustrates the matrix domain, the embedded fracture network, and the orientation of the fracture interfaces. 
	
	In our analysis, although the fracture does not propagate (i.e., the crack front remains stationary), the fracture width can still change over time due to fluid injection into the crack and fluid leakage out of the crack into the surrounding medium. We assume that the fracture width is small enough (compared to other relevant length scales associated with the fracture) to allow the use of Reynolds' lubrication equation for modeling the flow within the fracture (see \cite{yew1997}).
	
	\begin{figure}[h]
		\centering
		\subfigure{
			\begin{minipage}[t]{0.85\textwidth}
				\centering
				\includegraphics[scale=0.35]{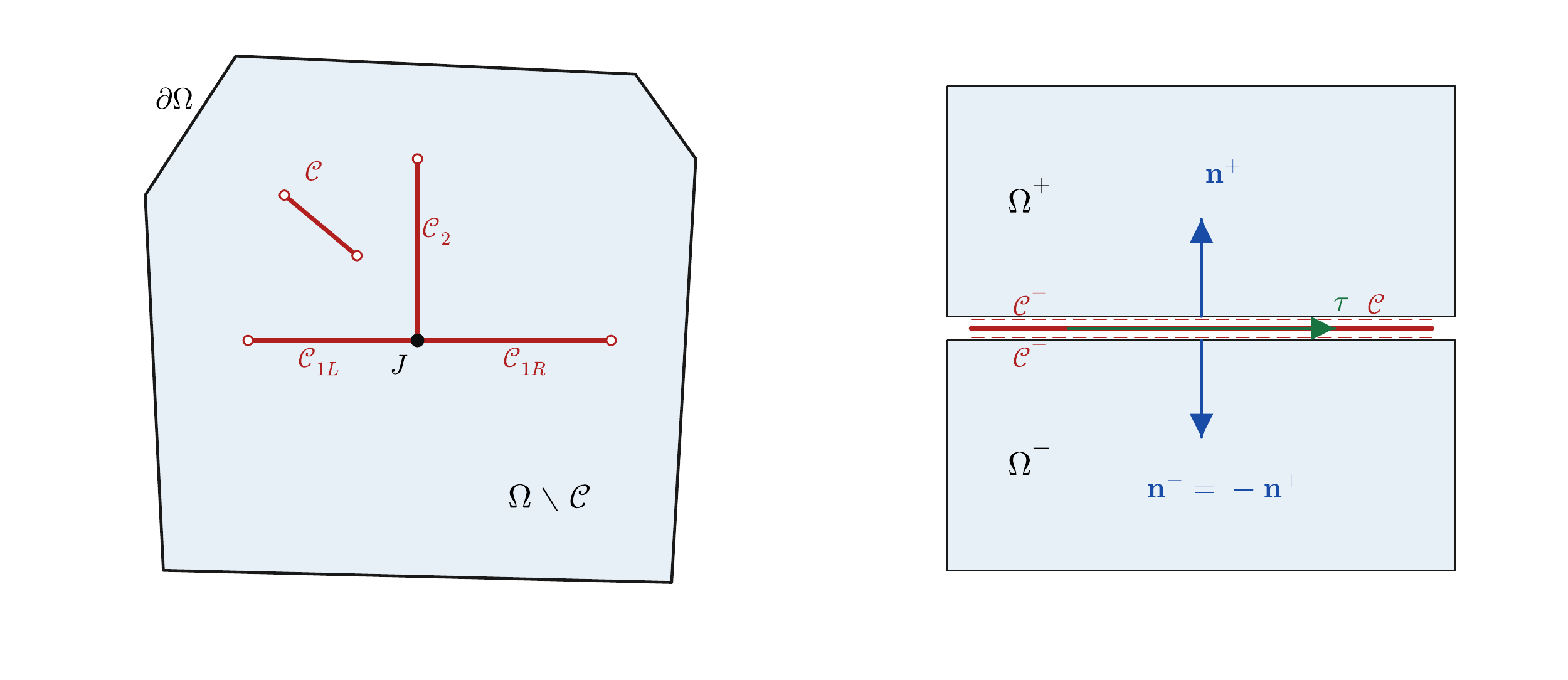}
			\end{minipage}
		}
		\centering
		\caption{Schematic illustration of the fractured domain: (a) The matrix domain and fracture network; (b) The local orientation and traces on a fracture interface.}
		\label{domain}
	\end{figure}
	\subsection{Model formulation}
\subsubsection{Equations in $\Omega\setminus\mathcal{C}$}
	The displacement of the solid is modeled in $\Omega\setminus\mathcal{C}$, a quasi-static Biot model is assumed, where the second order time derivative for the displacement are ignored.
The constitutive relation for the Cauchy stress tensor $\sigma ^{\mathrm{por}}$
 is given by 
 \begin{align*}
 \sigma ^{\mathrm{por}}( \mathbf{w} , p) = \sigma ( \mathbf{w} ) - \alpha p \mathbf{I},
 \end{align*}
 where $\mathbf{I}$ is the identity tensor, $\mathbf{w}$ is the solid’s displacement, $\alpha> 0$ is the dimensionless Biot coefficient, $p$ is the pressure in the matrix, $\sigma $ is the effective linear elastic stress tensor:
 \begin{align*}
	\sigma(\mathbf{w})=\lambda(\nabla\cdot\mathbf{w})\mathbf{I}+2G\epsilon(\mathbf{w}),
 \end{align*}
 where $\lambda> 0$ and $G> 0$ are the Lam\'{e} constants and  $\epsilon(\mathbf{w})=\frac{1}{2}(\nabla\mathbf{w}+\nabla\mathbf{w}^T)$  is the infinitesimal strain tensor (symmetrized gradient of displacements). Then the balance of linear momentum in the solid reads
	\begin{align}\label{w}
	- \nabla \cdot \sigma ^{\mathrm{por}}( \mathbf{w} , p) = \mathbf{f},
	\end{align}
	where $\mathbf{f}$ is the body force.
	
	For the fluid, we assume a linearized slightly compressible single-phase flow model for the fluid in the matrix $\Omega\setminus\mathcal{C}$. To derive the linearized model (as shown in \cite{ALMANI2024117253,girault2015}), we consider the mass balance equation in $\Omega\setminus\mathcal{C}$ reads
	\begin{align}\label{rho}
		\frac{\partial\left(\phi\rho\right)}{\partial t}+\nabla\cdot(\rho\mathbf{u})=g_s,\quad
		\mathbf{u}=-\frac{\mathbf{K}}{\eta}( \nabla p-\rho \mathbf{g}),
	\end{align}
	where $\rho$ is the fluid density, $\mathbf{u}$ is the velocity of the fluid, $\eta > 0$ is the constant fluid viscosity, $\rho \mathbf{g}$ is gravity loading term, $g_s$ is a mass source or sink term taking into account injection into or out of the reservoir. $\mathbf{K}$ is the permeability tensor in the matrix, assumed to be symmetric, bounded, uniformly positive definite in space and constant in time, there exists $0<k_1<k_2$ satisfying $k_{1 } \boldsymbol\zeta^T \boldsymbol\zeta \leq \boldsymbol\zeta^T \mathbf{K} \boldsymbol\zeta \leq k_{2 } \boldsymbol\zeta^T \boldsymbol\zeta,~\forall \boldsymbol\zeta \in \mathbb{R}^d$. For simplicity, we write down a $\mathbf{K} = k\mathbf{I},~ \alpha_1 = \eta k^{-1}.$
	 In addition, the porosity is given by $\phi=\phi_0 +\alpha\nabla\cdot\mathbf{w}+c_0p$, where $\phi_0$ is denotes the initial porosity and $c_0\ge 0$ is a storage coefficient. Furthermore, the fluid density $\rho=\rho_f(1+c_f(p-p_r))$, where $\rho_f$ is the constant reference density, $c_f$ is the fluid compressibility, and $p_r$ is the  reference pressure (assumed to be constant). Accordingly, the mass balance equation is expressed as
	 \begin{align*}
	 	\frac{\partial}{\partial t}\left(\rho_{f}\left(1+c_{f}\left(p-p_{r}\right)\right)\left(\phi_{0}+\alpha \nabla \cdot \mathbf{w}+c_0 p\right)\right)+\nabla \cdot\left(\rho_{f}\left(1+c_{f}\left(p-p_{r}\right)\right) \mathbf{u}\right) &=g_{s},
	 \end{align*}
which can be written as
	 \begin{align*}
	 &	\rho_{f}\left(c_0\left(1+c_{f}\left(p-p_{r}\right)\right)+c_{f}\left(\phi_{0}+\alpha \nabla \cdot \mathbf{w}+c_0 p\right)\right) \frac{\partial}{\partial t} p+\rho_{f} \alpha\left(1+c_{f}\left(p-p_{r}\right)\right) \nabla \cdot \frac{\partial}{\partial t} \mathbf{w}\\
	 &\qquad	+\nabla \cdot\left(\rho_{f}\left(1+c_{f}\left(p-p_{r}\right)\right) \mathbf{u}\right) =g_{s} .
	 \end{align*}
	 Based on the small magnitude of $c_f$ (typically on the order of $10^{-5}$ or
	$10^{-6}$), we adopt the following approximations:
	 \begin{align*}
	 	&c_0(1 + c_f(p - p_r)) +	c_f (\phi_0 + \alpha \nabla \cdot \mathbf{u} + c_0 p) \approx c_0,\quad
	 	\rho_{f} (1 + c_f(p - p_r))\alpha\approx \rho_{f} \alpha, \\
	 &	\rho_{f} (1 + c_f(p - p_r)) \mathbf{u} \approx \rho_{f} \mathbf{u}, \quad
	 	\rho_{f} (1 + c_f(p - p_r)) \mathbf{g} \approx \rho_{f} \mathbf{g}.
	 \end{align*}
	Therefore, dividing by $\rho_{f}$, we have,
	 \begin{align}\label{p}
	  \frac{\partial}{\partial t} (c_0p+\alpha\nabla \cdot\mathbf{w} )+ \nabla \cdot \mathbf{u}= \tilde{g_s},\quad
	  \mathbf{u}=-\frac{\mathbf{K}}{\eta}( \nabla p-\rho \mathbf{g}),
	 \end{align}
	where $\tilde{g_s}=g_s/\rho_{f}$.
	
	\subsubsection{Equation in $\mathcal{C}$}
The fluid in the fracture is also assumed to be slightly compressible. Following a similar model as in \cite{ALMANI2024117253,girault2015}, the conservation of mass in the fracture can be written as
	\begin{align*}
\frac{\partial\left(w_f\rho\right)}{\partial t}+\overline\nabla\cdot(\rho\mathbf{u}_f)=g_f+\rho z,\quad 
w_f=-\left[\mathbf{w}\right]_{\mathcal{C}}\cdot\mathbf{n}^+,\mathrm{~}\mathbf{u}_f=-\frac{\mathbf{K}_f}{\eta}(\overline\nabla p_c-\rho\mathbf{g}),~z=\left[\mathbf{u}\right]_\mathcal{C}\cdot\mathbf{n}^+,
	\end{align*}
	where $w_f$ represents the width of the fracture, $\overline\nabla$ is the tangential derivative along the fracture, $p_c$ is the pressure in the fracture, $\mathbf{u}_f$ is the flux unknowns in the fracture, $\mathbf{K}_f=k_f\mathbf{I}$ is the permeability tensor in the fracture, we set $\alpha_2=\eta k_f^{-1}$, $g_f$ is a known injection term into the fracture. $z$ is the leakage term connecting the flow in matrix to the flow in the fracture. We assume that $w_f$ is bounded in $\mathcal{C}$ and vanishes on $\partial\mathcal{C}$.
	\begin{remark}
		When $w_f=0$ on the boundary of $\mathcal{C}$, the boundary term vanishes. In other words, our paper discusses the scenario under the assumption shown in Figure \ref{a}, excluding the situation in Figure \ref{b}.
		\begin{figure}[h]
			\centering
			\subfigure[Type 1]{
				\begin{minipage}[t]{0.5\textwidth}
					\centering
					\includegraphics[scale=0.65]{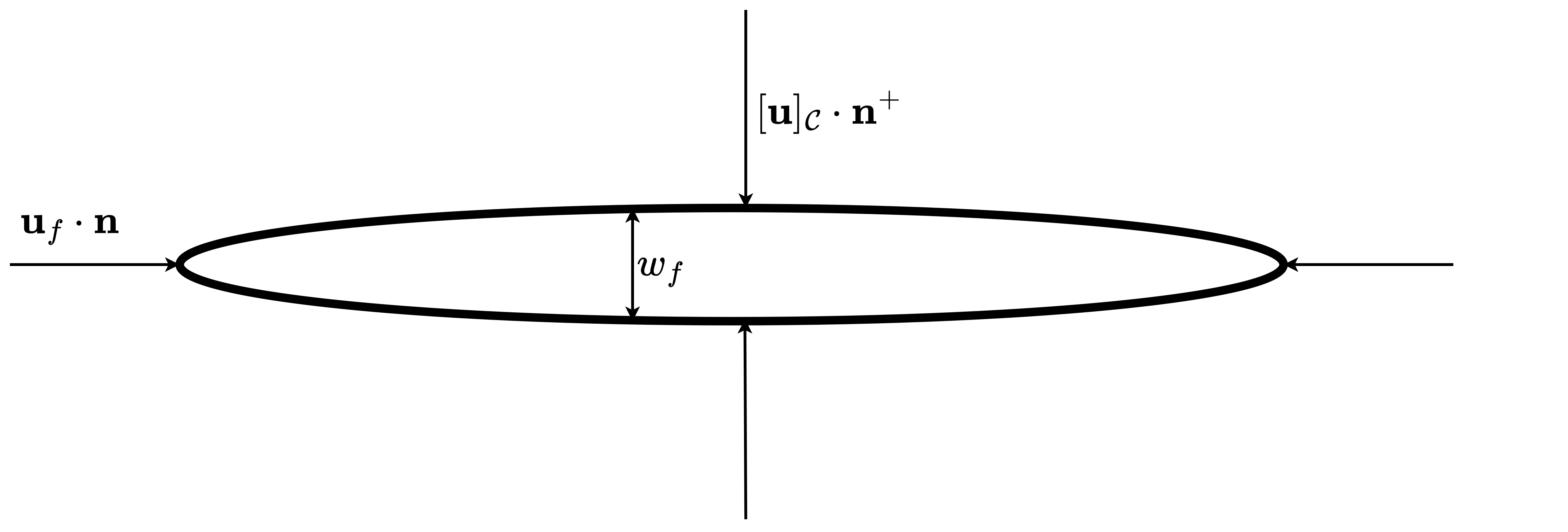}
				\end{minipage}
				\label{a}
			}
			\hfill
			\subfigure[Type 2]{
				\begin{minipage}[t]{0.45\textwidth}
					\centering
					\includegraphics[scale=0.65]{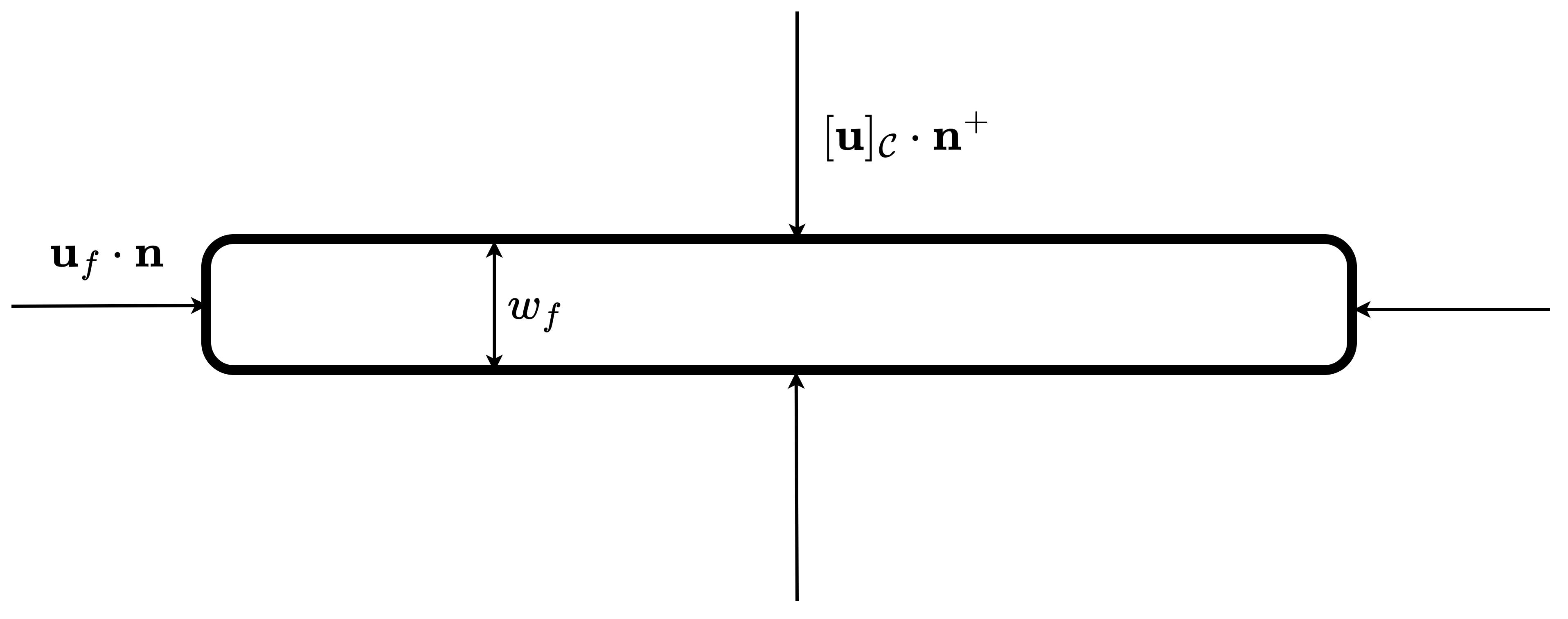}
				\end{minipage}
				\label{b}
			}
			\centering
			\caption{The fracture cross-sectional diagram.}
		\end{figure}
	\end{remark}
	The total fluid compressibility by
$
		\frac{\mathrm{d}\rho}{\mathrm{d} p}=\rho c_{f},~ \forall\mathbf{x}\in\Omega.
$
Then approximating $\rho$ by $\rho_{f}$ in the time derivative, linearizing the diffusion term as in the previous section, dividing by $\rho_{f}$, and setting
	$\tilde{g_f}=g_f/\rho_{f}$, yield the equation in $\mathcal{C}$:
	\begin{align}\label{wf}
		\frac{\partial\left(w_f+c_fp_c\right)}{\partial t}+\overline\nabla \cdot\mathbf{u}_f=\tilde{g_f}+\left[\mathbf{u}\right]_\mathcal{C}\cdot\mathbf{n}^+,\quad
	\mathbf{u}_f=-\frac{\mathbf{K}_f}{\eta}(\overline\nabla p_c-\rho \mathbf{g}),\quad
	w_f+[\mathbf{w}]_{\mathcal{C}}\cdot\mathbf{n}^+=0.
	\end{align}
		For any function $f$ defined in $\Omega\setminus \mathcal{C}$ that has a trace, let $f^*$ denote the trace of $f$ on $\mathcal{C}$. Then we define the jump of $f$ on $\mathcal{C}$ in the direction of $\mathbf{n}^+$ by
$
		[f]_{\mathcal{C}}=f^+-f^-.
$
	
	In summary, within the domain $\Omega\setminus\mathcal{C}$, the governing equations are given by \eqref{w} and \eqref{p}, whereas in the fracture domain $\mathcal{C}$, it is defined by equation \eqref{wf}. The corresponding unknowns are the displacement $\mathbf{w}$, pressure $p$ in the matrix, and pressure $p_c$ in the fracture. This system will be supplemented in the following section by the appropriate interface, boundary, and initial conditions.
	
	\subsubsection{Interface, boundary, and initial conditions}
	
		Let $\{\uptau_j\}$, $1\le j\le d-1$, be a set of orthonormal tangent vectors on $\mathcal{C}$. The balance of the normal traction vectors yields the interface conditions on each side (or face) of $\mathcal{C}$:
	\begin{align}\label{sigma}
		\sigma^{\mathrm{por}}\mathbf{n}=-p_c\mathbf{n}.
	\end{align}
	Then the continuity of $p_c$ through $\mathcal{C}$ yields
	\begin{align*}
		[\sigma^{\mathrm{por}}]_{\mathcal{C}}\mathbf{n}=\mathbf{0}.
	\end{align*}
	Formula \eqref{sigma} also implies 
	\begin{align}\label{sig}
		\sigma^{\mathrm{por}}\mathbf{n}\cdot\mathbf{n}=-p_c,~~\sigma^{\mathrm{por}}\mathbf{n}\cdot\uptau=\mathbf{0}.
	\end{align}
	We assume that the boundary conditions and the initial condition at time $t=0$:
	\begin{align}
		&\sigma^{\mathrm{por}}\mathbf{n}=-p_c\mathbf{n},\mathrm{~}\quad\forall\mathbf{x}\in\mathcal{C},~~\mathbf{w}=\mathbf{0},~~\mathbf{u}\cdot\mathbf{n}=0,\quad\forall\mathbf{x}\in\partial\Omega,\\
		&(\alpha\nabla\cdot\mathbf{w}+c_0p)(0)=\alpha\nabla\cdot\mathbf{w}^0+c_0p^0.
	\end{align}
	To overcome the locking phenomenon, a widely used approach is the mixed formulation, in which an additional variable $\xi=\alpha p-\lambda \nabla\cdot\mathbf{w}$ is introduced to represent the total pressure \cite{Boon2021,TP4-OyarzuaRhebergenSolano-2021-ESAIMMMNA}. In the subsequent analysis, the gravity vector is assumed to be zero, i.e., $\mathbf{g}=\mathbf{0}$. Therefore the complete problem statement, called Problem (P), is as follows.
	\begin{mdframed}
	\begin{align}\label{wx}
	- \nabla \cdot (2G\epsilon(\mathbf{w}))+\nabla\xi&= \mathbf{f},\quad &&\forall \mathbf{x} \in\Omega\setminus\mathcal{C},\\
	\nabla\cdot\mathbf{w}+\frac{1}{\lambda}\xi-\frac{\alpha}{\lambda}p&=0,~~\quad&&\forall\mathbf{x}\in\Omega\setminus\mathcal{C},\\
	(c_0+\frac{\alpha^2}{\lambda})\frac{\partial p}{\partial t}-\frac{\alpha}{\lambda}\frac{\partial \xi}{\partial t}-\nabla\cdot(\frac{\mathbf{K}}{\eta}\nabla p)&=\tilde{g_s},~~\quad&&\forall\mathbf{x}\in\Omega\setminus\mathcal{C},\\\label{wf1}
	\frac{\partial\left(w_f+c_fp_c\right)}{\partial t}-\overline\nabla\cdot(\frac{\mathbf{K}_f}{\eta}\overline\nabla p_c)+[\frac{\mathbf{K}}{\eta}\nabla p]_{\mathcal{C}}\cdot\mathbf{n}^+&=\tilde{g_f},~~\quad&&\forall\mathbf{x}\in\mathcal{C},\\
	w_f+[\mathbf{w}]_{\mathcal{C}}\cdot\mathbf{n}^+&=0,~~\quad&&\forall\mathbf{x}\in\mathcal{C},\\
	\sigma^{\mathrm{por}}\mathbf{n}=-p_c\mathbf{n},\quad	p_c&=p|_{\mathcal{C}},~~\quad&&\forall\mathbf{x}\in\mathcal{C}.
\end{align}
	\end{mdframed}
	The well-posedness of the continuous formulation, including the existence and uniqueness of the solution, has been established in \cite{BONALDI2024,girault2015,girault2019mixed}.
	
		\subsection{Variational formulation}
	In this section, we first introduce the classical definitions of Sobolev spaces and their associated norms, followed by the definitions of variational and bilinear forms for fluid-crack and porous-elastic coupling models. 
	
	 We shall use the standard notations
	and definitions for Sobolev spaces \cite{brenner2008mathematical} throughout the paper.
	Let $\Gamma$ be a part of $\partial\Omega$ with positive measure,
	$
	\mathcal{V}:=\{v\in L^2(\Omega); v|_{\Omega\setminus\mathcal{C}}\in H^1(\Omega\setminus \mathcal{C})\},
	$
	normed by the graph norm:
	$
	\|v\|_\mathcal{V}^2=\|v|_{\Omega\setminus\mathcal{C}}\|_{H^1(\Omega\setminus \mathcal{C})}^2.
	$
	The space for the displacement is
	\begin{align}
		\mathcal{W}:=\{\mathbf{v}\in \mathcal{V}^d; [\mathbf{v}]_{\Gamma\setminus\mathcal{C}}=0,\mathbf{v}_{|\partial\Omega}=0\},
	\end{align}
	with the norm of $\mathcal{W}^d:$
	$
	\|\mathbf{v}\|_{\mathcal{W}}=\left(\sum_{i=1}^d\|v_i\|_{\mathcal{V}}^2\right)^{1/2}.
	$
	The space for the total pressure $\xi$ is $\varPsi:=L^2(\Omega\setminus\mathcal{C})$.
	Additionally, the following inf-sup condition holds: There exists a constant $\beta_0> 0$ such that
	\begin{equation}
		\sup_{\mathbf{v} \in W} \frac{(\operatorname{div} \mathbf{v}, q)}{\|\mathbf{v}\|_{H^1(\Omega\setminus\mathcal{C})}} 
		\ge \beta_0 \|q\|_{L^2(\Omega\setminus\mathcal{C})}, 
		\quad \forall q \in L^2(\Omega\setminus\mathcal{C}).
	\end{equation}
		For any $ \nu \in \mathcal{W}$, the following Poincar\'e--Friedrichs's inequality and Korn's inequality hold: There exist two constants $C_{PF} >0$ and $C_{Korn} > 0$  such that
	\begin{align}\label{pc}
		\|\nu\|_{L^2(\Omega)}\leq C_{PF} \|\nu\|_{H^1(\Omega)} \leq C_{PF} C_{Korn} \|\epsilon(\nu)\|_{L^2(\Omega)}.
	\end{align}
	
	We denote the space $H^1(\mathcal{C})$:
	$
	H^1(\mathcal{C})=\{z\in L^2(\mathcal{C});\overline{\nabla} z\in L^2(\mathcal{C})^{d-1}\},
	$
	equipped with the norm:
	$
	\|z\|_{H^1(\mathcal{C})}=\left(\|z\|_{H^{1/2}(\mathcal{C})}^2+\|\overline{\nabla}z\|_{L^2(\mathcal{C})}^2\right)^{1/2}.
	$
	We can specify the pressure space $Q$: 
	\begin{align}
		Q:=\{q\in H^1(\Omega);q_c\in H^1(\mathcal{C})\} ,\quad\text{where }q_c=q_{|\mathcal{C}},
	\end{align}
	equipped with the graph norm:
	$
	\|q\|_Q=\left(\|q\|_{H^1(\Omega)}^2+\|q_c\|_{H^1(\mathcal{C})}^2\right)^{1/2}.
	$
	
	The space for the pressure in the fracture is
	\begin{align}
		Q_c:=H^{1}(\mathcal{C}).
	\end{align}
	
	For given $\mathbf{f}\in L^2(\Omega\setminus\mathcal{C}\times(0,T))^d$, $\tilde{g_s}\in L^2(\Omega\times(0,T))$ and $\tilde{g_f}\in L^2(0,T;H^{-1/2}(\mathcal{C}))$, 
find $\mathbf{w}\in L^{\infty}(0, T;\mathcal{W})$, $\xi\in L^2(0,T; \varPsi)$, $p\in L^{\infty}(0, T; L^2(\Omega))\cap L^{\infty}(0, T; Q)$ and $p_c \in L^2(0, T; H^{1}(\mathcal{C}))$, such that for any $\mathbf{v}\in L^{\infty}(0, T;\mathcal{W})$, $\theta\in L^2(0,T; \varPsi)$, $q\in L^{\infty}(0, T; L^2(\Omega))\cap L^{\infty}(0, T; Q)$ and $q_c \in L^2(0, T; H^{1}(\mathcal{C}))$,
	\begin{mdframed}
	\begin{align}\label{vwx}
	&	2G(\epsilon(\mathbf{w}),\epsilon(\mathbf{v}))_{\Omega\setminus\mathcal{C}}- (\xi,\nabla\cdot\mathbf{v})_{\Omega\setminus\mathcal{C}}+(p_c,[\mathbf{v}]\cdot\mathbf{n}^+)_{\mathcal{C}}= (\mathbf{f},\mathbf{v})_{\Omega\setminus\mathcal{C}},\\\label{vpsix}
	&	(\frac{1}{\lambda}\xi,\theta)_{\Omega\setminus\mathcal{C}}+(\nabla\cdot\mathbf{w},\theta)_{\Omega\setminus\mathcal{C}}-(\frac{\alpha}{\lambda}p,\theta)_{\Omega\setminus\mathcal{C}}=0,\\\label{vpx}
	&	(c_0+\frac{\alpha^2}{\lambda})(\partial_tp,q)_{\Omega\setminus\mathcal{C}}-(\frac{\alpha}{\lambda}\partial_t\xi,q)_{\Omega\setminus\mathcal{C}}+(\frac{\mathbf{K}}{\eta}\nabla p,\nabla q)_{\Omega\setminus\mathcal{C}}	-([\frac{\mathbf{K}}{\eta}\nabla p]\cdot\mathbf{n}^+,q)_{\mathcal{C}}=(\tilde{g_s},q)_{\Omega\setminus\mathcal{C}},\\\label{vwfx}
	&	(\partial_t w_f,q_c)_{\mathcal{C}}+(c_f\partial_t p_c,q_c)_{\mathcal{C}}+(\frac{\mathbf{K}_f}{\eta}\overline\nabla p_c,\overline\nabla q_c)_{\mathcal{C}}+([\frac{\mathbf{K}}{\eta}\nabla p]\cdot\mathbf{n}^+,q_c)_{\mathcal{C}}=(\tilde{g_f},q_c)_{\mathcal{C}}.
	\end{align}
\end{mdframed}
	\subsection{Stability analysis}
		In this section, we first introduce the total free energy of the model and then derive the energy dissipation law for the model in the continuous case.
	According to the second law of thermodynamics, the total free energy in a closed system will dissipate over time \cite{2018Thermodynamically,TANG2025321}.  We denote the $\|\cdot\|_1$ and $\|\cdot\|_0$  the standard $H^1$ norm and $L^2$ norm in $\Omega\setminus\mathcal{C}$, respectively.
	Similarly, We use the $\|\cdot\|_{1,\mathcal{C}}$ and $\|\cdot\|_{0,\mathcal{C}}$  the $H^1$ norm and $L^2$ norm in $\mathcal{C}$, respectively.
We define the total free energy $E_{\mathrm{tot}}$ within the system as
\begin{align}
	E_{\mathrm{tot}}=E_s+E_v+E_f+E_f^\mathrm{frac},
\end{align}
where
\begin{align*}
	E_s(t)&=\frac{1}{2}\int_{\Omega\setminus\mathcal{C}}2G\epsilon(\mathbf{w}): \epsilon(\mathbf{w})\mathrm{d}\mathbf{x},~~~
	E_v(t)=\frac{1}{2}\int_{\Omega\setminus\mathcal{C}}\frac{1}{\lambda}(\alpha p-\xi)^2 \mathrm{d}\mathbf{x},\\\label{Ef}
	E_f(t)&=\frac{1}{2}\int_{\Omega\setminus\mathcal{C}}c_0p^2\mathrm{d}\mathbf{x},~~~
	E_f^{\mathrm{frac}}(t)=\frac{1}{2}\int_{\mathcal{C}}c_fp_c^2\mathrm{d}\mathbf{s}.
\end{align*}

\begin{theorem}
	For the closed system with the boundary conditions $\mathbf{w} = \mathbf{0}$ and $\mathbf{u}\cdot\mathbf{n}=0$ on the boundary $\partial\Omega$, where $ \mathbf{n}$ denotes the normal unit outward vector to $\partial\Omega$, the gravity $\mathbf{g} = \mathbf{0}$. The coupled model \eqref{w}--\eqref{wf} satisfies the following discrete energy dissipation law as
	\begin{align}
		\frac{\partial E_{\mathrm{tot}}}{\partial t}+(\frac{\mathbf{K}}{\eta}\nabla p,\nabla p)_{\Omega\setminus\mathcal{C}}+(\frac{\mathbf{K}_f}{\eta}\overline\nabla p_c,\overline\nabla p_c)_{\mathcal{C}}= (\mathbf{f},\frac{\partial \mathbf{w}}{\partial t})_{\Omega\setminus\mathcal{C}}+(\tilde{g_s},p)_{\Omega\setminus\mathcal{C}}+(\tilde{g_f},p_c)_{\mathcal{C}}.
	\end{align}
	Especially, if \,$\mathbf{f}=0$, $\tilde{g_s}=0$ and $\tilde{g_f}=0$, we have,
	\begin{align}
		\frac{\partial E_{\mathrm{tot}}}{\partial t}=-\|(\frac{\mathbf{K}}{\eta})^{1/2}\nabla p\|^2_{0}-\|(\frac{\mathbf{K}_f}{\eta})^{1/2}\overline\nabla p_c\|^2_{0,\mathcal{C}}\le 0.
	\end{align}
\end{theorem}
\begin{proof}
	We now derive the equation for the variation of total free energy with respect to time. For the variational form of the displacement \eqref{vwx}, by letting the test function $\mathbf{v}=\partial_t\mathbf{w}$, we can obtain the following result,
	\begin{align}
		(\mathbf{f},\partial_t\mathbf{w})_{\Omega\setminus\mathcal{C}}&=		\frac{\partial E_s}{\partial t}\underbrace{- (\xi,\nabla\cdot\partial_t \mathbf{w})_{\Omega\setminus\mathcal{C}}}_{\textcircled1}\underbrace{+(p_c,[\partial_t \mathbf{w}]\cdot\mathbf{n}^+)_{\mathcal{C}}}_{\textcircled2}.
	\end{align}

	For the Equation \eqref{vpsix}, differentiating with respect to $t$ and and setting the test function to $\theta=\xi$,  we have,
	\begin{align}
	0=&\underbrace{(\nabla\cdot\partial_t\mathbf{w},\xi)_{\Omega\setminus\mathcal{C}}}_{\textcircled1}+(\frac{1}{\lambda}\partial_t\xi,\xi)_{\Omega\setminus\mathcal{C}}-(\frac{\alpha}{\lambda}\partial_tp,\xi)_{\Omega\setminus\mathcal{C}}.
	\end{align}
	
	For flow equation in the matrix \eqref{vpx}, letting $q=p$, we can derive,
	\begin{align}
		&(\tilde{g_s},p)_{\Omega\setminus\mathcal{C}}=\frac{\partial E_f}{\partial t}+\frac{\alpha^2}{\lambda}(\partial_tp,p)_{\Omega\setminus\mathcal{C}}-(\frac{\alpha}{\lambda}\partial_t\xi,p)_{\Omega\setminus\mathcal{C}}+(\frac{\mathbf{K}}{\eta}\nabla p,\nabla p)_{\Omega\setminus\mathcal{C}}	\underbrace{-([\frac{\mathbf{K}}{\eta}\nabla p]\cdot\mathbf{n}^+,p)_{\mathcal{C}}}_{\textcircled3}.
	\end{align}
	We note the identity,
		\begin{align*}
		\frac{1}{\lambda}(\alpha\partial_t p-\partial_t\xi,\alpha p-\xi)_{\Omega\setminus\mathcal{C}}=(\frac{1}{\lambda}\partial_t\xi,\xi)_{\Omega\setminus\mathcal{C}}-(\frac{\alpha}{\lambda}\partial_tp,\xi)_{\Omega\setminus\mathcal{C}}+\frac{\alpha^2}{\lambda}(\partial_tp,p)_{\Omega\setminus\mathcal{C}}-(\frac{\alpha}{\lambda}\partial_t\xi,p)_{\Omega\setminus\mathcal{C}}.
	\end{align*}
	Next, for the mass conservation equation of the fracture \eqref{vwfx}, by letting $q_c=p_c$, we note that $w_f=-\left[\mathbf{w}\right]_{\mathcal{C}}\cdot\mathbf{n}^+$, and can get the following conclusion,
	\begin{align}
		(\tilde{g_f},p_c)_{\mathcal{C}}&=\frac{\partial E_f^{\mathrm{frac}}}{\partial t}	\underbrace{-(\partial_t ([\mathbf{w}]\cdot\mathbf{n}^+) ,p_c)_{\mathcal{C}}}_{\textcircled2}+(\frac{\mathbf{K}_f}{\eta}\overline\nabla p_c,\overline\nabla p_c)_{\mathcal{C}}\underbrace{+([\frac{\mathbf{K}}{\eta}\nabla p]\cdot\mathbf{n}^+,p_c)_{\mathcal{C}}}_{\textcircled3}.
	\end{align}
	
	Following the model derivations, we can deduce that the model obeys the energy dissipation law within a closed system,
	\begin{align}\label{CE}
		\frac{\partial E_{\mathrm{tot}}}{\partial t}=&\frac{\partial E_s}{\partial t}+\frac{\partial E_v}{\partial t}+\frac{\partial E_f}{\partial t}+\frac{\partial E_f^{\mathrm{frac}}}{\partial t}\notag\\
		=&(\mathbf{f},\frac{\partial \mathbf{w}}{\partial t})_{\Omega\setminus\mathcal{C}}+(\tilde{g_s},p)_{\Omega\setminus\mathcal{C}}+(\tilde{g_f},p_c)_{\mathcal{C}}-(\frac{\mathbf{K}}{\eta}\nabla p,\nabla p)_{\Omega\setminus\mathcal{C}}-(\frac{\mathbf{K}_f}{\eta}\overline\nabla p_c,\overline\nabla p_c)_{\mathcal{C}},
	\end{align}
	which indicates that the total energy is always dissipated with time. The proof is completed.
\end{proof}

\section{Time discretization and stability analysis}\label{s3}
The energy dissipation law is a fundamental principle inherent in nature, which is also reflected in our proposed model. Therefore, an effective numerical method should maintain this law. In this section, we develop an energy-stable time-discrete scheme. 
We divide the time interval $[0, T]$ into $N$ time steps as $0 = t_0 < t_1 < \cdots < t_N = T$ and denote the time step size by $\tau =t^{n+1}-t^{n}$. For any variable $v$, the superscript $n$ in $v^n$ indicates the approximation of $v$ at the time $t^n$.
\subsection{Time-decoupled discrete scheme}
We propose the following semi-implicit discrete scheme in time:\\
When $n=0$, we compute the results of the first step using the full coupling format.\\
$\mathbf{Initial~step:}$
\begin{align}
	(c_0p+\alpha\nabla\cdot \mathbf{w})(0)=c_0p^0+\alpha\nabla\cdot \mathbf{w}^0,\quad \xi^0=\alpha p^0-\lambda\nabla\cdot\mathbf{w}^0.
\end{align}
Given $p^0$ and $\mathbf{w}^0$, we propose to first apply the implicit Euler time-stepping scheme to compute $\mathbf{w}^1\in \mathcal{W}$, $\xi^1 \in \varPsi$, $p^1\in Q$, and $p_c^1\in Q_c$, such that for any $(\mathbf{v},\theta,q,q_{c})\in \mathcal{W}\times\varPsi\times Q\times Q_{c}$, 
	\begin{align}\label{Iw}
&	2G(\epsilon(\mathbf{w}^1),\epsilon(\mathbf{v}))_{\Omega\setminus\mathcal{C}}- (\xi^1,\nabla\cdot\mathbf{v})_{\Omega\setminus\mathcal{C}}+(p_c^1,[\mathbf{v}]\cdot\mathbf{n}^+)_{\mathcal{C}}= (\mathbf{f}^1,\mathbf{v})_{\Omega\setminus\mathcal{C}},\\\label{Ipsi}
&	\frac{1}{\lambda}(\frac{\xi^1-\xi^0}{\tau},\theta)_{\Omega\setminus\mathcal{C}}+(\frac{\nabla\cdot(\mathbf{w}^1-\mathbf{w}^0)}{\tau},\theta)_{\Omega\setminus\mathcal{C}}-\frac{\alpha}{\lambda}(\frac{p^1-p^0}{\tau},\theta)_{\Omega\setminus\mathcal{C}}=0,\\\label{Ip}
&	(c_0+\frac{\alpha^2}{\lambda})(\frac{p^1-p^0}{\tau},q)_{\Omega\setminus\mathcal{C}}-(\frac{\alpha}{\lambda}\frac{\xi^1-\xi^0}{\tau},q)_{\Omega\setminus\mathcal{C}}+(\frac{\mathbf{K}}{\eta}\nabla p^1,\nabla q)_{\Omega\setminus\mathcal{C}}	-([\frac{\mathbf{K}}{\eta}\nabla p^1]\cdot\mathbf{n}^+,q)_{\mathcal{C}}=(\widetilde{g_s^1},q)_{\Omega\setminus\mathcal{C}},\\\label{Ipc}
&	(\frac{w_f^1-w_f^0}{\tau },q_c)_{\mathcal{C}}+(c_f\frac{p_c^1-p_c^0}{\tau },q_c)_{\mathcal{C}}+(\frac{\mathbf{K}_f}{\eta}\overline\nabla p_c^1,\overline\nabla q_c)_{\mathcal{C}}+([\frac{\mathbf{K}}{\eta}\nabla p^1]\cdot\mathbf{n}^+,q_c)_{\mathcal{C}}=(\widetilde{g_f^1},q_c)_{\mathcal{C}}.
\end{align}
When $n\ge1$, we use the following decoupling algorithm to compute the unknowns.\\
$\mathbf{Step\, 1:}$ Given $\mathbf{w}^{n}\in \mathcal{W}$, $\xi^{n-1},~\xi^n \in \varPsi$, find $p^{n+1}\in Q$ and $p_c^{n+1}\in Q_c$ , such that for any $(q,q_{c})\in Q\times Q_{c}$, 
\begin{align}\label{sp}
&	(c_0+\frac{\alpha^2}{\lambda})(\frac{p^{n+1}-p^{n}}{\tau},q)_{\Omega\setminus\mathcal{C}}+(\frac{\mathbf{K}}{\eta}\nabla p^{n+1},\nabla q)_{\Omega\setminus\mathcal{C}}	-([\frac{\mathbf{K}}{\eta}\nabla p^{n+1}]\cdot\mathbf{n}^+,q)_{\mathcal{C}}\notag\\&\quad=(\frac{\alpha}{\lambda}\frac{\xi^{n}-\xi^{n-1}}{\tau},q)_{\Omega\setminus\mathcal{C}}+(\widetilde{g_s^{n+1}},q)_{\Omega\setminus\mathcal{C}},\\\label{spc}
&(c_f\frac{p_c^{n+1}-p_c^{n}}{\tau },q_c)_{\mathcal{C}}+	(
\frac{\mathbf{K}_f}{\eta}\overline\nabla p_c^{n+1},\overline\nabla q_c)_{\mathcal{C}}+([\frac{\mathbf{K}}{\eta}\nabla p^{n+1}]\cdot\mathbf{n}^+,q_c)_{\mathcal{C}}
\notag\\
&\quad=-(\frac{w_f^{n}-w_f^{n-1}}{\tau },q_c)_{\mathcal{C}}+(\widetilde{g_f^{n+1}},q_c)_{\mathcal{C}}.
\end{align}
\\
$\mathbf{Step \,2:}$ Given $p^{n+1}\in Q$ and $p_c^{n+1}\in Q_c$, find $\mathbf{w}^{n+1}\in \mathcal{W}$ and $\xi^{n+1} \in \varPsi$, such that
for any $(\mathbf{v},\theta)\in \mathcal{W}\times\varPsi$,  
\begin{align}\label{sw}
&2G(\epsilon(\mathbf{w}^{n+1}),\epsilon(\mathbf{v}))_{\Omega\setminus\mathcal{C}}- (\xi^{n+1},\nabla\cdot\mathbf{v})_{\Omega\setminus\mathcal{C}}+(p_c^{n+1},[\mathbf{v}]\cdot\mathbf{n}^+)_{\mathcal{C}}\notag\\
&\quad+	L_w([\mathbf{w}^{n+1}-\mathbf{w}^{n}]\cdot\mathbf{n}^+,[\mathbf{v}]\cdot\mathbf{n}^+)_{\mathcal{C}}-L_w([\mathbf{w}^{n}-\mathbf{w}^{n-1}]\cdot\mathbf{n}^+,[\mathbf{v}]\cdot\mathbf{n}^+)_{\mathcal{C}}= (\mathbf{f}^{n+1},\mathbf{v})_{\Omega\setminus\mathcal{C}},\\\label{spsi}
&(\frac{1}{\lambda}\xi^{n+1},\theta)_{\Omega\setminus\mathcal{C}}+(\nabla\cdot\mathbf{w}^{n+1},\theta)_{\Omega\setminus\mathcal{C}}-\frac{\alpha}{\lambda}(p^{n+1},\theta)_{\Omega\setminus\mathcal{C}}=0,
\end{align}
where the stabilization parameter $L_w$ will be defined later.\\
$\mathbf{Step \,3:}$ Given $\mathbf{w}^{n+1}\in \mathcal{W}$, find the width of the fracture $w_f^{n+1}$, such that
\begin{align}
	w_f^{n+1}=-[\mathbf{w}^{n+1}]_{\mathcal{C}}\cdot\mathbf{n}^+.
\end{align}

\begin{remark}
	To ensure a stable initial state and facilitate the error analysis, a fully coupled scheme is applied at the initial stage to compute $(\mathbf{w}^1,\xi^1,p^1,p_c^1)$. 
	From the second time step onward, the problem is decomposed into two independent subsystems. 
	These subsystems are solved sequentially to obtain the solution at the $(n+1)$-th time level.
\end{remark}

\subsection{Stability analysis in time}

In this subsection, we establish an unconditional discrete energy estimate. Throughout this subsection, we neglect source terms and gravity contributions for notational simplicity. 
\begin{theorem}\label{Ithe}
Let $(\mathbf{w}^1,\xi^1,p^1,p_c^1)$ be the solution of the initial step \eqref{Iw}--\eqref{Ipc} in the semi-discrete scheme, then we have the following stability result:
\begin{align}
&G\|\epsilon(\mathbf{w}^1)\|_{0}^2+\frac{1}{2\lambda}\|\alpha p^1-\xi^1\|_{0}^2+\frac{c_0}{2}\|p^1\|_{0}^2+\frac{c_f}{2}\|p_c^1\|_{0,\mathcal C}^2+\frac{\tau}{2}\|(\frac{\mathbf{K}}{\eta})^{1/2}\nabla p^1\|^2_{0}+\frac{\tau}{2}\|(\frac{\mathbf{K}_f}{\eta})^{1/2}\overline\nabla p^1_c\|^2_{0,\mathcal{C}}\notag\\
&\quad+\frac{G}{2}\|\epsilon(\mathbf{w}^1-\mathbf{w}^0)\|_{0}^2+\frac{1}{2\lambda}\|\alpha (p^1-p^0)-(\xi^1-\xi^0)\|_{0}^2+\frac{c_0}{2}\|p^1-p^0\|_{0}^2+\frac{c_f}{2}\|p_c^1-p_c^0\|_{0,\mathcal C}^2\notag\\
&\le G\|\epsilon(\mathbf{w}^0)\|_{0}^2+\frac{1}{2\lambda}\|\alpha p^0-\xi^0\|_{0}^2+\frac{c_0}{2}\|p^0\|_{0}^2+\frac{c_f}{2}\|p_c^0\|_{0,\mathcal C}^2\notag\\
&\quad+\frac{C_{PF}^2C_{Korn}^2}{2G}\|\mathbf{f}^1\|^2_0+\frac{\tau}{2}\|(\frac{\mathbf{K}}{\eta})^{1/2}\widetilde{g_s^{1}}\|_0^2+\frac{\tau}{2}\|(\frac{\mathbf{K}_f}{\eta})^{1/2}\widetilde{g_f^{1}}\|_{0,\mathcal{C}}^2.
\end{align}
\end{theorem}

\begin{proof}
Taking the test functions for Equations \eqref{Iw}--\eqref{Ipc} be $\mathbf{v}=\mathbf{w}^1-\mathbf{w}^0$, $\theta=\xi^1$, $q=\tau p^1$, and $q_c=\tau p_c^1$, respectively. Using Young's inequality, the Cauchy--Schwarz inequality, Poincar\'e--Friedrichs's inequality and Korn's inequality \eqref{pc}, we can derive the result of the Theorem \ref{Ithe}.
\end{proof}

The standard discrete Gronwall inequality \cite{riviere2008} is also employed in the proof of this theorem.
\begin{lemma}[Discrete Gronwall  inequality] \label{Gron} Let $\Delta t,~B,~C>0$ and let $\{a_N\},~\{b_N\},~\{c_N\}$ be sequences of nonnegative numbers satisfying
	\begin{align*}
		a_N+\Delta t\sum_{i=0}^{N}b_i\leq B+C\Delta t\sum_{i=0}^{N}a_i+\Delta t\sum_{i=0}^{N}c_i, \quad \forall N\geq0.
	\end{align*}
	Then, if\, $C\Delta t<1$,
	\begin{align*}
		a_N+\Delta t\sum_{i=0}^Nb_i\leq e^{C(N+1)\Delta t}\left(B+\Delta t\sum_{i=0}^Nc_i\right),\quad \forall N\geq0.
	\end{align*}
\end{lemma}
\begin{theorem}\label{thm:psi-energy}
	For the closed system with the boundary conditions $\mathbf{w}^{n+1} = \mathbf{0}$ and $\mathbf{u}^{n+1}\cdot\mathbf{n}=0$ on the boundary $\partial\Omega$, where $ \mathbf{n}$ denotes the normal unit outward vector to $\partial\Omega$, the gravity $\mathbf{g} = \mathbf{0}$, $\widetilde{g_f^{n+1}} =0$, $\widetilde{g_s^{n+1}}=0$, and $\mathbf{f}^{n+1}=\mathbf{0}$.  Assume that the stabilization parameter is $L_w= \frac{4}{c_f}$.  Then, for $n\ge1$, the decoupling scheme \eqref{sp}--\eqref{spsi} satisfies the following unconditional stability result:
	\begin{align}
		&(\frac{c_0}{4}+\frac{\alpha^2}{2\lambda})\|p^{N+1}\|_{0}^2
		+G\|\epsilon(\mathbf{w}^{N+1})\|_{0}^2
		+\frac{1}{4\lambda}\|\xi^{N+1}\|_{0}^2	+\frac{c_f}{4}\|p_c^{N+1}\|_{0,\mathcal C}^2+\frac{L_w}{4}
		\|[\delta\mathbf{w}^{N+1}]\cdot \mathbf{n}^+\|_{0,\mathcal C}^2\notag\\
		&\qquad+\sum_{n=1}^{N}\left((\frac{c_0}{4}+\frac{\alpha^2}{2\lambda})\|\delta p^{n+1}\|_{0}^2
		+\frac{c_f}{4}\|\delta p_c^{n+1}\|_{0,\mathcal C}^2
		+\frac{G}{2}\|\epsilon(\delta\mathbf{w}^{n+1})\|_{0}^2+\frac{1}{4\lambda}\sum_{n=1}^{N}\|\delta\xi^{n+1}\|_0^2\right)\notag\\
		&\qquad+\sum_{n=1}^{N}\left(\tau\Bigl\|\frac{\mathbf K^{1/2}}{\eta^{1/2}}\nabla p^{n+1}\Bigr\|_{0}^2
		+\tau\Big\|\frac{\mathbf K_f^{1/2}}{\eta^{1/2}}\overline\nabla p_c^{n+1}\Big\|_{0,\mathcal C}^2
		+\frac{L_w}{2}\|\delta j^{n+1}-\delta j^n\|_{0,\mathcal C}^2\right)\notag\\
		&\le(\frac{3c_0}{4}+\frac{\alpha^2}{2\lambda})\|p^{1}\|_{0}^2
		+G\|\epsilon(\mathbf{w}^{1})\|_{0}^2
		+\frac{3}{4\lambda}\|\xi^{1}\|_{0}^2+\frac{1}{8\lambda}\|\xi^{0}\|_{0}^2	+\frac{3c_f}{4}\|p_c^{1}\|_{0,\mathcal C}^2\notag\\
		&\qquad+\frac{1}{16\lambda}\|\delta\xi^{1}\|_0^2+(\frac{G}{2}+\frac{3L_w}{4})
		\|[\delta\mathbf{w}^{1}]\cdot \mathbf{n}^+\|_{0,\mathcal C}^2.
	\end{align}
For any large values of $\lambda$, the result mentioned above remains valid.
\end{theorem}
\begin{proof}
	We first introduce the difference notation,
	\begin{align*}\label{eq:diff_notation}
		\delta p^{n+1}&:=p^{n+1}-p^n,\qquad
		\delta p_c^{n+1}:=p_c^{n+1}-p_c^n,\qquad
		\delta \mathbf{w}^{n+1}:=\mathbf{w}^{n+1}-\mathbf{w}^n,\notag\\
			j^n&:=[\mathbf{w}^n]_{\mathcal{C}}\cdot \mathbf{n}^+,
		\qquad
		\delta j^{n+1}:=j^{n+1}-j^n,
		\qquad
		\delta \xi^{n+1}:=\xi^{n+1}-\xi^n.
	\end{align*}
	Since $
	w_f^{n+1}=-[\mathbf{w}^{n+1}]_{\mathcal{C}}\cdot \mathbf{n}^+=-j^{n+1},
$
	we have $-(w_f^n-w_f^{n-1})=\delta j^n$.

	We test \eqref{sp} with $q= \tau p^{n+1}$, \eqref{spc} with $q_c=\tau p_c^{n+1}$, the displacement equation \eqref{sw} with $\mathbf{v}=\delta \mathbf{w}^{n+1}$, and take the time variables $t=t^n$ and $t=t^{n+1}$ from Equation \eqref{spsi} and perform a difference operation, with $\theta=\xi^{n+1}$. Using the identity
	\begin{equation}\label{eq:polarization_identity}
		(a-b,a)=\frac12\Bigl(\|a\|^2-\|b\|^2+\|a-b\|^2\Bigr).
	\end{equation}
Letting $\widetilde{g_f^{n+1}}=0$, $\widetilde{g_s^{n+1}}=0$, and $\mathbf{f}^{n+1}=\mathbf{0}$, we obtain, respectively,
	\begin{align}
	&(\frac{c_0}{2}+\frac{\alpha^2}{2\lambda})\Bigl(
	\|p^{n+1}\|_{0}^2-\|p^n\|_{0}^2+\|\delta p^{n+1}\|_{0}^2
	\Bigr)
	+\tau\Bigl\|\frac{\mathbf K^{1/2}}{\eta^{1/2}}\nabla p^{n+1}\Bigr\|_{0}^2\notag\\
	&\qquad -\tau\Bigl([\tfrac{\mathbf K}{\eta}\nabla p^{n+1}]\cdot\mathbf{n}^+,p^{n+1}\Bigr)_{\mathcal C}
	=
	\frac{\alpha}{\lambda}(\delta\xi^n,p^{n+1})_{\Omega\setminus\mathcal C},
	\label{eq:primary_p1}
	\\
	&\frac{c_f}{2}\Bigl(
	\|p_c^{n+1}\|_{0,\mathcal C}^2-\|p_c^n\|_{0,\mathcal C}^2+\|\delta p_c^{n+1}\|_{0,\mathcal C}^2
	\Bigr)
	+\tau\Big\|\frac{\mathbf K_f^{1/2}}{\eta^{1/2}}\overline\nabla p_c^{n+1}\Big\|_{0,\mathcal C}^2 \notag\\
	&\qquad
	+\tau\Bigl([\tfrac{\mathbf K}{\eta}\nabla p^{n+1}]\cdot\mathbf{n}^+,p_c^{n+1}\Bigr)_{\mathcal C}
	=
	(\delta j^n,p_c^{n+1})_{\mathcal C},
	\label{eq:primary_pc1}
	\\
	&G\Bigl(
	\|\epsilon(\mathbf{w}^{n+1})\|_{0}^2
	-\|\epsilon(\mathbf{w}^{n})\|_{0}^2
	+\|\epsilon(\delta\mathbf{w}^{n+1})\|_{0}^2
	\Bigr)
	-(\xi^{n+1},\nabla\cdot\delta\mathbf{w}^{n+1})_{\Omega\setminus\mathcal C}
	+(p_c^{n+1},\delta j^{n+1})_{\mathcal C}\notag\\
	&\qquad
	+\frac{L_w}{2}\Bigl(
	\|\delta j^{n+1}\|_{0,\mathcal C}^2-\|\delta j^n\|_{0,\mathcal C}^2+\|\delta j^{n+1}-\delta j^n\|_{0,\mathcal C}^2
	\Bigr)
	=0,
	\label{eq:primary_w1}
	\\
	&\frac{1}{\lambda}(\delta\xi^{n+1},\xi^{n+1})_{\Omega\setminus\mathcal C}
	+(\nabla\cdot\delta\mathbf{w}^{n+1},\xi^{n+1})_{\Omega\setminus\mathcal C}=	\frac{\alpha}{\lambda}(\delta p^{n+1},\xi^{n+1})_{\Omega\setminus\mathcal C}.
	\label{eq:primary_psi1}
\end{align}
	Summing \eqref{eq:primary_p1}--\eqref{eq:primary_psi1}, we first note that
	\begin{align*}
		-(\xi^{n+1},\nabla\cdot\delta\mathbf{w}^{n+1})_{\Omega\setminus\mathcal C}
		+(\nabla\cdot\delta\mathbf{w}^{n+1},\xi^{n+1})_{\Omega\setminus\mathcal C}=0.
	\end{align*}
	Moreover, due to the interface continuity condition $	p^{n+1}|_{\mathcal C}=p_c^{n+1}$,
	\begin{align*}
		-\Bigl([\tfrac{\mathbf K}{\eta}\nabla p^{n+1}]\cdot\mathbf{n}^+,p^{n+1}\Bigr)_{\mathcal C}
		+\Bigl([\tfrac{\mathbf K}{\eta}\nabla p^{n+1}]\cdot\mathbf{n}^+,p_c^{n+1}\Bigr)_{\mathcal C}=0.
	\end{align*}

And	summing over $n$ from 1 to $N$,
we arrive at the energy relation,
	\begin{align}
	&(\frac{c_0}{2}+\frac{\alpha^2}{2\lambda})\|p^{N+1}\|_{0}^2
	+G\|\epsilon(\mathbf{w}^{N+1})\|_{0}^2
	+\frac{1}{2\lambda}\|\xi^{N+1}\|_{0}^2	+\frac{c_f}{2}\|p_c^{N+1}\|_{0,\mathcal C}^2+\frac{L_w}{2}
	\|[\delta\mathbf{w}^{N+1}]\cdot \mathbf{n}^+\|_{0,\mathcal C}^2\notag\\
	&\qquad+\sum_{n=1}^{N}\left((\frac{c_0}{2}+\frac{\alpha^2}{2\lambda})\|\delta p^{n+1}\|_{0}^2
	+\frac{c_f}{2}\|\delta p_c^{n+1}\|_{0,\mathcal C}^2
	+\frac{1}{2\lambda}\|\delta\xi^{n+1}\|_{0}^2
	+G\|\epsilon(\delta\mathbf{w}^{n+1})\|_{0}^2\right)\notag\\
	&\qquad+\sum_{n=1}^{N}\left(\tau\Bigl\|\frac{\mathbf K^{1/2}}{\eta^{1/2}}\nabla p^{n+1}\Bigr\|_{0}^2
	+\tau\Big\|\frac{\mathbf K_f^{1/2}}{\eta^{1/2}}\overline\nabla p_c^{n+1}\Big\|_{0,\mathcal C}^2
	+\frac{L_w}{2}\|\delta j^{n+1}-\delta j^n\|_{0,\mathcal C}^2\right)\notag\\
	&=(\frac{c_0}{2}+\frac{\alpha^2}{2\lambda})\|p^{1}\|_{0}^2
	+G\|\epsilon(\mathbf{w}^{1})\|_{0}^2
	+\frac{1}{2\lambda}\|\xi^{1}\|_{0}^2	+\frac{c_f}{2}\|p_c^{1}\|_{0,\mathcal C}^2+\frac{L_w}{2}
	\|[\delta\mathbf{w}^{1}]\cdot \mathbf{n}^+\|_{0,\mathcal C}^2\notag\\
	&\qquad+ \sum_{n=1}^{N}(I_1^{n+1}+I_2^{n+1}),
\end{align}
where 
\begin{align*}
	&I_1^{n+1}=	\frac{\alpha}{\lambda}(\delta\xi^n,p^{n+1})_{\Omega\setminus\mathcal C}+\frac{\alpha}{\lambda}(\delta p^{n+1},\xi^{n+1})_{\Omega\setminus\mathcal C},\quad I_2^{n+1}=-(p_c^{n+1},\delta j^{n+1}-\delta j^n)_{\mathcal C}.
\end{align*}
	Using decomposition techniques, setting $\epsilon_1=\mathrm{min}\{1/8,4\alpha^2/c_0\lambda\},~ \epsilon_2=\mathrm{min}\{1/4,2\alpha^2/c_0\lambda\}$, we expand the term $I_1^{n+1}$ above as 
	\begin{align*}
		&\sum_{n=1}^{N}I_1=\sum_{n=1}^{N}\frac{\alpha}{\lambda}(\delta\xi^{n},\delta p^{n+1})_{\Omega\setminus\mathcal{C}}+\sum_{n=1}^{N}\frac{\alpha}{\lambda}(\delta p^{n+1},\delta\xi^{n+1})_{\Omega\setminus\mathcal{C}}+\frac{\alpha}{\lambda}(p^{N+1},\xi^{N})_{\Omega\setminus\mathcal{C}}-\frac{\alpha}{\lambda}(\xi^{0},p^1)_{\Omega\setminus\mathcal{C}}\\
		&\le\frac{\epsilon_1}{2\lambda}\|\delta\xi^{1}\|_0^2+ \frac{\epsilon_1}{2\lambda}\sum_{n=1}^{N-1}\|\delta\xi^{n+1}\|_0^2+\frac{\alpha^2}{2\epsilon_1\lambda}\sum_{n=1}^{N}\|\delta p^{n+1}\|_0^2+\frac{\epsilon_1}{2\lambda}\sum_{n=1}^{N}\|\delta\xi^{n+1}\|_0^2+\frac{\alpha^2}{2\epsilon_1\lambda}\sum_{n=1}^{N}\|\delta p^{n+1}\|_0^2\\
		&\qquad+\frac{\alpha^2}{2\epsilon_2\lambda}\|p^{N+1}\|_0^2+\frac{\epsilon_2}{2\lambda}\|\xi^N\|_0^2+\frac{\alpha^2}{2\epsilon_2\lambda}\|p^{1}\|_0^2+\frac{\epsilon_2}{2\lambda}\|\xi^0\|_0^2\\
		&\le \frac{1}{4\lambda}\sum_{n=1}^{N}\|\delta\xi^{n+1}\|_0^2+\frac{c_0}{4}\sum_{n=1}^{N}\|\delta p^{n+1}\|_0^2+\frac{c_0}{4}\|p^{N+1}\|_0^2\\
		&\quad+\frac{c_0}{4}\|p^{1}\|_0^2+\frac{1}{8\lambda}\|\xi^0\|_0^2+\frac{1}{16\lambda}\|\delta\xi^{1}\|_0^2+\frac{1}{4\lambda}\|\xi^{N+1}\|_0^2.
	\end{align*}

	In order to estimate the second coupling term $I_2$, we using the inequality (see, \cite{ALMANI2024117253,Girault2016}), let $C_1$, $C_2$, and $C_3$ denote, respectively, the constants of the trace, Poincar\'{e}, and Korn's inequality in $\Omega\setminus\mathcal{C}$: For all $\mathbf{v}\in \mathcal{W}$,
	$
	\|\mathbf{v}|_{\Omega\setminus\mathcal{C}}\|_{L^2(\mathcal{C})}\le C_1\|\mathbf{v}\|_{H^1(\Omega\setminus\mathcal{C})},~
	\|\mathbf{v}\|_{L^2(\Omega\setminus\mathcal{C})}\le C_2|\mathbf{v}|_{H^1(\Omega\setminus\mathcal{C})},~
	|\mathbf{v}|_{H^1(\Omega\setminus\mathcal{C})} \le C_3\|\varepsilon(\mathbf{v})\|_{L^2(\Omega\setminus\mathcal{C})}.
	$
	By combining these three inequalities we derive $\mathbf{v}\in \mathcal{W}$,
	\begin{align}\label{eq1}
		\|\epsilon(\mathbf{v})\|^2_{L^2(\Omega\setminus\mathcal{C})}\geq C^*\|\left[\mathbf{v}\right]\|^2_{L^2(\mathcal{C})}\geq C^*\|\left[\mathbf{v}\right]\cdot\mathbf{n}^+\|^2_{L^2(\mathcal{C})},
	\end{align}
	where $C^*=1/2C_1^2(C_2^2+1)C_3^2$.\\
	With Young's inequality and Cauchy--Schwarz inequality, choosing $\epsilon_3=L_w/2$, $L_w=4/c_f$, $\epsilon_4=\mathrm{min}\{C^*G,2/c_f\}$, we get,
	\begin{align*}
		&\sum_{n=1}^{N}	I_2^{n+1}=\sum_{n=1}^{N}(-(\delta j^{n+1},p_c^{n+1})+(\delta j^{n},\delta p_c^{n+1})+(\delta j^{n},p_c^n))_{\mathcal C}\\
		&=-(\delta j^{N+1},p_c^{N+1})_{\mathcal C}+(\delta j^{1}, p_c^{1})_{\mathcal C}+\sum_{n=1}^{N}(\delta j^{n},\delta p_c^{n+1})_{\mathcal C}\\
		&\le \frac{\epsilon_3}{2}\|\delta j^{N+1}\|^2_{0,\mathcal{C}}+\frac{1}{2\epsilon_3}\|p_c^{N+1}\|^2_{0,\mathcal{C}}+\frac{\epsilon_3}{2}\|\delta j^{1}\|^2_{0,\mathcal{C}}+\frac{1}{2\epsilon_3}\|p_c^{1}\|^2_{0,\mathcal{C}}\\
		&\quad+\sum_{n=1}^{N}\frac{\epsilon_4}{2}\|\delta j^{n+1}\|^2_{0,\mathcal{C}}+\sum_{n=1}^{N}\frac{1}{2\epsilon_4}\|\delta p_c^{n+1}\|^2_{0,\mathcal{C}}+\frac{\epsilon_4}{2}\|\delta j^{1}\|^2_{0,\mathcal{C}}\\
		&\le \frac{L_w}{4}\|\delta j^{N+1}\|^2_{0,\mathcal{C}}+\frac{c_f}{4}\|p_c^{N+1}\|^2_{0,\mathcal{C}}+\frac{c_f}{4}\|p_c^{1}\|^2_{0,\mathcal{C}}\\
		&\quad+\sum_{n=1}^{N}\frac{G}{2}	\|\epsilon(\delta\mathbf w^{n+1})\|_{0}^2+\sum_{n=1}^{N}\frac{c_f}{4}\|\delta p_c^{n+1}\|^2_{0,\mathcal{C}}+(\frac{L_w}{4}+\frac{G}{2})\|\delta j^{1}\|^2_{0,\mathcal{C}}.
	\end{align*}

We can obtain the result by summing all the equations above and using  the discrete Gronwall inequality, 
\begin{align}\label{seq}
	&(\frac{c_0}{4}+\frac{\alpha^2}{2\lambda})\|p^{N+1}\|_{0}^2
	+G\|\epsilon(\mathbf{w}^{N+1})\|_{0}^2
	+\frac{1}{4\lambda}\|\xi^{N+1}\|_{0}^2	+\frac{c_f}{4}\|p_c^{N+1}\|_{0,\mathcal C}^2+\frac{L_w}{4}
	\|[\delta\mathbf{w}^{N+1}]\cdot \mathbf{n}^+\|_{0,\mathcal C}^2\notag\\
	&\qquad+\sum_{n=1}^{N}\left((\frac{c_0}{4}+\frac{\alpha^2}{2\lambda})\|\delta p^{n+1}\|_{0}^2
	+\frac{c_f}{4}\|\delta p_c^{n+1}\|_{0,\mathcal C}^2
	+\frac{G}{2}\|\epsilon(\delta\mathbf{w}^{n+1})\|_{0}^2+\frac{1}{4\lambda}\sum_{n=1}^{N}\|\delta\xi^{n+1}\|_0^2\right)\notag\\
	&\qquad+\sum_{n=1}^{N}\left(\tau\Bigl\|\frac{\mathbf K^{1/2}}{\eta^{1/2}}\nabla p^{n+1}\Bigr\|_{0}^2
	+\tau\Big\|\frac{\mathbf K_f^{1/2}}{\eta^{1/2}}\overline\nabla p_c^{n+1}\Big\|_{0,\mathcal C}^2
	+\frac{L_w}{2}\|\delta j^{n+1}-\delta j^n\|_{0,\mathcal C}^2\right)\notag\\
	&\le(\frac{3c_0}{4}+\frac{\alpha^2}{2\lambda})\|p^{1}\|_{0}^2
	+G\|\epsilon(\mathbf{w}^{1})\|_{0}^2
	+\frac{3}{4\lambda}\|\xi^{1}\|_{0}^2+\frac{1}{8\lambda}\|\xi^{0}\|_{0}^2	+\frac{3c_f}{4}\|p_c^{1}\|_{0,\mathcal C}^2\notag\\
	&\qquad+\frac{1}{16\lambda}\|\delta\xi^{1}\|_0^2+(\frac{G}{2}+\frac{3L_w}{4})
	\|[\delta\mathbf{w}^{1}]\cdot \mathbf{n}^+\|_{0,\mathcal C}^2.
\end{align}
The proof is completed.
\end{proof}

\section{Full discretization and convergence analysis}\label{s4}
In this section, we formulate and analyze the fully discretized fractured poroelastic system using the time-decoupled discrete scheme
\subsection{Finite element spaces}
Let $\bar{\Omega}$ be partitioned into a conforming union of finite elements consisting of convex quadrilaterals in 2D or tetrahedra in 3D of characteristic size $h$ without hanging nodes. We denote the partition by $T_h$ and assume that it is shape-regular. For convenience, we also assume that $T_h$ triangulates exactly $\Omega\setminus\mathcal{C}$, and no element crosses $\mathcal{C}$. Since $\mathcal{C}$ is assumed to be a line or polygon, we triangulate by using line segments or as polygons by taking the trace of $T_h$.

In this work, Taylor--Hood elements are adopted for the pair $(\mathbf{w},\xi)$, while Lagrange finite elements are used for $p$ and $p_c$. The corresponding finite element spaces on $T_h$ are defined as follows.
We define the space of discrete displacements $\mathcal{W}_h$, the space of discrete pressure in the matrix $Q_h$, the space of discrete pressure in the fracture $Q_{ch}$, and  the space of discrete total pressure $\varPsi_h$ as follows:

$\mathcal{W}_h := \{\mathbf{v}_h \in H^1(\Omega\setminus\mathcal{C})^d; \,\forall T \in T_h, \mathbf{v}_h|_T \in \mathbb{P}_k^d, [\mathbf{v}_h]_{\Gamma \setminus \mathcal{C}} = 0, \mathbf{v}_h|_{\Gamma} = 0\},$

$\varPsi_h := \{\psi_h\in L^2(\Omega\setminus\mathcal{C}); \,\forall T \in T_h, \psi_h|_T \in \mathbb{P}_{k-1}\},$

$Q_h := \{p_h\in H^1(\Omega); \,\forall T \in T_h, p_h|_T \in \mathbb{P}_l\},$

$Q_{ch} := \{p_{ch} \in H^{1}(\mathcal{C});\, \forall T \in T_h, p_{ch}|_T \in \mathbb{P}_{l}\},$\\
where $k\ge2$ and $l\ge1$ are integers.

In addition, the space $Q_{ch}$ is equipped with the norm:
\begin{align*}
	|v|_{H^{1/2}(\mathcal{C})} = \left( \int_\mathcal{C} \int_\mathcal{C} \frac{|v(x) - v(y)|^2}{|x-y|^d} dx dy \right)^{1/2},~~ \|v\|_{H^{1/2}(\mathcal{C})} = (\|v\|^2_{L^2(\mathcal{C})} + |v|_{H^{1/2}(\mathcal{C})}^2)^{1/2}.
\end{align*}
The space of displacements $\mathcal{W}_h$ is equipped with the norm:
$\|\mathbf{v}_h\|_{\mathcal{V}_h}=\left(\sum_{i=1}^d\|\mathbf{v}_{hi}\|_{\Omega\setminus\mathcal{C}}^2\right)^{1/2}.$
Moreover, the spaces $Q_h$ and $\varPsi_h$ are normed by the usual $L^2$ norm. The finite element spaces satisfy the following discrete inf-sup condition. 
Specifically, there exists a positive constant $\tilde{\beta}$, independent of $h$, such that
\begin{equation}
	\sup_{\mathbf{v}_h \in \mathcal{W}_h\setminus \{0\}}
	\frac{(\phi_h,\nabla\cdot\mathbf{v}_h)_{\Omega \setminus \mathcal{C}}}{\|\mathbf{v}_h\|_{1}}
	\ge \tilde{\beta} \|\phi_h\|_{0},
	\quad \forall \phi_h \in \varPsi_h.
\end{equation}
	
\subsection{Fully discrete decoupling scheme.}\label{sf}
Based on the above finite element spaces, we construct a fully discrete decoupled scheme for the poroelastic model with fractures.

When $n=0$, we use the fully coupled formulation to calculate the results for the first step.\\
$\mathbf{Initial~step:}$
\begin{align}
	(c_0p_h+\alpha\nabla\cdot \mathbf{w}_h)(0)=c_0p^0_h+\alpha\nabla\cdot \mathbf{w}^0_h,\quad \xi^0_h=\alpha p_h^0-\lambda\nabla\cdot\mathbf{w}^0_h.
\end{align}
Given $p^0_h$ and $\mathbf{w}^0_h$, we propose to first apply the implicit Euler time-stepping scheme to compute $\mathbf{w}^1_h\in \mathcal{W}_h$, $\xi^1_h \in \varPsi_h$, $p^1_h\in Q_h$, and $p_{ch}^1\in Q_{ch}$, such that for any $(\mathbf{v}_h,\theta_h,q_h,q_{ch})\in \mathcal{W}_h\times\varPsi_h\times Q_h\times Q_{ch}$, satisfy the following coupling system:
\begin{align}
&	2G(\epsilon(\mathbf{w}^1_h),\epsilon(\mathbf{v}_h))_{\Omega\setminus\mathcal{C}}- (\xi^1_h,\nabla\cdot\mathbf{v}_h)_{\Omega\setminus\mathcal{C}}+(p_{ch}^1,[\mathbf{v}_h]\cdot\mathbf{n}^+)_{\mathcal{C}}= (\mathbf{f}^1,\mathbf{v}_h)_{\Omega\setminus\mathcal{C}},\\
&	\frac{1}{\lambda}(\frac{\xi^1_h-\xi^0_h}{\tau},\theta_h)_{\Omega\setminus\mathcal{C}}+(\frac{\nabla\cdot(\mathbf{w}^1_h-\mathbf{w}^0_h)}{\tau},\theta_h)_{\Omega\setminus\mathcal{C}}=\frac{\alpha}{\lambda}(\frac{p^1_h-p^0_h}{\tau},\theta_h)_{\Omega\setminus\mathcal{C}},\\
&	(c_0+\frac{\alpha^2}{\lambda})(\frac{p^1_h-p^0_h}{\tau},q_h)_{\Omega\setminus\mathcal{C}}-\frac{\alpha}{\lambda}(\frac{\xi^1_h-\xi^0_h}{\tau},q_h)_{\Omega\setminus\mathcal{C}}+(\frac{\mathbf{K}}{\eta}\nabla p^1_h,\nabla q_h)_{\Omega\setminus\mathcal{C}}	\notag\\
&\qquad-([\frac{\mathbf{K}}{\eta}\nabla p^1_h]\cdot\mathbf{n}^+,q_h)_{\mathcal{C}}=(\widetilde{g_s^1},q_h)_{\Omega\setminus\mathcal{C}},\\
&	(\frac{w_{fh}^1-w_{fh}^0}{\tau },q_{ch})_{\mathcal{C}}+(c_f\frac{p_{ch}^1-p_{ch}^0}{\tau },q_{ch})_{\mathcal{C}}+(\frac{\mathbf{K}_f}{\eta}\overline\nabla p_{ch}^1,\overline\nabla q_{ch})_{\mathcal{C}}	\notag\\
&\qquad+([\frac{\mathbf{K}}{\eta}\nabla p^1_h]\cdot\mathbf{n}^+,q_{ch})_{\mathcal{C}}=(\widetilde{g_f^1},q_{ch})_{\mathcal{C}}.
\end{align}
When $n\ge1$, we use the following decoupling algorithm to compute the unknowns.\\
$\mathbf{Step \,1:}$ Given $\mathbf{w}^{n}_h\in \mathcal{W}_h$,~ $\xi^{n-1}_h,~\xi^n_h \in \varPsi_h$, find $p^{n+1}_h\in Q_h$ and $p_{ch}^{n+1}\in Q_{ch}$ , such that for any $(q_h,q_{ch})\in Q_h\times Q_{ch}$, 
\begin{align}\label{fp}
&	(c_0+\frac{\alpha^2}{\lambda})(\frac{p^{n+1}_h-p^{n}_h}{\tau},q_h)_{\Omega\setminus\mathcal{C}}+(\frac{\mathbf{K}}{\eta}\nabla p^{n+1}_h,\nabla q_h)_{\Omega\setminus\mathcal{C}}	-([\frac{\mathbf{K}}{\eta}\nabla p^{n+1}_h]\cdot\mathbf{n}^+,q_h)_{\mathcal{C}}\notag\\
&	\quad=\frac{\alpha}{\lambda}(\frac{\xi^{n}_h-\xi^{n-1}_h}{\tau},q_h)_{\Omega\setminus\mathcal{C}}+(\widetilde{g_s^{n+1}},q_h)_{\Omega\setminus\mathcal{C}},\\\label{fpc}
&	(c_f\frac{p_{ch}^{n+1}-p_{ch}^{n}}{\tau },q_{ch})_{\mathcal{C}}+(
	\frac{\mathbf{K}_f}{\eta}\overline\nabla p_{ch}^{n+1},\overline\nabla q_{ch})_{\mathcal{C}}+([\frac{\mathbf{K}}{\eta}\nabla p^{n+1}_h]\cdot\mathbf{n}^+,q_{ch})_{\mathcal{C}}
	\notag\\
	&\quad=-(\frac{w_{fh}^{n}-w_{fh}^{n-1}}{\tau },q_{ch})_{\mathcal{C}}+(\widetilde{g_f^{n+1}},q_{ch})_{\mathcal{C}}.
\end{align}
\\
$\mathbf{Step \,2:}$ Find $\mathbf{w}_h^{n+1}\in \mathcal{W}_h$ and $\xi^{n+1}_h \in \varPsi_h$, such that for any $(\mathbf{v}_h,\theta_h)\in \mathcal{W}_h\times\varPsi_h$,  
\begin{align}\label{fw}
&	2G(\epsilon(\mathbf{w}_h^{n+1}),\epsilon(\mathbf{v}_h))_{\Omega\setminus\mathcal{C}}- (\xi^{n+1}_h,\nabla\cdot\mathbf{v}_h)_{\Omega\setminus\mathcal{C}}+(p_{ch}^{n+1},[\mathbf{v}_h]\cdot\mathbf{n}^+)_{\mathcal{C}}\notag\\
&\quad	+	L_w([\mathbf{w}^{n+1}_h-\mathbf{w}^{n}_h]\cdot\mathbf{n}^+,[\mathbf{v}_h]\cdot\mathbf{n}^+)_{\mathcal{C}}-L_w([\mathbf{w}_h^{n}-\mathbf{w}_h^{n-1}]\cdot\mathbf{n}^+,[\mathbf{v}_h]\cdot\mathbf{n}^+)_{\mathcal{C}}= (\mathbf{f}^{n+1},\mathbf{v}_h)_{\Omega\setminus\mathcal{C}},\\\label{fpsi}
&	(\frac{1}{\lambda}\xi^{n+1}_h,\theta_h)_{\Omega\setminus\mathcal{C}}+(\nabla\cdot\mathbf{w}^{n+1}_h,\theta_h)_{\Omega\setminus\mathcal{C}}=(\frac{\alpha}{\lambda}p^{n+1}_h,\theta_h)_{\Omega\setminus\mathcal{C}},
\end{align}
where the stabilization parameter $L_w$ will be defined later.\\
$\mathbf{Step \,3:}$ Given $\mathbf{w}^{n+1}_h\in \mathcal{W}_h$, find the width of the fracture $w_{fh}^{n+1}$, such that
\begin{align}
	w_{fh}^{n+1}=-[\mathbf{w}^{n+1}_h]_{\mathcal{C}}\cdot\mathbf{n}^+.
\end{align}
\subsection{Interpolation operators}
We construct the following interpolation operators:
\begin{align}
\Pi_{\mathbf{w}} \colon \mathcal{W} \to \mathcal{W}_h,\quad \Pi_{\xi}\colon \varPsi \to \varPsi_h,\quad \Pi_p\colon Q \to Q_h,\quad\Pi_{pc}\colon Q_c \to Q_{ch}.
\end{align}
\begin{lemma} [\cite{TP4-OyarzuaRhebergenSolano-2021-ESAIMMMNA}]\label{lew}
	For any $(\mathbf{w},\xi)\in \mathcal{W}\times\varPsi$, we define the interpolant $(\Pi_{\mathbf{w}}\mathbf{w},\Pi_{\xi}\xi)\in\mathcal{W}_h\times\varPsi_h$ as
	\begin{align*}
&	2G(\epsilon(\Pi_{\mathbf{w}}\mathbf{w}-\mathbf{w}),\epsilon(\mathbf{v}_h))_{\Omega\setminus\mathcal{C}}- (\Pi_{\xi}\xi-\xi,\nabla\cdot\mathbf{v}_h)_{\Omega\setminus\mathcal{C}}+( \Pi_{pc} p_c-p_c,[\mathbf{v}_h]\cdot\mathbf{n}^+)_{\mathcal{C}}=0,\quad \forall\mathbf{v}_h\in\mathcal{W}_h,\notag\\
&	(\nabla\cdot(\Pi_{\mathbf{w}}\mathbf{w}-\mathbf{w}),\theta_h)_{\Omega\setminus\mathcal{C}}=0,\quad \forall\theta_h\in\varPsi_h.
	\end{align*}
If $\mathbf{w}\in H^{k+1}_0(\Omega\setminus\mathcal{C})^d$ and $\xi\in H^k(\Omega\setminus\mathcal{C})$, we have the following approximation property:
	\begin{align}
	\|\Pi_{\mathbf{w}}\mathbf{w}-\mathbf{w}\|_1+\|\Pi_{\xi}\xi-\xi\|_0 \le Ch^k\Big(\|\mathbf{w}\|_{H^{k+1}(\Omega\setminus\mathcal{C})}+\|\xi\|_{H^k(\Omega\setminus\mathcal{C})}\Big).
	\end{align}
\end{lemma}

\begin{lemma} [\cite{he2026time,zhao2025optimally}]\label{lep}
We define the \( L^2 \)-projection operators, 
\( \Pi_{p}: Q \to Q_h \) and 
\( \Pi_{pc} : Q_c \to Q_{ch} \), which satisfy the following properties:
\begin{align}
	(p - \Pi_p p, q_h)_{\Omega\setminus\mathcal{C}} = 0, \quad\forall q_h\in Q_h,\quad
	(p_c - \Pi_{pc} p_c, q_{ch})_{\mathcal{C}} = 0,\quad\forall q_{ch}\in Q_{ch}.
\end{align}
Furthermore, the following error estimate holds:
\begin{align}
\|p-\Pi_p p\|_1\le h^l\|p\|_{H^{l+1}(\Omega\setminus\mathcal{C})},\quad \|p-\Pi_p p\|_0\le h^{l+1}\|p\|_{H^{l+1}(\Omega\setminus\mathcal{C})}.
\end{align}
\end{lemma}

\subsection{Error estimate}
\begin{assumption}\label{a1}
Assume that the exact solutions $\mathbf{w}$, $\xi$, $p$, and $p_c$ of the model \eqref{vwx}--\eqref{vwfx} satisfy the following regularity conditions:
\begin{align*}
&	\mathbf{w}\in L^{\infty}(0,T;H^{k+1}(\Omega\setminus\mathcal{C})^d),~~~\xi\in L^{\infty}(0,T;H^k(\Omega\setminus\mathcal{C})),~~~p\in L^{\infty}(0,T;H^{l+1}(\Omega\setminus\mathcal{C})),\\
&p_c \in L^2(0,T;H^{l+1}(\mathcal{C})),~~~	\partial_t\mathbf{w}\in L^{2}(0,T;H^{k+1}(\Omega\setminus\mathcal{C})^d),~~~	\partial_{tt}\mathbf{w}\in L^{2}(0,T;H^{1}(\Omega\setminus\mathcal{C})^d),\\
&\partial_{t}\xi\in L^{\infty}(0,T;H^k(\Omega\setminus\mathcal{C})),~~~\partial_{tt}\xi\in L^{2}(0,T;L^2(\Omega\setminus\mathcal{C})),~~~\partial_{t}p\in L^{2}(0,T;H^{l+1}(\Omega\setminus\mathcal{C})),\\
&\partial_{tt}p\in L^{2}(0,T;L^2(\Omega\setminus\mathcal{C})),~~~\partial_{t}p_c\in L^{2}(0,T;H^{l+1}(\mathcal{C})),~~~\partial_{tt}p_c\in L^{2}(0,T;L^2(\mathcal{C})).
\end{align*}
\end{assumption}
For ease of analysis, the expected errors are decomposed into
\begin{align*}
e_{\mathbf{w}}^{n}&:=\mathbf{w}^{n}-\mathbf{w}_h^{n}=\mathbf{w}^{n}-\Pi_{\mathbf{w}}\mathbf{w}^{n}+\Pi_{\mathbf{w}}\mathbf{w}^{n}-\mathbf{w}_h^{n}=e_{\mathbf{w}}^{I,n}+e_{\mathbf{w}}^{h,n},\\
e_{\xi}^{n}&:={\xi}^{n}-{\xi}^{n}_h={\xi}^{n}-\Pi_{\xi}\xi^n+\Pi_{\xi}\xi^n-{\xi}^{n}_h=e_{\xi}^{I,n}+e_{\xi}^{h,n},\\
e_{p}^{n}&:={p}^{n}-{p}^{n}_h={p}^{n}-\Pi_{p}p^n+\Pi_{p}p^n-{p}^{n}_h=e_{p}^{I,n}+e_{p}^{h,n},\\
e_{pc}^{n}&:=p_c^{n}-p_{ch}^{n}=p_c^{n}-\Pi_{pc}p_c^n+\Pi_{pc}p_c^n-p_{ch}^{n}=e_{pc}^{I,n}+e_{pc}^{h,n}.
\end{align*}
Meanwhile, we review the following differential notation:
\begin{align*}
\delta e_{\mathbf{w}}^{I,n+1}=e_{\mathbf{w}}^{I,n+1}-e_{\mathbf{w}}^{I,n},\quad \delta e_{\xi}^{I,n+1}=e_{\xi}^{I,n+1}-e_{\xi}^{I,n},\quad \delta e_{p}^{I,n+1}=e_{p}^{I,n+1}-e_{p}^{I,n},\quad \delta e_{pc}^{I,n+1}=e_{pc}^{I,n+1}-e_{pc}^{I,n}.
\end{align*}
Additionally, we will remark
the Taylor expansion as follows (see \cite{PhillipsWheeler2007a,PhillipsWheeler2007b}):
\begin{equation}\label{Taylor}
	X^n - X^{n-1}  = \Delta t \partial_t X^n + \int_{t_{n-1}}^{t_n} (s - t_{n-1}) \partial_{tt} X(s) \, \mathrm{d}s,
\end{equation}
\begin{equation}
	\left\|X^n - X^{n-1} \right\|\le \Delta t \| \partial_t X^n \| + (\Delta t)^{\frac{3}{2}} \| \partial_{tt} X \|_{L^2(t_{n-1}, t_n; L^2(\Omega))}.
\end{equation}

\begin{theorem}\label{th4}
Under Assumption \ref{a1}. For $1\le n\le N$, let $(\mathbf{w}^{n+1},\xi^{n+1},p^{n+1},p_c^{n+1})\in\mathcal{W}\times\varPsi\times Q\times Q_{c} $ denote the solutions of \eqref{vwx}--\eqref{vwfx} at the time $t^{n+1}$, and $(\mathbf{w}^{n+1}_h,\xi^{n+1}_h,p^{n+1}_h,p_{ch}^{n+1})\in\mathcal{W}_h\times\varPsi_h\times Q_h\times Q_{ch}$ represent the solutions of the fully discrete decoupling schemes \eqref{fp}--\eqref{fpsi}. If $L_w=\frac{4}{c_f}$, we obtain the following estimate:
\begin{align}
&\frac{c_0}{2}\|e_p^{h,N+1}\|^2_{L^{2}(\Omega\setminus\mathcal{C})}+\frac{c_f}{2}\|e_{pc}^{h,N+1}\|^2_{L^2(\mathcal{C})}+G\|\epsilon(e_{\mathbf{w}}^{h,N+1})\|_{L^{2}(\Omega\setminus\mathcal{C})}^2 +\frac{1}{2\lambda}\|e_{\xi}^{h,N+1}\|_{L^{2}(\Omega\setminus\mathcal{C})}^2\notag\\
&\qquad+\sum_{n=1}^{N}\left(\tau\|\sqrt{\frac{\mathbf{K}}{\eta}}\nabla e_p^{h,n+1}\|_{L^{2}(\Omega\setminus\mathcal{C})}^2+\tau\|\sqrt{\frac{\mathbf{K}_f}{\eta}}\overline\nabla e_{pc}^{h,n+1}\|_{L^2(\mathcal{C})}^2\right)\notag\\
&\le C\tau^2\left(\int_{0}^{T}\|\partial_{tt}p\|^2_{L^{2}(\Omega\setminus\mathcal{C})}\mathrm{d}s+\int_{0}^{T}\|\partial_{tt}\xi\|^2_{L^{2}(\Omega\setminus\mathcal{C})}\mathrm{d}s+\int_0^T \|\partial_{tt} p_c\|_{L^2(\mathcal{C})}^{2}\mathrm{d}s+\int_0^T
\|\partial_{tt}\mathbf{w}\|^2_{H^1(\Omega\setminus\mathcal{C})}\mathrm{d}s\right)\notag\\
&\qquad+Ch^{2l+2}\int_{0}^{T}\|\partial_tp\|^2_{H^{l+1}(\Omega\setminus\mathcal{C})}\mathrm{d}s+Ch^{2k}\int_{0}^{T}\|\partial_t\xi\|^2_{H^{k}(\Omega\setminus\mathcal{C})}\mathrm{d}s+Ch^{2k}\int_0^T
\|\partial_t\mathbf w\|^2_{H^{k+1}(\Omega\setminus\mathcal C)}\mathrm{d}s \notag\\
&\qquad+Ch^{2l+2} \int_0^T \|\partial_t p_c\|_{H^{l+1}(\mathcal{C})}^{2} \mathrm{d}s+C h^{2k+1}\int_0^T
\|\partial_t\mathbf{w}\|^2_{H^{k+1}(\Omega\setminus\mathcal{C})}\,\mathrm{d}s,
\end{align}
where $C$ is a generic positive constant independent of $h$ and $\tau$.
\end{theorem}

\begin{proof}
	Throughout the proof, $C$ denotes a generic positive constant independent of $h$ and $\tau$, 
$\epsilon$ is an arbitrarily small constant, whose value may be chosen as needed.
We begin by deriving the error equation by subtracting \eqref{fp}--\eqref{fw} from \eqref{vpx}, \eqref{vwfx} and \eqref{vwx} at the time $t^{n+1}$, respectively. Then, we have,
\begin{align*}
&	(c_0+\frac{\alpha^2}{\lambda})(\partial_tp^{n+1}-\frac{p^{n+1}_h-p^{n}_h}{\tau},q_h)_{\Omega\setminus\mathcal{C}}+(\frac{\mathbf{K}}{\eta}\nabla e_p^{n+1},\nabla q_h)_{\Omega\setminus\mathcal{C}}	-([\frac{\mathbf{K}}{\eta}\nabla e_p^{n+1}]\cdot\mathbf{n}^+,q_h)_{\mathcal{C}}\notag\\
&	\quad-\frac{\alpha}{\lambda}(\partial_t\xi^{n+1}-\frac{\xi^{n}_h-\xi^{n-1}_h}{\tau},q_h)_{\Omega\setminus\mathcal{C}}=0,\\
&c_f(\partial_tp^{n+1}_c-\frac{p_{ch}^{n+1}-p_{ch}^{n}}{\tau },q_{ch})_{\mathcal{C}}+	(
\frac{\mathbf{K}_f}{\eta}\overline\nabla e_{pc}^{n+1},\overline\nabla q_{ch})_{\mathcal{C}}+([\frac{\mathbf{K}}{\eta}\nabla e_p^{n+1}]\cdot\mathbf{n}^+,q_{ch})_{\mathcal{C}}
\notag\\
&\quad+(\partial_tw_f^{n+1}-\frac{w_{fh}^{n}-w_{fh}^{n-1}}{\tau },q_{ch})_{\mathcal{C}}=0,\\
&	2G(\epsilon(e_{\mathbf{w}}^{n+1}),\epsilon(\mathbf{v}_h))_{\Omega\setminus\mathcal{C}}- (e_{\xi}^{n+1},\nabla\cdot\mathbf{v}_h)_{\Omega\setminus\mathcal{C}}+(e_{pc}^{n+1},[\mathbf{v}_h]\cdot\mathbf{n}^+)_{\mathcal{C}}\notag\\
&\quad	+	L_w(([\partial_t\mathbf{w}^{n+1}]-[\mathbf{w}^{n+1}_h-\mathbf{w}^{n}_h])\cdot\mathbf{n}^+,[\mathbf{v}_h]\cdot\mathbf{n}^+)_{\mathcal{C}}-L_w(([\partial_t\mathbf{w}^{n+1}]-[\mathbf{w}_h^{n}-\mathbf{w}_h^{n-1}])\cdot\mathbf{n}^+,[\mathbf{v}_h]\cdot\mathbf{n}^+)_{\mathcal{C}}=\mathbf{0}.
\end{align*}
Making the difference between the $(n+1)$-th case and the $n$-th case of \eqref{vpsix} and \eqref{fpsi}, we obtain,
\begin{align}\label{epsi}
&	\frac{1}{\lambda}(e_{\xi}^{n+1}-e_{\xi}^{n},\theta_h)_{\Omega\setminus\mathcal{C}}+(\nabla\cdot (e_{\mathbf{w}}^{n+1}-e_{\mathbf{w}}^{n}),\theta_h)_{\Omega\setminus\mathcal{C}}=\frac{\alpha}{\lambda}(e_p^{n+1}-e_p^{n},\theta_h)_{\Omega\setminus\mathcal{C}}.
\end{align}
Using the properties of the projections, as well as the orthogonality in Lemma \ref{lew} and Lemma \ref{lep}, we rewrite the above equation as
\begin{align}
	\label{tp}
		&	(c_0+\frac{\alpha^2}{\lambda})(\delta e_p^{h,n+1},q_h)_{\Omega\setminus\mathcal{C}}+\tau(\frac{\mathbf{K}}{\eta}\nabla e_p^{h,n+1},\nabla q_h)_{\Omega\setminus\mathcal{C}}	-\tau([\frac{\mathbf{K}}{\eta}\nabla e_p^{n+1}]\cdot\mathbf{n}^+,q_h)_{\mathcal{C}}-\frac{\alpha}{\lambda}(\delta e_{\xi}^{h,n},q_h)_{\Omega\setminus\mathcal{C}}\notag\\
		&\quad =(c_0+\frac{\alpha^2}{\lambda})(\delta \Pi_p p^{n+1}-\tau\partial_t p^{n+1},q_h)_{\Omega\setminus\mathcal{C}}+\frac{\alpha}{\lambda}(\tau \partial_t\xi^{n+1}-\delta\Pi_{\xi}\xi^{n},q_h)_{\Omega\setminus\mathcal{C}},\\\label{tpc}
		&c_f(\delta e_{pc}^{h,n+1},q_{ch})_{\mathcal{C}}+\tau	(
		\frac{\mathbf{K}_f}{\eta}\overline\nabla e_{pc}^{h,n+1},\overline\nabla q_{ch})_{\mathcal{C}}+\tau([\frac{\mathbf{K}}{\eta}\nabla e_p^{n+1}]\cdot\mathbf{n}^+,q_{ch})_{\mathcal{C}}-([\delta e_{\mathbf{w}}^{h,n}]\cdot\mathbf{n}^+,q_{ch})_{\mathcal{C}}\notag\\
		&\quad=c_f(\delta \Pi_{pc}p_{c}^{n+1}-\tau\partial_tp_c^{n+1},q_{ch})_{\mathcal{C}}+(\tau\partial_t[\mathbf{w}^{n+1}]\cdot\mathbf{n}^+-[\delta \Pi_{\mathbf{w}}\mathbf{w}^n]\cdot\mathbf{n}^+,q_{ch})_{\mathcal{C}},\\\label{tw}
		&	2G(\epsilon(e_{\mathbf{w}}^{h,n+1}),\epsilon(\mathbf{v}_h))_{\Omega\setminus\mathcal{C}}- (e_{\xi}^{h,n+1},\nabla\cdot\mathbf{v}_h)_{\Omega\setminus\mathcal{C}}+(e_{pc}^{n+1},[\mathbf{v}_h]\cdot\mathbf{n}^+)_{\mathcal{C}}\notag\\
		&\qquad	-	L_w([\delta\Pi_{\mathbf{w}}\mathbf{w}^{n+1}-\delta \Pi_{\mathbf{w}}\mathbf{w}^{n}]\cdot\mathbf{n}^+,[\mathbf{v}_h]\cdot\mathbf{n}^+)_{\mathcal{C}}+L_w([\delta e_{\mathbf{w}}^{h,n+1}-\delta e_{\mathbf{w}}^{h,n}]\cdot\mathbf{n}^+,[\mathbf{v}_h]\cdot\mathbf{n}^+)_{\mathcal{C}}=\mathbf{0},\\\label{psi0}
		&	\frac{1}{\lambda}(\delta e_{\xi}^{h,n+1},\theta_h)_{\Omega\setminus\mathcal{C}}+\frac{1}{\lambda}(\delta e_{\xi}^{I,n+1},\theta_h)_{\Omega\setminus\mathcal{C}}+(\nabla\cdot (\delta e_{\mathbf{w}}^{h,n+1}),\theta_h)_{\Omega\setminus\mathcal{C}}+(\nabla\cdot (\delta e_{\mathbf{w}}^{I,n+1}),\theta_h)_{\Omega\setminus\mathcal{C}}\notag\\
		&\quad=\frac{\alpha}{\lambda}(\delta e_p^{h,n+1},\theta_h)_{\Omega\setminus\mathcal{C}}+\frac{\alpha}{\lambda}(\delta e_p^{I,n+1},\theta_h)_{\Omega\setminus\mathcal{C}}.
\end{align}
According to the second equation in the variational form \eqref{vpsix}, we get the total pressure equation with discrete time difference,
\begin{align}\label{psi1}
	(\frac{1}{\lambda}\delta\xi^{n+1},\theta_h)_{\Omega\setminus\mathcal{C}}+(\nabla\cdot(\delta\mathbf{w}^{n+1}),\theta_h)_{\Omega\setminus\mathcal{C}}=\frac{\alpha}{\lambda}(\delta p^{n+1},\theta_h)_{\Omega\setminus\mathcal{C}},
\end{align}
and the total pressure equation of the continuous time derivatives,
\begin{align}\label{psi2}
	(\frac{1}{\lambda}\tau\partial_t\xi^{n+1},\theta_h)_{\Omega\setminus\mathcal{C}}+(\tau\nabla\cdot(\partial_t\mathbf{w}^{n+1}),\theta_h)_{\Omega\setminus\mathcal{C}}=\frac{\alpha}{\lambda}(\tau\partial_t p^{n+1},\theta_h)_{\Omega\setminus\mathcal{C}}.
\end{align}
Combining \eqref{psi0}--\eqref{psi2}, we can obtain,
\begin{align}\label{erpsi}
&\frac{1}{\lambda}(\delta e_{\xi}^{h,n+1},\theta_h)_{\Omega\setminus\mathcal{C}}+(\nabla\cdot (\delta e_{\mathbf{w}}^{h,n+1}),\theta_h)_{\Omega\setminus\mathcal{C}}-\frac{\alpha}{\lambda}(\delta e_p^{h,n+1},\theta_h)_{\Omega\setminus\mathcal{C}}\notag\\
&\quad=\frac{1}{\lambda}(\delta\xi^{n+1}-\tau\partial_t\xi^{n+1},\theta_h)_{\Omega\setminus\mathcal{C}}+(\nabla\cdot(\delta\Pi_{\mathbf{w}}\mathbf{w}^{n+1})-\tau\nabla\cdot(\partial_t\mathbf{w}^{n+1}),\theta_h)_{\Omega\setminus\mathcal{C}}\notag\\
&\qquad+\frac{\alpha}{\lambda}(\tau\partial_t p^{n+1}-\delta\Pi_p p^{n+1},\theta_h)_{\Omega\setminus\mathcal{C}}.
\end{align}
Then, setting $q_h=e_p^{h,n+1}$, $q_{ch}=e_{pc}^{h,n+1}$, $\mathbf{v}_h=\delta e_{\mathbf{w}}^{h,n+1}$ and $\theta_h=e_{\xi}^{h,n+1}$ in the error Equations \eqref{tp}--\eqref{tw} and \eqref{erpsi}, we can get,
\begin{align}
		&	(c_0+\frac{\alpha^2}{\lambda})(\delta e_p^{h,n+1},e_p^{h,n+1})_{\Omega\setminus\mathcal{C}}+\tau(\frac{\mathbf{K}}{\eta}\nabla e_p^{h,n+1},\nabla e_p^{h,n+1})_{\Omega\setminus\mathcal{C}}	-\tau([\frac{\mathbf{K}}{\eta}\nabla e_p^{n+1}]\cdot\mathbf{n}^+,e_p^{h,n+1})_{\mathcal{C}}\notag\\
		&\quad = (c_0+\frac{\alpha^2}{\lambda})(\delta \Pi_p p^{n+1}-\tau\partial_t p^{n+1},e_p^{h,n+1})_{\Omega\setminus\mathcal{C}}\notag\\
		&\qquad +\frac{\alpha}{\lambda}(\tau \partial_t\xi^{n+1}-\delta\Pi_{\xi}\xi^{n},e_p^{h,n+1})_{\Omega\setminus\mathcal{C}}+\frac{\alpha}{\lambda}(\delta e_{\xi}^{h,n},e_p^{h,n+1})_{\Omega\setminus\mathcal{C}},\\
		&c_f(\delta e_{pc}^{h,n+1},e_{pc}^{h,n+1})_{\mathcal{C}}+\tau	(
		\frac{\mathbf{K}_f}{\eta}\overline\nabla e_{pc}^{h,n+1},\overline\nabla e_{pc}^{h,n+1})_{\mathcal{C}}+\tau([\frac{\mathbf{K}}{\eta}\nabla e_p^{n+1}]\cdot\mathbf{n}^+,e_{pc}^{h,n+1})_{\mathcal{C}}-([\delta e_{\mathbf{w}}^{h,n}]\cdot\mathbf{n}^+,e_{pc}^{h,n+1})_{\mathcal{C}}\notag\\
		&\quad=c_f(\delta \Pi_{pc}p_{c}^{n+1}-\tau\partial_tp_c^{n+1},e_{pc}^{h,n+1})_{\mathcal{C}}+(\tau\partial_t[\mathbf{w}^{n+1}]\cdot\mathbf{n}^+-[\delta \Pi_{\mathbf{w}}\mathbf{w}^n]\cdot\mathbf{n}^+,e_{pc}^{h,n+1})_{\mathcal{C}},\\
		&	2G(\epsilon(e_{\mathbf{w}}^{h,n+1}),\epsilon(\delta e_{\mathbf{w}}^{h,n+1}))_{\Omega\setminus\mathcal{C}}- (e_{\xi}^{h,n+1},\nabla\cdot(\delta e_{\mathbf{w}}^{h,n+1}))_{\Omega\setminus\mathcal{C}}\notag\\
		&\qquad+(e_{pc}^{n+1},[\delta e_{\mathbf{w}}^{h,n+1}]\cdot\mathbf{n}^+)_{\mathcal{C}}-	L_w([\delta\Pi_{\mathbf{w}}\mathbf{w}^{n+1}-\delta \Pi_{\mathbf{w}}\mathbf{w}^{n}]\cdot\mathbf{n}^+,[\delta e_{\mathbf{w}}^{h,n+1}]\cdot\mathbf{n}^+)_{\mathcal{C}}\notag\\
		&\qquad	+L_w([\delta e_{\mathbf{w}}^{h,n+1}-\delta e_{\mathbf{w}}^{h,n}]\cdot\mathbf{n}^+,[\delta e_{\mathbf{w}}^{h,n+1}]\cdot\mathbf{n}^+)_{\mathcal{C}}=\mathbf{0},\\
		&	\frac{1}{\lambda}(\delta e_{\xi}^{h,n+1},e_{\xi}^{h,n+1})_{\Omega\setminus\mathcal{C}}+(\nabla\cdot (\delta e_{\mathbf{w}}^{h,n+1}),e_{\xi}^{h,n+1})_{\Omega\setminus\mathcal{C}}-\frac{\alpha}{\lambda}(\delta e_p^{h,n+1},e_{\xi}^{h,n+1})_{\Omega\setminus\mathcal{C}}\notag\\
		&\quad=\frac{1}{\lambda}(\delta\xi^{n+1}-\tau\partial_t\xi^{n+1},e_{\xi}^{h,n+1})_{\Omega\setminus\mathcal{C}}+(\nabla\cdot(\delta\Pi_{\mathbf{w}}\mathbf{w}^{n+1}-\tau\partial_t\mathbf{w}^{n+1}),e_{\xi}^{h,n+1})_{\Omega\setminus\mathcal{C}}\notag\\
		&\qquad+\frac{\alpha}{\lambda}(\tau\partial_t p^{n+1}-\delta\Pi_p p^{n+1},e_{\xi}^{h,n+1})_{\Omega\setminus\mathcal{C}}.
\end{align}
Adding the above equations, using Lemma \ref{lew} and Lemma \ref{lep}, and summing over $n$ from 1 to $N$, we get LHS = RHS, that is, the terms on the left side of the equation are
\begin{align}
\mathrm{LHS}&:=(\frac{c_0}{2}+\frac{\alpha^2}{2\lambda})\|e_p^{h,N+1}\|^2_{0}-(\frac{c_0}{2}+\frac{\alpha^2}{2\lambda})\|e_p^{h,1}\|^2_{0}+\frac{c_f}{2}\|e_{pc}^{h,N+1}\|^2_{0,\mathcal{C}}-\frac{c_f}{2}\|e_{pc}^{h,1}\|^2_{0,\mathcal{C}}\notag\\
&\qquad +G\|\epsilon(e_{\mathbf{w}}^{h,N+1})\|_0^2-G\|\epsilon(e_{\mathbf{w}}^{h,1})\|_0^2+\frac{1}{2\lambda}\|e_{\xi}^{h,N+1}\|_0^2-\frac{1}{2\lambda}\|e_{\xi}^{h,1}\|_0^2+\sum_{n=1}^{N}\left(\tau\|\sqrt{\frac{\mathbf{K}}{\eta}}\nabla e_p^{h,n+1}\|_0^2\right)\notag\\
&\qquad+\frac{L_w}{2}\|[\delta e_{\mathbf{w}}^{h,N+1}]\cdot\mathbf{n}^+\|_{0,\mathcal{C}}^2-\frac{L_w}{2}\|[\delta e_{\mathbf{w}}^{h,1}]\cdot\mathbf{n}^+\|_{0,\mathcal{C}}^2+\sum_{n=1}^{N}\left(\tau\|\sqrt{\frac{\mathbf{K}_f}{\eta}}\overline\nabla e_{pc}^{h,n+1}\|_{0,\mathcal{C}}^2\right)\notag\\
&\qquad +\sum_{n=1}^{N}\left((\frac{c_0}{2}+\frac{\alpha^2}{2\lambda})\|\delta e_p^{h,n+1}\|^2_{0}+\frac{c_f}{2}\|\delta e_{pc}^{h,n+1}\|^2_{0,\mathcal{C}}\right)\notag\\
&\qquad+\sum_{n=1}^{N}\left(G\|\epsilon(\delta e_{\mathbf{w}}^{h,n+1})\|_0^2+\frac{1}{2\lambda}\|\delta e_{\xi}^{h,n+1}\|_0^2+\frac{L_w}{2}\|[\delta e_{\mathbf{w}}^{h,n+1}-\delta e_{\mathbf{w}}^{h,n}]\cdot\mathbf{n}^+\|_{0,\mathcal{C}}^2\right),
\end{align}
and the terms on the right side of the equation are 
\begin{align}\label{Re}
\mathrm{RHS}&:=\sum_{n=1}^{N}(c_0+\frac{\alpha^2}{\lambda})(\delta \Pi_p p^{n+1}-\tau\partial_t p^{n+1},e_p^{h,n+1})_{\Omega\setminus\mathcal{C}}+\sum_{n=1}^{N}\frac{\alpha}{\lambda}(\tau \partial_t\xi^{n+1}-\delta\Pi_{\xi}\xi^{n},e_p^{h,n+1})_{\Omega\setminus\mathcal{C}}\notag\\
&\qquad+\sum_{n=1}^{N}\frac{\alpha}{\lambda}(\delta e_{\xi}^{h,n},e_p^{h,n+1})_{\Omega\setminus\mathcal{C}}+\sum_{n=1}^{N}\frac{\alpha}{\lambda}(\delta e_p^{h,n+1},e_{\xi}^{h,n+1})_{\Omega\setminus\mathcal{C}}+\sum_{n=1}^{N}c_f(\delta \Pi_{pc}p_{c}^{n+1}-\tau\partial_tp_c^{n+1},e_{pc}^{h,n+1})_{\mathcal{C}}\notag\\
&\qquad+\sum_{n=1}^{N}(\tau\partial_t[\mathbf{w}^{n+1}]\cdot\mathbf{n}^+-[\delta \Pi_{\mathbf{w}}\mathbf{w}^n]\cdot\mathbf{n}^+,e_{pc}^{h,n+1})_{\mathcal{C}}-\sum_{n=1}^{N}([\delta e_{\mathbf{w}}^{h,n+1}-\delta e_{\mathbf{w}}^{h,n}]\cdot\mathbf{n}^+,e_{pc}^{h,n+1})_{\mathcal{C}}\notag\\
&\qquad+\sum_{n=1}^{N}
L_w([\delta\Pi_{\mathbf{w}}\mathbf{w}^{n+1}-\delta \Pi_{\mathbf{w}}\mathbf{w}^{n}]\cdot\mathbf{n}^+,[\delta e_{\mathbf{w}}^{h,n+1}]\cdot\mathbf{n}^+)_{\mathcal{C}}\notag\\
&\qquad+\sum_{n=1}^{N}\frac{1}{\lambda}(\delta\xi^{n+1}-\tau\partial_t\xi^{n+1},e_{\xi}^{h,n+1})_{\Omega\setminus\mathcal{C}}+\sum_{n=1}^{N}(\nabla\cdot(\delta\Pi_{\mathbf{w}}\mathbf{w}^{n+1}-\tau\partial_t\mathbf{w}^{n+1}),e_{\xi}^{h,n+1})_{\Omega\setminus\mathcal{C}}\notag\\
&\qquad+\sum_{n=1}^{N}\frac{\alpha}{\lambda}(\tau\partial_t p^{n+1}-\delta\Pi_p p^{n+1},e_{\xi}^{h,n+1})_{\Omega\setminus\mathcal{C}}\notag\\
&:=\sum_{i=1}^{11}E_i.
\end{align}
Then, we proceed to bound each term $E_i,\, i=1,\cdots,11,$ in \eqref{Re} individually. By applying the Cauchy--Schwarz inequality, Young's inequality, the first terms $E_1$ and the second terms $E_2$
can be estimated as follows:
\begin{align*}
E_1&=\sum_{n=1}^{N}(\frac{c_0}{2}+\frac{\alpha^2}{2\lambda})(\delta \Pi_p p^{n+1}-\tau\partial_t p^{n+1},e_p^{h,n+1})_{\Omega\setminus\mathcal{C}}\\
&\le \sum_{n=1}^{N}\tau \frac{c_0}{4}\|e_p^{h,n+1}\|^2_0+\frac{\epsilon_1\alpha^2}{4}\sum_{n=1}^{N}\tau\|e_p^{h,n+1}\|^2_0\\
&\quad+C(c_0+\frac{\alpha^2}{\epsilon_1\lambda^2})\left(h^{2l+2}\int_{0}^{T}\|\partial_tp\|^2_{H^{l+1}(\Omega\setminus\mathcal{C})}\mathrm{d}s+\tau^2\int_{0}^{T}\|\partial_{tt}p\|^2_{L^{2}(\Omega\setminus\mathcal{C})}\mathrm{d}s\right),\\
E_2&\le \frac{\epsilon_2\alpha^2}{4}\sum_{n=1}^{N}\tau\|e_p^{h,n+1}\|^2_0+\frac{C}{\epsilon_2\lambda^2}\left(h^{2k}\int_{0}^{T}\|\partial_t\xi\|^2_{H^{k}(\Omega\setminus\mathcal{C})}\mathrm{d}s+\tau^2\int_{0}^{T}\|\partial_{tt}\xi\|^2_{L^{2}(\Omega\setminus\mathcal{C})}\mathrm{d}s\right).
\end{align*}
To estimate the terms $E_3$ and $E_4$, using decomposition techniques, we have,
\begin{align*}
	&E_3+E_4=\sum_{n=1}^{N}\frac{\alpha}{\lambda}
	(\delta e_{\xi}^{h,n},e_p^{h,n+1})_{\Omega\setminus\mathcal C}
	+
	\sum_{n=1}^{N}\frac{\alpha}{\lambda}
	(\delta e_p^{h,n+1},e_{\xi}^{h,n+1})_{\Omega\setminus\mathcal C}
	\notag\\
	&=\frac{\alpha}{\lambda}\sum_{n=1}^{N}(\delta e_{\xi}^{h,n},\delta e_p^{h,n+1})_{\Omega\setminus\mathcal C}+\frac{\alpha}{\lambda}\sum_{n=1}^{N}(\delta e_{\xi}^{h,n+1},\delta e_p^{h,n+1})_{\Omega\setminus\mathcal C}+\frac{\alpha}{\lambda}(e_{\xi}^{h,N}, e_p^{h,N+1})_{\Omega\setminus\mathcal C}-\frac{\alpha}{\lambda}(e_{\xi}^{h,0}, e_p^{h,1})_{\Omega\setminus\mathcal C}	\notag\\
	&\le \frac{1}{4\lambda}\sum_{n=1}^{N}\|\delta e_{\xi}^{h,n+1}\|_0^2+\frac{c_0}{4}\sum_{n=1}^{N}\|\delta e_p^{h,n+1}\|_0^2+\frac{c_0}{4}\|e_p^{h,N+1}\|_0^2+\frac{c_0}{4}\|e_p^{h,1}\|_0^2\notag\\
	&\quad+\frac{1}{8\lambda}\|e_{\xi}^{h,0}\|_0^2+\frac{1}{16\lambda}\|\delta e_{\xi}^{h,1}\|_0^2+\frac{1}{4\lambda}\|e_{\xi}^{h,N+1}\|_0^2.
\end{align*}
Utilizing the technique similar to that in the first item, we can obtain the estimate for $E_5$, as follows:
\begin{align*}
	E_5
	\leq \sum_{n=1}^N \tau \frac{c_f}{8} \|e_{pc}^{h,n+1}\|_{0,\mathcal{C}}^2
	+ \frac{C c_f}{4} \left( h^{2l+2} \int_0^T \|\partial_t p_c\|_{H^{l+1}(\mathcal{C})}^{2} \mathrm{d}s + \tau^2 \int_0^T \|\partial_{tt} p_c\|_{L^2(\mathcal{C})}^{2}\mathrm{d}s \right).
\end{align*}
For the term $E_6$, adding and subtracting $[\delta\mathbf{w}^n]\cdot\mathbf{n}^+$, using Taylor expansion \eqref{Taylor} in time and the trace inequality, we conclude that
\begin{align*}
E_6=	&\sum_{n=1}^{N}\big((\tau\partial_t[\mathbf{w}^{n+1}]\cdot\mathbf{n}^+-[\delta\mathbf{w}^n]\cdot\mathbf{n}^+)-([\delta \Pi_{\mathbf{w}}\mathbf{w}^n]\cdot\mathbf{n}^+-[\delta\mathbf{w}^n]\cdot\mathbf{n}^+),e_{pc}^{h,n+1}\big)_{\mathcal{C}}
	\notag\\
	&\le
	\sum_{n=1}^{N}\tau\frac{c_f}{8}\|e_{pc}^{h,n+1}\|^2_{0,\mathcal{C}}
	+C h^{2k+1}\int_0^T
	\|\partial_t\mathbf{w}\|^2_{H^{k+1}(\Omega\setminus\mathcal{C})}\,\mathrm{d}s
	+C\tau^2\int_0^T
	\|\partial_{tt}\mathbf{w}\|^2_{H^1(\Omega\setminus\mathcal{C})}\,\mathrm{d}s.
\end{align*}
By applying the Cauchy--Schwarz inequality, Young's inequality, we can bound the term $E_7$,
\begin{align*}
E_7&=\sum_{n=1}^{N}\left(-([\delta e_{\mathbf{w}}^{h,n+1}]\cdot\mathbf{n}^+,p_c^{n+1})_{\mathcal C}+([\delta e_{\mathbf{w}}^{h,n}]\cdot\mathbf{n}^+,\delta p_c^{n+1})_{\mathcal C}+([\delta e_{\mathbf{w}}^{h,n}]\cdot\mathbf{n}^+,p_c^n)_{\mathcal C}\right)\\
&\le \frac{L_w}{4}\|[\delta e_{\mathbf{w}}^{h,N+1}]\cdot\mathbf{n}^+\|_{0,\mathcal{C}}^2+\frac{c_f}{4}\|e_{pc}^{h,N+1}\|^2_{0,\mathcal{C}}+\frac{c_f}{4}\|e_{pc}^{h,1}\|^2_{0,\mathcal{C}}+(\frac{L_w}{4}+\frac{G}{2})\|[\delta e_{\mathbf{w}}^{h,1}]\cdot\mathbf{n}^+\|_{0,\mathcal{C}}^2\\
&\quad+\sum_{n=1}^{N}\frac{G}{2}\|\epsilon(\delta e_{\mathbf{w}}^{h,n+1})\|_0^2+\sum_{n=1}^{N}\frac{c_f}{4}\|\delta e_{pc}^{h,n+1}\|^2_{0,\mathcal{C}}.
\end{align*}
For the term $E_8$, by the trace inequality and the integral representation of the second-order time difference, we obtain,
\begin{align*}
	E_8&=\sum_{n=1}^{N}
	L_w([\delta\Pi_{\mathbf{w}}\mathbf{w}^{n+1}-\delta \Pi_{\mathbf{w}}\mathbf{w}^{n}]\cdot\mathbf{n}^+,
	[\delta e_{\mathbf{w}}^{h,n+1}]\cdot\mathbf{n}^+)_{\mathcal{C}}
	\notag\\
	&\le
	\sum_{n=1}^{N}\tau \frac{L_w}{4}
	\|[\delta e_{\mathbf w}^{h,n+1}]\cdot\mathbf n^+\|_{0,\mathcal C}^2
+
	CL_w\left(\tau^2\int_0^T
	\|\partial_{tt}\mathbf w\|_{H^1(\Omega\setminus\mathcal C)}^2\mathrm{d}s+h^{2k+1}\int_0^T
	\|\partial_{t}\mathbf w\|_{H^{k+1}(\Omega\setminus\mathcal C)}^2\mathrm{d}s\right).
\end{align*}
For any $\epsilon_3>0$, we obtain the following bound for the term $E_9$,
\begin{align*}
E_9&\le \frac{\epsilon_3\tau}{3}\sum_{n=1}^{N}\|e_{\xi}^{h,n+1}\|_0^2+\frac{C}{\epsilon_3\lambda^2}\tau^2\int_{0}^{T}\|\partial_{tt}\xi\|^2_{L^{2}(\Omega\setminus\mathcal{C})}\mathrm{d}s.
\end{align*}
With the aid of the Cauchy--Schwarz and Young's inequalities and the properties of the projection operators, $E_{10}$ and $E_{11}$ can be bounded as follows,
\begin{align*}
E_{10}	
	&\le
	\sum_{n=1}^{N}\tau\frac{\epsilon_4}{3}\|e_{\xi}^{h,n+1}\|^2_0+
	\frac{C}{\epsilon_4}\left(h^{2k}\int_0^T
	\|\partial_t\mathbf w\|^2_{H^{k+1}(\Omega\setminus\mathcal C)}\mathrm{d}s+\tau^2\int_0^T
	\|\partial_{tt}\mathbf w\|^2_{H^1(\Omega\setminus\mathcal C)}\mathrm{d}s\right),\\
	E_{11}	
	&\le
	\sum_{n=1}^{N}\tau\frac{\epsilon_5}{3}\|e_{\xi}^{h,n+1}\|^2_0+
	\frac{C\alpha^2}{\epsilon_5\lambda^2}\left(h^{2l+2}\int_0^T
	\|\partial_tp\|^2_{H^{l+1}(\Omega\setminus\mathcal C)}\mathrm{d}s+\tau^2\int_0^T
	\|\partial_{tt}p\|^2_{L^2(\Omega\setminus\mathcal C)}\mathrm{d}s\right).
\end{align*}

Let $\tilde{C}$ be a constant depending on $\epsilon_j$ ($j=1,\cdots,5$) and the model parameters. Combining the estimates of $E_i$, $i=1,\cdots,11$, we obtain,
\begin{align}
\mathrm{LHS}&\ge(\frac{c_0}{2}+\frac{\alpha^2}{2\lambda})\|e_p^{h,N+1}\|^2_{0}+\frac{c_f}{2}\|e_{pc}^{h,N+1}\|^2_{0,\mathcal{C}}+G\|\epsilon(e_{\mathbf{w}}^{h,N+1})\|_0^2+\frac{L_w}{2}\|[\delta e_{\mathbf{w}}^{h,N+1}]\cdot\mathbf{n}^+\|_{0,\mathcal{C}}^2\notag\\
&\qquad+\frac{1}{2\lambda}\|e_{\xi}^{h,N+1}\|_0^2 +\sum_{n=1}^{N}\left((\frac{c_0}{2}+\frac{\alpha^2}{2\lambda})\|\delta e_p^{h,n+1}\|^2_{0}+\frac{c_f}{2}\|\delta e_{pc}^{h,n+1}\|^2_{0,\mathcal{C}}+G\|\epsilon(\delta e_{\mathbf{w}}^{h,n+1})\|_0^2\right)\notag\\
&\qquad+\sum_{n=1}^{N}\left(\frac{1}{2\lambda}\|\delta e_{\xi}^{h,n+1}\|_0^2+\frac{L_w}{2}\|[\delta e_{\mathbf{w}}^{h,n+1}-\delta e_{\mathbf{w}}^{h,n}]\cdot\mathbf{n}^+\|_{0,\mathcal{C}}^2\right)\notag\\
&\qquad+\sum_{n=1}^{N}\left(\tau\|\sqrt{\frac{\mathbf{K}}{\eta}}\nabla e_p^{h,n+1}\|_0^2+\tau\|\sqrt{\frac{\mathbf{K}_f}{\eta}}\overline\nabla e_{pc}^{h,n+1}\|_{0,\mathcal{C}}^2\right),
\end{align}
and the terms on the right side of the equation are bounded as
\begin{align*}
\mathrm{RHS}&\le(\frac{c_0}{2}+\frac{\alpha^2}{2\lambda})\|e_p^{h,1}\|^2_{0}+\frac{c_f}{2}\|e_{pc}^{h,1}\|^2_{0,\mathcal{C}}+G\|\epsilon(e_{\mathbf{w}}^{h,1})\|_0^2+\frac{1}{2\lambda}\|e_{\xi}^{h,1}\|_0^2+\frac{L_w}{2}\|[\delta e_{\mathbf{w}}^{h,1}]\cdot\mathbf{n}^+\|_{0,\mathcal{C}}^2\notag\\
&\qquad +\sum_{n=1}^{N}\tau \frac{c_0}{4}\|e_p^{h,n+1}\|^2_0+\frac{\epsilon_1\alpha^2}{4}\sum_{n=1}^{N}\tau\|e_p^{h,n+1}\|^2_0+\frac{\epsilon_2\alpha^2}{4}\sum_{n=1}^{N}\tau\|e_p^{h,n+1}\|^2_0\notag\\
&\qquad +\frac{1}{4\lambda}\sum_{n=1}^{N}\|\delta e_{\xi}^{h,n+1}\|_0^2+\frac{c_0}{4}\sum_{n=1}^{N}\|\delta e_p^{h,n+1}\|_0^2+\frac{c_0}{4}\|e_p^{h,N+1}\|_0^2+\frac{c_0}{4}\|e_p^{h,1}\|_0^2\notag\\
&\qquad+\frac{1}{8\lambda}\|e_{\xi}^{h,0}\|_0^2+\frac{1}{16\lambda}\|\delta e_{\xi}^{h,1}\|_0^2+\frac{1}{4\lambda}\|e_{\xi}^{h,N+1}\|_0^2+\sum_{n=1}^N \tau \frac{c_f}{8} \|e_{pc}^{h,n+1}\|_{0,\mathcal{C}}^2+\sum_{n=1}^N \tau \frac{c_f}{8} \|e_{pc}^{h,n+1}\|_{0,\mathcal{C}}^2\notag\\
&\qquad+\frac{L_w}{4}\|[\delta e_{\mathbf{w}}^{h,N+1}]\cdot\mathbf{n}^+\|_{0,\mathcal{C}}^2+\frac{c_f}{4}\|e_{pc}^{h,N+1}\|^2_{0,\mathcal{C}}+\frac{c_f}{4}\|e_{pc}^{h,1}\|^2_{0,\mathcal{C}}+(\frac{L_w}{4}+\frac{G}{2})\|[\delta e_{\mathbf{w}}^{h,1}]\cdot\mathbf{n}^+\|_{0,\mathcal{C}}^2\notag\\
&\qquad+\sum_{n=1}^{N}\frac{G}{2}\|\epsilon(\delta e_{\mathbf{w}}^{h,n+1})\|_0^2+\sum_{n=1}^{N}\frac{c_f}{4}\|\delta e_{pc}^{h,n+1}\|^2_{0,\mathcal{C}}+\sum_{n=1}^{N}\tau \frac{L_w}{4}
\|[\delta e_{\mathbf w}^{h,n+1}]\cdot\mathbf n^+\|_{0,\mathcal C}^2
\notag\\
&\qquad+\frac{\epsilon_3\tau}{3}\sum_{n=1}^{N}\|e_{\xi}^{h,n+1}\|_0^2+ \frac{\epsilon_4\tau}{3}\sum_{n=1}^{N}\|e_{\xi}^{h,n+1}\|_0^2+\sum_{n=1}^{N}\frac{\epsilon_5\tau}{3}\|e_{\xi}^{h,n+1}\|^2_0\notag\\
&\qquad+\tilde{C}\tau^2\left(\int_{0}^{T}\|\partial_{tt}p\|^2_{L^{2}(\Omega\setminus\mathcal{C})}\mathrm{d}s+\int_{0}^{T}\|\partial_{tt}\xi\|^2_{L^{2}(\Omega\setminus\mathcal{C})}\mathrm{d}s+\int_0^T \|\partial_{tt} p_c\|_{L^2(\mathcal{C})}^{2}\mathrm{d}s\right)\notag\\
&\qquad+\tilde{C}h^{2l+2}\int_{0}^{T}\|\partial_tp\|^2_{H^{l+1}(\Omega\setminus\mathcal{C})}\mathrm{d}s+\tilde{C}h^{2k}\int_{0}^{T}\|\partial_t\xi\|^2_{H^{k}(\Omega\setminus\mathcal{C})}\mathrm{d}s+h^{2k}\int_0^T
\|\partial_t\mathbf w\|^2_{H^{k+1}(\Omega\setminus\mathcal C)}\mathrm{d}s \notag\\
&\qquad+\tilde{C}h^{2l+2} \int_0^T \|\partial_t p_c\|_{H^{l+1}(\mathcal{C})}^{2} \mathrm{d}s+\tilde{C} h^{2k+1}\int_0^T
\|\partial_t\mathbf{w}\|^2_{H^{k+1}(\Omega\setminus\mathcal{C})}\,\mathrm{d}s+\tilde{C}\tau^2\int_0^T
\|\partial_{tt}\mathbf{w}\|^2_{H^1(\Omega\setminus\mathcal{C})}\mathrm{d}s.
\end{align*}
After absorbing the corresponding terms on the right-hand side, we properly choose the auxiliary parameters $\epsilon_1$, $\epsilon_2$, $\epsilon_3$, $\epsilon_4$, $\epsilon_5$, and the stabilization parameter $L_w$ such that all coefficients of the terms on the left-hand side remain positive. Specifically, the parameters are selected as follows:
\begin{align*}
	&\epsilon_1=\epsilon_2=\frac{1}{\lambda}, \quad
	\epsilon_3=\epsilon_4=\epsilon_5=\frac{1}{4\lambda},
\quad L_w=\frac{4}{c_f}.
\end{align*}
Regarding the error contributed by the first time step, we obtain the following from the a priori analysis for the coupled scheme at the initial time:
\begin{align}
	&\frac{c_0}{2}\|e_p^{h,1}\|^2_{0}+\frac{c_f}{2}\|e_{pc}^{h,1}\|^2_{0,\mathcal{C}}+G\|\epsilon(e_{\mathbf{w}}^{h,1})\|_0^2 +\frac{1}{2\lambda}\|e_{\xi}^{h,1}\|_0^2+\tau\|\sqrt{\frac{\mathbf{K}}{\eta}}\nabla e_p^{h,1}\|_0^2+\tau\|\sqrt{\frac{\mathbf{K}_f}{\eta}}\overline\nabla e_{pc}^{h,1}\|_{0,\mathcal{C}}^2\notag\\
	&\le C\tau^2\left(\int_{0}^{T}\|\partial_{tt}p\|^2_{L^{2}(\Omega\setminus\mathcal{C})}\mathrm{d}s+\int_{0}^{T}\|\partial_{tt}\xi\|^2_{L^{2}(\Omega\setminus\mathcal{C})}\mathrm{d}s+\int_0^T \|\partial_{tt} p_c\|_{L^2(\mathcal{C})}^{2}\mathrm{d}s+\int_0^T
	\|\partial_{tt}\mathbf{w}\|^2_{H^1(\Omega\setminus\mathcal{C})}\mathrm{d}s\right)\notag\\
	&\qquad+Ch^{2l+2}\int_{0}^{T}\|\partial_tp\|^2_{H^{l+1}(\Omega\setminus\mathcal{C})}\mathrm{d}s+Ch^{2k}\int_{0}^{T}\|\partial_t\xi\|^2_{H^{k}(\Omega\setminus\mathcal{C})}\mathrm{d}s+Ch^{2k}\int_0^T
	\|\partial_t\mathbf w\|^2_{H^{k+1}(\Omega\setminus\mathcal C)}\mathrm{d}s \notag\\
	&\qquad+Ch^{2l+2} \int_0^T \|\partial_t p_c\|_{H^{l+1}(\mathcal{C})}^{2} \mathrm{d}s+C h^{2k+1}\int_0^T
	\|\partial_t\mathbf{w}\|^2_{H^{k+1}(\Omega\setminus\mathcal{C})}\,\mathrm{d}s.
\end{align}
Note that $e_{\mathbf{w}}^{h,0}=\mathbf{0}$, $e_p^{h,0}=0$, $e_{pc}^{h,0}$, and $e_{\xi}^{h,0}=0$. Thus, we have,
\begin{align}\label{total_the}
&(\frac{c_0}{4}+\frac{\alpha^2}{2\lambda})\|e_p^{h,N+1}\|^2_{0}+\frac{c_f}{4}\|e_{pc}^{h,N+1}\|^2_{0,\mathcal{C}}+G\|\epsilon(e_{\mathbf{w}}^{h,N+1})\|_0^2+\frac{1}{4\lambda}\|e_{\xi}^{h,N+1}\|_0^2+\frac{L_w}{4}\|[\delta e_{\mathbf{w}}^{h,N+1}]\cdot\mathbf{n}^+\|_{0,\mathcal{C}}^2\notag\\
&\qquad +\sum_{n=1}^{N}\left(\tau\|\sqrt{\frac{\mathbf{K}}{\eta}}\nabla e_p^{h,n+1}\|_0^2+\tau\|\sqrt{\frac{\mathbf{K}_f}{\eta}}\overline\nabla e_{pc}^{h,n+1}\|_{0,\mathcal{C}}^2\right)\notag\\
&\le \tau\sum_{n=1}^{N}\left((\frac{c_0}{4}+\frac{\alpha^2}{2\lambda})\|e_p^{h,n+1}\|^2_{0}+\frac{c_f}{4} \|e_{pc}^{h,n+1}\|_{0,\mathcal{C}}^2+\frac{1}{4\lambda}\|e_{\xi}^{h,n+1}\|_0^2+G\|\epsilon(e_{\mathbf{w}}^{h,n+1})\|_0^2+\frac{L_w}{4}\|[\delta e_{\mathbf{w}}^{h,n+1}]\cdot\mathbf{n}^+\|_{0,\mathcal{C}}^2\right)\notag\\
&\qquad +C\tau^2\left(\int_{0}^{T}\|\partial_{tt}p\|^2_{L^{2}(\Omega\setminus\mathcal{C})}\mathrm{d}s+\int_{0}^{T}\|\partial_{tt}\xi\|^2_{L^{2}(\Omega\setminus\mathcal{C})}\mathrm{d}s+\int_0^T \|\partial_{tt} p_c\|_{L^2(\mathcal{C})}^{2}\mathrm{d}s+\int_0^T
\|\partial_{tt}\mathbf{w}\|^2_{H^1(\Omega\setminus\mathcal{C})}\mathrm{d}s\right)\notag\\
&\qquad+Ch^{2l+2}\int_{0}^{T}\|\partial_tp\|^2_{H^{l+1}(\Omega\setminus\mathcal{C})}\mathrm{d}s+Ch^{2k}\int_{0}^{T}\|\partial_t\xi\|^2_{H^{k}(\Omega\setminus\mathcal{C})}\mathrm{d}s+Ch^{2k}\int_0^T
\|\partial_t\mathbf w\|^2_{H^{k+1}(\Omega\setminus\mathcal C)}\mathrm{d}s \notag\\
&\qquad+Ch^{2l+2} \int_0^T \|\partial_t p_c\|_{H^{l+1}(\mathcal{C})}^{2} \mathrm{d}s+C h^{2k+1}\int_0^T
\|\partial_t\mathbf{w}\|^2_{H^{k+1}(\Omega\setminus\mathcal{C})}\,\mathrm{d}s\notag\\
&:=C\tau^2\Phi_{\tau}^2+Ch^{2l+2}\Phi_p^2+Ch^{2k}\Phi_{\xi}^2+Ch^{2k}\Phi_{\mathbf{w}}^2+Ch^{2l+2}\Phi_{pc}^2.
\end{align}
Finally, we use the discrete Gronwall's  inequality to obtain the result of the theorem. This completes the proof. 
\end{proof}

\begin{corollary}
Under Assumption \ref{a1}. For $1\le n\le N$, let $(\mathbf{w}^{n+1},\xi^{n+1},p^{n+1},p_c^{n+1})\in\mathcal{W}\times\varPsi\times Q\times Q_{c} $ denote the solutions of \eqref{vwx}--\eqref{vwfx} at the time $t^{n+1}$, and $(\mathbf{w}^{n+1}_h,\xi^{n+1}_h,p^{n+1}_h,p_{ch}^{n+1})\in\mathcal{W}_h\times\varPsi_h\times Q_h\times Q_{ch}$ represent the solutions of the fully discrete decoupling schemes \eqref{fp}--\eqref{fpsi}. If $L_w=\frac{4}{c_f}$, we obtain the following optimal convergence result:
\begin{align}
	&\sqrt{\frac{c_0}{2}+\frac{\alpha^2}{2\lambda}}\|e_p^{N+1}\|_{0}+\sqrt{\frac{c_f}{2}}\|e_{pc}^{N+1}\|_{0,\mathcal{C}}+\sqrt{G}\|\epsilon(e_{\mathbf{w}}^{N+1})\|_0 +\sqrt{\frac{1}{2\lambda}}\|e_{\xi}^{N+1}\|_0\notag\\
	&\qquad+\left[\sum_{n=1}^{N}\left(\tau\|\sqrt{\frac{\mathbf{K}}{\eta}}\nabla e_p^{n+1}\|_0^2+\tau\|\sqrt{\frac{\mathbf{K}_f}{\eta}}\overline\nabla e_{pc}^{n+1}\|_{0,\mathcal{C}}^2\right)\right]^{1/2}\notag\\
	&\le C\tau\Phi_{\tau}+Ch^{l+1}\Phi_p+Ch^{k}\Phi_{\xi}+Ch^{k}\Phi_{\mathbf{w}}+Ch^{l+1}\Phi_{pc},
\end{align}
where $C$ is a generic positive constant independent of $h$ and $\tau$, $\Phi_i (i=\tau,p,\xi,\mathbf{w},pc)$ is defined in \eqref{total_the}.
\end{corollary}
\begin{proof}
	The desired result follows from Theorem \ref{th4} and the triangle inequality.
\end{proof}
\section{Numerical experiments}\label{s5}
In this section, we will conduct numerical experiments in two dimensions to validate the theoretical predictions discussed in Section \ref{sf}. All numerical experiments in this work are conducted using the open-source software FreeFem++, with solver = "SPARSESOLVER" employed for the resulting linear systems \cite{MR3043640}. 
\subsection{Example 1: Verification of temporal convergence}\label{e1}
Test for the convergence rates in time, we solve the numerical scheme \eqref{fp}--\eqref{fpsi} with a relatively small $\Delta t_{\mathrm{init}}=10^{-6}$ to obtain $\mathbf{w}_h^1$, $\xi_h^1$, $p_h^1$ and $p_{ch}^1$. Let $\Omega=[-1,1]$, $\mathcal{C}=\{x=0\}\times [-1,1]$, and the final time is $T = 1+10^{-6}$.  The mesh partitioning for Examples I is shown in Figure \ref{mesh1}.
\begin{figure}[ht]
	\centering
	\subfigure{
		\begin{minipage}[t]{0.5\textwidth}
			\centering
			\includegraphics[scale=0.1]{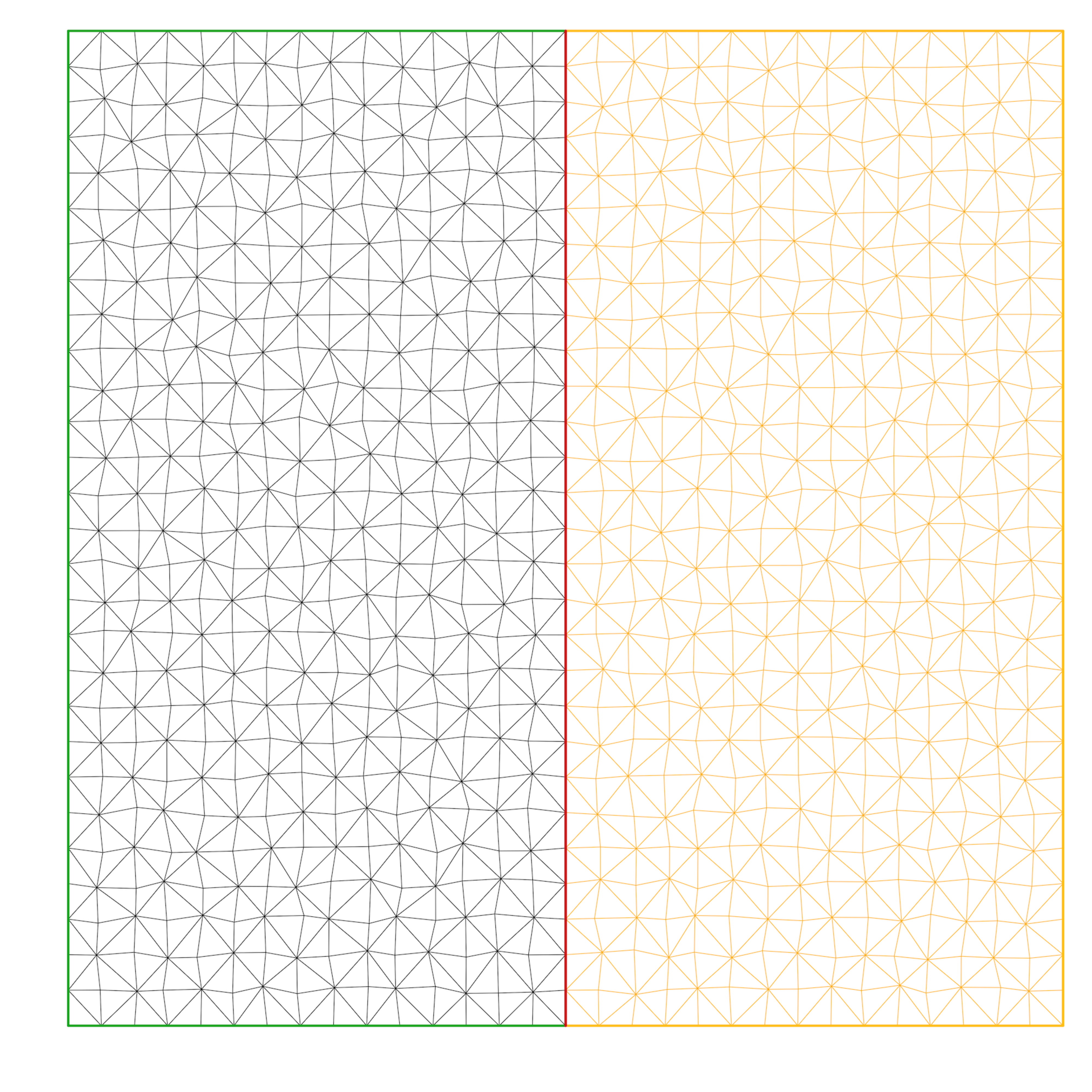}
		\end{minipage}
	}
	\centering
	\caption{Mesh configuration and fracture location setup.}
	\label{mesh1}
\end{figure}
We choose the body force $f$, the
source/sink term $q$, and the injection $q_f$, the initial conditions and boundary data on such.
The exact solution is as follows:
\begin{align}
	\begin{cases}
		w_1=100\sin(t)\left(\dfrac{x}{\lambda}+y\right),\\
		w_2=100\sin(t)\left(x+\dfrac{y}{\lambda}\right),\\
		p=\sin(t)(x+y),\\
		\xi=\alpha\sin(t)(x+y)-200\sin(t),\\
		p_c=\sin(t)y.
	\end{cases}
\end{align}
The following parameters $G=1$, $\lambda=1$, $c_0=1$, $c_f=1$, $\alpha=1$, $\mathbf{K}/\eta=1$, $\mathbf{K}_f/\eta=1$, $L_w=4$. We fix the spatial discretization parameters as $h=1/32$, $k=2$, and $l=1$, and then refine the time step size $\Delta t$. We compute $L^2$
norm of the discretization error at final time $T = 1 + 10^{-6}$. 
\begin{table}[h!]
	\centering
	\caption{Temporal convergence rates and $L^2$ errors for variables $\mathbf{w}$, $\xi$, $p$, and $p_c$ with $k=2$, $l=1$.}
	\setlength{\tabcolsep}{10pt}
		\renewcommand{\arraystretch}{1.18}
	\begin{tabular}{ccccccccc}
		\toprule
		$1/\Delta t$ & $\|e_{\mathbf{w}}^N\|_0$ & $\text{Rate}$ & $\|e_{p}^N\|_0$ & $\text{Rate}$ & $\|e_{pc}^N\|_{0,c}$ & $\text{Rate}$ & $\|e_{\xi}^N\|_0$ & $\text{Rate}$ \\	\midrule
		2   & 2.24e-02 & -    & 7.50e-02 & -    & 1.39e-01 & -    & 5.48e-02 & - \\
		4   & 1.08e-02 & 1.05 & 3.61e-02 & 1.05 & 5.27e-02 & 1.40 & 2.64e-02 & 1.05 \\
		8   & 5.30e-03 & 1.03 & 1.77e-02 & 1.03 & 2.24e-02 & 1.23 & 1.29e-02 & 1.03 \\
		16  & 2.60e-03 & 1.02 & 8.70e-03 & 1.02 & 1.02e-02 & 1.13 & 6.40e-03 & 1.02 \\
		32  & 1.30e-03 & 1.01 & 4.30e-03 & 1.01 & 4.90e-03 & 1.07 & 3.20e-03 & 1.01 \\
		64  & 6.00e-04 & 1.00 & 2.20e-03 & 1.00 & 2.40e-03 & 1.03 & 1.60e-03 & 1.00 \\ 	\bottomrule
	\end{tabular}
	\label{tab:lu_lp_lpc_lz_rate}
\end{table}

Table \ref{tab:lu_lp_lpc_lz_rate} reports the temporal errors and convergence rates for the displacement $\mathbf{w}$, matrix pressure $p$, fracture pressure $p_c$, and auxiliary variable $\xi$. As the time step is refined, the errors of all variables decrease monotonically. The convergence rates of $\mathbf{w}$, $p$, $p_c$ and $\xi$ are close to first order for sufficiently small time steps. These results indicate that the proposed time-discretization scheme exhibits a stable temporal convergence behavior and is broadly consistent with first-order accuracy in time.

\subsection{Example 2: Comparison of computational efficiency}
To compare the computational efficiency of the fully coupled and decoupled schemes, we adopt the same mesh configuration, exact solutions, and model parameters as those used in Example \ref{e1}. In the fully coupled scheme, all unknowns are solved simultaneously at each time level, whereas in the decoupled scheme, the flow variables and the mechanical variables are solved sequentially with the coupling terms treated in a time-decoupled discrete scheme. The same spatial mesh, model parameters, and linear solver are used for both schemes, so that the difference in CPU time mainly reflects the effect of the temporal coupling strategy.
\begin{table*}[h!]
	\centering
	\caption{Comparison of CPU time between the fully coupled and decoupled schemes.}
	\label{tab:cpu_time_comparison}
	\setlength{\tabcolsep}{8pt}
	\renewcommand{\arraystretch}{1.18}
	\begin{tabular}{llcccc}
		\toprule
		$1/\Delta t$ & Scheme
		& $\|e_{\mathbf{w}}^N\|_{0}$ \& $\|e_p^N\|_{0}$ \& $\|e_{p_c}^N\|_{0,c}$
		& CPU time (s)
		& Time reduction \\
		\midrule
		\multirow{2}{*}{4}
		& Fully Coupled & 3.07e+00 \& 9.67e+00 \& 5.27e-02 & 2.50 & -- \\
		& Decoupled     & 2.58e+00 \& 8.47e+00 \& 6.04e-02 & 2.08 & 16.94\% \\
		\midrule
		\multirow{2}{*}{8}
		& Fully Coupled & 1.32e+00 \& 4.17e+00 \& 2.24e-02 & 4.92 & -- \\
		& Decoupled     & 1.43e+00 \& 4.70e+00 \& 2.99e-02 & 3.36 & 31.69\% \\
		\midrule
		\multirow{2}{*}{16}
		& Fully Coupled & 6.09e-01 \& 1.92e+00 \& 1.02e-02 & 10.23 & -- \\
		& Decoupled     & 7.51e-01 \& 2.46e+00 \& 1.76e-02 & 6.09 & 40.49\% \\
		\midrule
		\multirow{2}{*}{32}
		& Fully Coupled & 2.92e-01 \& 9.19e-01 \& 4.90e-03 & 20.11 & -- \\
		& Decoupled     & 3.82e-01 \& 1.24e+00 \& 1.22e-02 & 11.64 & 42.11\% \\
		\midrule
		\multirow{2}{*}{64}
		& Fully Coupled & 1.43e-01 \& 4.50e-01 \& 2.40e-03 & 40.19 & -- \\
		& Decoupled     & 1.91e-01 \& 6.18e-01 \& 9.70e-03 & 22.62 & 43.70\% \\
		\bottomrule
	\end{tabular}
\end{table*}

The comparison results are reported in Table \ref{tab:cpu_time_comparison}. It can be observed that the fully coupled scheme gives smaller errors for all variables, while the decoupled scheme requires much less computational time. In particular, when the time step is sufficiently small, the decoupled scheme achieves a time reduction of more than 40\%. These results indicate that the decoupled scheme provides a more computationally efficient alternative while preserving the expected accuracy trend.
\subsection{Example 3: Verification of temporal and spatial convergence}
We then investigate the spatial convergence behavior in the $L^2$ and $H^1$ norms for the displacement $\mathbf{w}$, matrix pressure $p$, fracture pressure $p_c$, and total pressure $\xi$ at the final time $T=1$.
In this example, we use the same mesh configuration as that in Example \ref{e1}, as illustrated in Figure \ref{mesh1}. The exact solutions are prescribed as follows:
\begin{align}
	\begin{cases}
		w_1=e^{-t}\left((\sin(2\pi y)(\cos(2\pi x)-1) + \frac{\sin(\pi x)\sin(\pi y)}{\lambda+G})\right),\\
		w_2=e^{-t}\left((\sin(2\pi x)(1-\cos(2\pi y)) + \frac{\sin(\pi x)\sin(\pi y)}{\lambda+G})\right),\\
		p=e^{-t}\sin(\pi x)\sin(\pi y),\\
		p_c=e^{-t}(1+\cos(\pi y)).
	\end{cases}
\end{align}

Two sets of physical parameters are considered in this test. The first set corresponds to a standard parameter regime and is used to examine the spatial convergence behavior under a regular configuration. The second set is designed to test the robustness of the method in a nearly incompressible regime, where the Lam\'e parameter \(\lambda\) is chosen to be sufficiently large. For each parameter set, we consider two pairs of finite element spaces: \(k=2,\ l=1\) and \(k=3,\ l=2\). This comparison allows us to assess the spatial accuracy of the proposed discretization under different polynomial degrees and to examine whether the method remains stable without a noticeable locking effect in the nearly incompressible case.
In this example, we select fixed physical parameters: $G=1,$ $\alpha=1$. To verify the optimal error estimates, we set the time step to $\Delta t=h^k$, and test the following two cases:
\begin{align*}
\begin{cases}
\text{Case\,1:}\quad \lambda=1, \quad c_0=1, \quad c_f=1, \quad\mathbf{K}/\eta=1, \quad\mathbf{K}_f/\eta=1, \quad L_w=4,\\
\text{Case\,2:} \quad \lambda=10^8, \quad c_0=10^{-5}, \quad c_f=0.1, \quad\mathbf{K}/\eta=5\times10^{-4}, \quad\mathbf{K}_f/\eta=5\times10^{-4}, \quad L_w=40.\\
\end{cases} 
\end{align*}
\begin{table}[h!]
	\centering
	\caption{Spatial convergence $L^2$ errors and rates for Case 1}
		\label{Case1}
	\setlength{\tabcolsep}{10pt}
		\renewcommand{\arraystretch}{1.18}
	\begin{tabular}{ccccccccc}
	\toprule
		$1/h$ & $\|e_{\mathbf{w}}^N\|_0$ & Rate & $\|e_p^N\|_0$ & Rate & $\|e_{pc}^N\|_{0,\mathcal{C}}$ & Rate & $\|e_{\xi}^N\|_0$ & Rate \\
		\midrule
		\multicolumn{9}{c}{\text{k=2,\, l=1}} \\
		\midrule
		4  & 2.84e-01 & - & 2.41e-01 & - & 4.70e-02 & - & 4.15e-01 & - \\
		8  & 3.72e-02 & 2.93 & 6.02e-02 & 2.00 & 1.14e-02 & 2.05 & 1.24e-01 & 1.74 \\
		16 & 4.69e-03 & 2.99 & 1.39e-02 & 2.11 & 2.73e-03 & 2.06 & 2.12e-02 & 2.55 \\
		32 & 5.99e-04 & 2.97 & 3.08e-03 & 2.18 & 6.65e-04 & 2.04 & 4.60e-03 & 2.21 \\
		\midrule
		\multicolumn{9}{c}{\text{k=3,\, l=2}} \\
		\midrule
		4  & 7.33e-02 & - & 2.33e-02 & - & 1.62e-03 & - & 1.84e-01 & - \\
		8  & 5.84e-03 & 3.65 & 2.61e-03 & 3.16 & 1.87e-04 & 3.12 & 1.56e-02 & 3.55 \\
		16 & 3.85e-04 & 3.92 & 3.50e-04 & 2.90 & 2.33e-05 & 3.00 & 1.90e-03 & 3.04 \\
		32 & 2.24e-05 & 4.11 & 4.56e-05 & 2.94 & 2.94e-06 & 2.98 & 1.92e-04 & 3.31 \\
		\bottomrule
	\end{tabular}
\end{table}
\begin{table}[h!]
	\centering
	\caption{Spatial convergence $L^2$ errors and rates for Case 2}
		\setlength{\tabcolsep}{10pt}
	\label{Case2}
		\renewcommand{\arraystretch}{1.18}
	\begin{tabular}{ccccccccc}
	\toprule
		$1/h$ & $\|e_{\mathbf{w}}^N\|_0$ & Rate & $\|e_p^N\|_0$ & Rate & $\|e_{pc}^N\|_{0,C}$ & Rate & $\|e_{\xi}^N\|_0$ & Rate \\
		\midrule
		\multicolumn{9}{c}{\text{k=2,\, l=1}} \\
		\midrule
		4  & 2.79e-01 & -    & 2.41e-01 & -    & 4.88e-02 & -    & 1.26e+00 & - \\
		8  & 3.71e-02 & 2.91 & 6.04e-02 & 2.00 & 1.19e-02 & 2.04 & 4.06e-01 & 1.63 \\
		16 & 4.54e-03 & 3.03 & 1.42e-02 & 2.09 & 2.85e-03 & 2.06 & 5.19e-02 & 2.97 \\
		32 & 5.52e-04 & 3.04 & 3.35e-03 & 2.08 & 6.93e-04 & 2.04 & 9.68e-03 & 2.42 \\
		\midrule
		\multicolumn{9}{c}{\text{k=3,\, l=2}} \\
		\midrule
		4  & 7.74e-02 & -    & 2.33e-02 & -    & 1.63e-03 & -    & 8.47e-01 & - \\
		8  & 5.87e-03 & 3.72 & 2.61e-03 & 3.16 & 1.86e-04 & 3.12 & 6.58e-02 & 3.69 \\
		16 & 3.84e-04 & 3.93 & 3.49e-04 & 2.90 & 2.33e-05 & 3.00 & 8.06e-03 & 3.03 \\
		32 & 2.22e-05 & 4.11 & 4.42e-05 & 2.98 & 2.94e-06 & 2.98 & 7.73e-04 & 3.38 \\
			\bottomrule
	\end{tabular}
\end{table}
\begin{table}[h!]
	\centering
	\caption{Spatial convergence $H^1$ errors and rates for Case 1}
	\label{Case1-h1}
		\setlength{\tabcolsep}{10pt}
	\renewcommand{\arraystretch}{1.18}
	\begin{tabular}{ccccccccc}
	\toprule
		$1/h$ & $\|e_{\mathbf{w}}^N\|_1$ & Rate & $\|e_p^N\|_1$ & Rate & $\|e_{pc}^N\|_{1,\mathcal{C}}$ & Rate & $\|e_{\xi}^N\|_1$ & Rate \\
		\midrule
		\multicolumn{9}{c}{\text{k=2,\, l=1}} \\
	\midrule
		4  & 3.40e+00 & -    & 1.62e+01 & -    & 4.78e-01 & -    & 4.81e+00 & - \\
		8  & 9.01e-01 & 1.91 & 8.05e+00 & 1.01 & 2.41e-01 & 0.99 & 2.70e+00 & 0.83 \\
		16 & 2.30e-01 & 1.97 & 3.93e+00 & 1.03 & 1.19e-01 & 1.01 & 1.07e+00 & 1.34 \\
		32 & 5.67e-02 & 2.02 & 1.92e+00 & 1.03 & 5.91e-02 & 1.01 & 5.07e-01 & 1.08 \\
		\midrule
		\multicolumn{9}{c}{\text{k=3,\, l=2}} \\
	\midrule
		4  & 1.03e+00 & -    & 3.97e+00 & -    & 4.43e-02 & -    & 3.25e+00 & - \\
		8  & 1.13e-01 & 3.19 & 9.16e-01 & 2.12 & 1.13e-02 & 1.97 & 6.21e-01 & 2.39 \\
		16 & 1.44e-02 & 2.96 & 2.42e-01 & 1.92 & 2.89e-03 & 1.97 & 1.45e-01 & 2.10 \\
		32 & 1.66e-03 & 3.12 & 5.94e-02 & 2.03 & 7.33e-04 & 1.98 & 2.89e-02 & 2.33 \\
	\bottomrule
	\end{tabular}
\end{table}
\begin{table}[htbp]
	\centering
	\caption{Spatial convergence $H^1$ errors and rates for Case 2}
	\label{Case2-h1}
		\setlength{\tabcolsep}{10pt}
	\renewcommand{\arraystretch}{1.18}
	\begin{tabular}{ccccccccc}
\toprule
		$1/h$ & $\|e_{\mathbf{w}}^N\|_1$ & Rate & $\|e_p^N\|_1$ & Rate & $\|e_{pc}^N\|_{1,c}$ & Rate & $\|e_{\xi}^N\|_1$ & Rate \\
	\midrule
		\multicolumn{9}{c}{\text{k=2,\, l=1}} \\
		\midrule
		4  & 3.42e+00 & -    & 3.62e-01 & -    & 4.79e-01 & -    & 1.37e+01 & - \\
		8  & 9.14e-01 & 1.90 & 1.80e-01 & 1.01 & 2.41e-01 & 0.99 & 8.05e+00 & 0.77 \\
		16 & 2.30e-01 & 1.99 & 8.79e-02 & 1.03 & 1.19e-01 & 1.01 & 2.47e+00 & 1.71 \\
		32 & 5.66e-02 & 2.02 & 4.29e-02 & 1.03 & 5.91e-02 & 1.01 & 1.05e+00 & 1.23 \\
		\midrule
			\multicolumn{9}{c}{\text{k=3,\, l=2}} \\
		\midrule
		4  & 1.03e+00 & -    & 8.88e-02 & -    & 4.43e-02 & -    & 1.48e+01 & - \\
		8  & 1.14e-01 & 3.17 & 2.05e-02 & 2.12 & 1.13e-02 & 1.97 & 2.52e+00 & 2.56 \\
		16 & 1.48e-02 & 2.95 & 5.41e-03 & 1.92 & 2.89e-03 & 1.97 & 5.80e-01 & 2.12 \\
		32 & 1.68e-03 & 3.13 & 1.33e-03 & 2.03 & 7.33e-04 & 1.98 & 1.10e-01 & 2.40 \\
\bottomrule
	\end{tabular}
\end{table}
As shown in Tables \ref{Case1}--\ref{Case2-h1}, the spatial tests demonstrate that the finite element combinations corresponding to $k=2,\ l=1$ and $k=3,\ l=2$ robustly achieve optimal convergence rates. In particular, these results indicate that the proposed decoupled algorithm remains effective even in extreme cases, such as those involving volumetric locking under nearly incompressible physical conditions with $\lambda = 10^8$.

\subsection{Example 4: Numerical test on a complex fracture network}
In this example, we evaluate the performance of the proposed splitting-based method on a more complex fracture configuration. The computational domain is defined as $\Omega=(0,1)^2$. The computational mesh and the fracture configuration are illustrated in Figure \ref{mesh2}. More precisely, the fracture geometry is prescribed as follows:
\begin{align*}
&\mathcal{C}_{1L}: (0.5,0.5)\rightarrow(0.125,0.5), \qquad \mathcal{C}_{1R}: (0.5,0.5)\rightarrow(0.875,0.5),\\
&\mathcal{C}_{2}: (0.5,0.5)\rightarrow(0.5,0.875),
\quad J=(0.5,0.5).
\end{align*}
Thus, $J$ is the common junction of the three fracture branches. The source terms and boundary data are determined by the following manufactured exact solution:
\begin{align}
	\begin{cases}
		w_1=0.003\,e^{-t}\sin(\pi x)\sin(\pi y),\\
		w_2=0.002\,e^{-t}\sin(\pi x)\sin(\pi y),\\
		p=e^{-t}\left(1+0.15\sin(\pi x)\sin(\pi y)
		+0.05\cos(2\pi x)\sin(\pi y)\right),\\
		p_c=e^{-t}\left(1+0.1\cos(\pi s/0.375)
		+0.03\sin(2\pi x)\sin(\pi y)\right),
	\end{cases}
\end{align}
where $s$ is the local arc-length coordinate measured from the junction $J$ toward the outer fracture tip.

The model parameters are chosen as
$c_0=0.4$, $ c_f=0.25$, $\mathbf{K}/\eta=0.06$,    $\mathbf{K}_f/\eta=0.03$,
$G=1.0$, $\lambda=2.0$, $\alpha=0.7$, $L_w=16$.
 The remaining source terms in the matrix pressure equation, fracture pressure equation, and elasticity equation are also obtained by substituting the manufactured solutions into the corresponding governing equations. In the computations, uniform grids with $h=1/64$, $k=2$, and $l=1$ are used, and the time step is selected by $\Delta t=2h^2.$

\begin{figure}[ht!]
	\centering
	\subfigure{
		\begin{minipage}[t]{0.9\textwidth}
			\centering
			\includegraphics[scale=0.35]{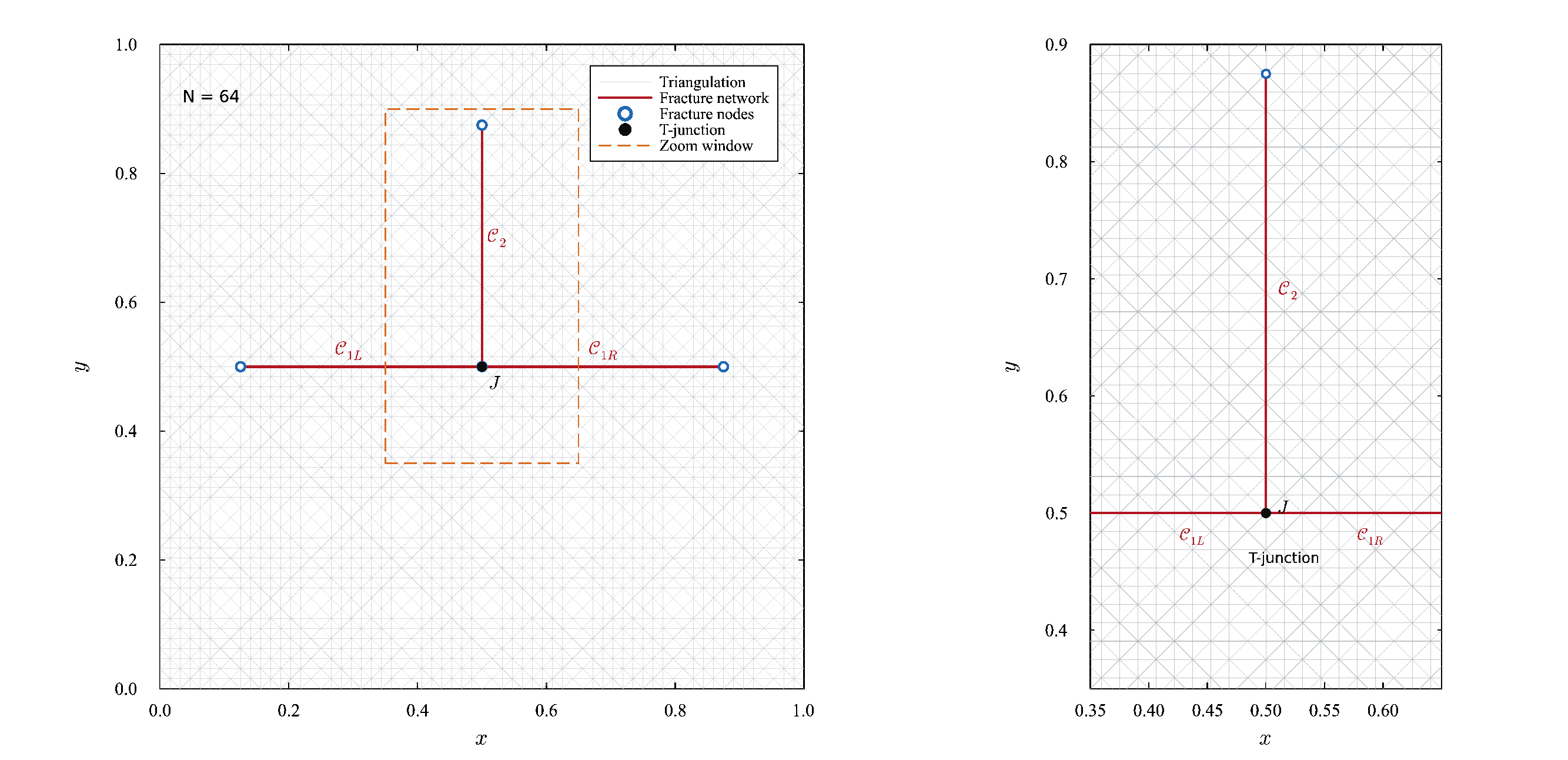}
		\end{minipage}
	}
	\centering
	\caption{Computational mesh and fracture configuration for the complex fracture-network example: (a) Global view of the structured mesh and fracture network; (b) Local view near the T-shaped fracture junction.}
	\label{mesh2}
\end{figure}

\begin{figure}[h!]
	\centering
	\subfigure{
		\begin{minipage}[t]{0.9\textwidth}
			\centering
			\includegraphics[scale=0.25]{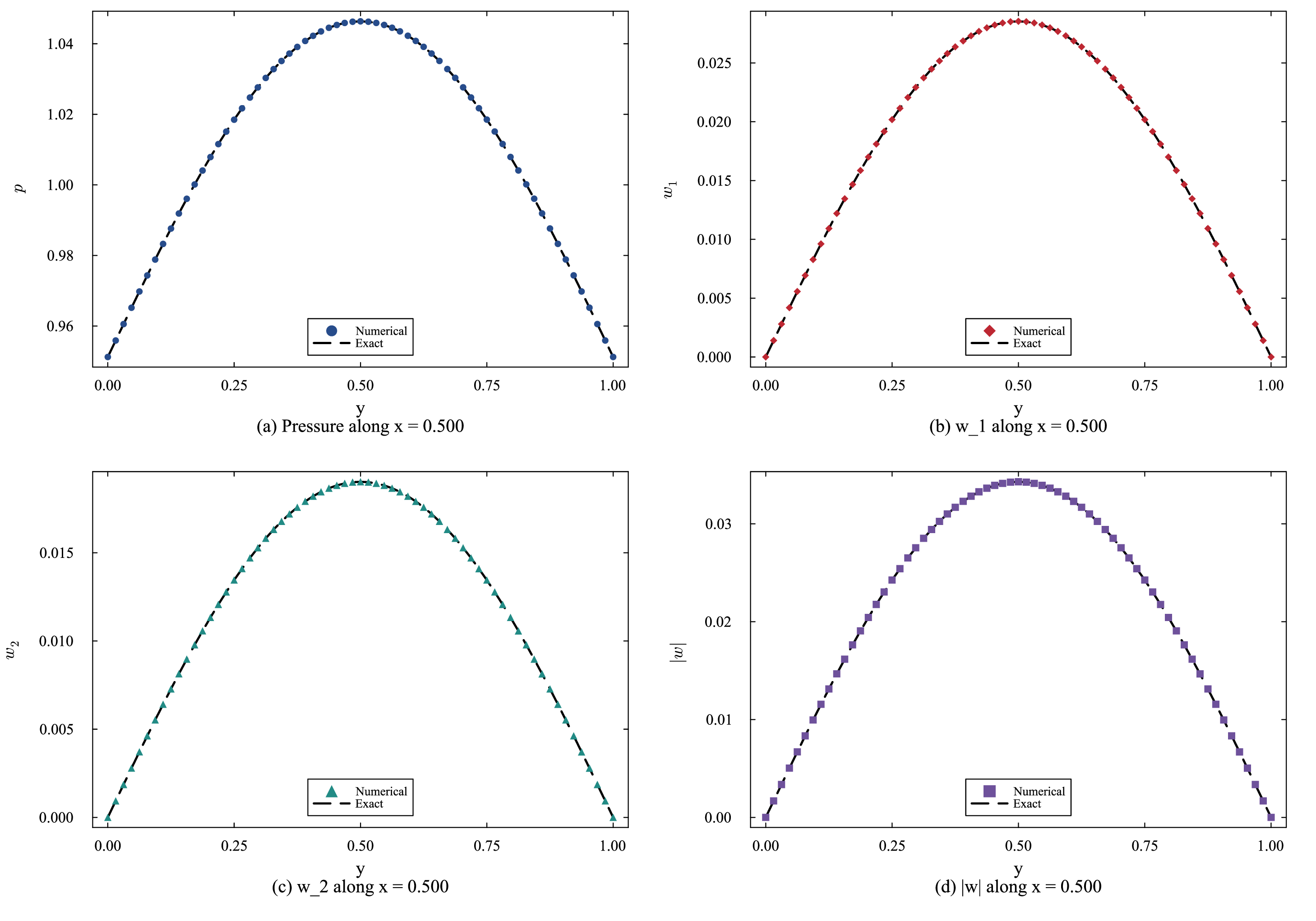}
		\end{minipage}
	}
	\centering
	\caption{Line cuts of the matrix pressure $p$, displacement components $w_1$ and $w_2$, and displacement magnitude $|\mathbf{w}|$ along $x=0.5$.}
	\label{x=0.5}
\end{figure}

\begin{figure}[h!]
	\centering
	\subfigure{
		\begin{minipage}[t]{1.0\textwidth}
			\centering
			\includegraphics[scale=0.25]{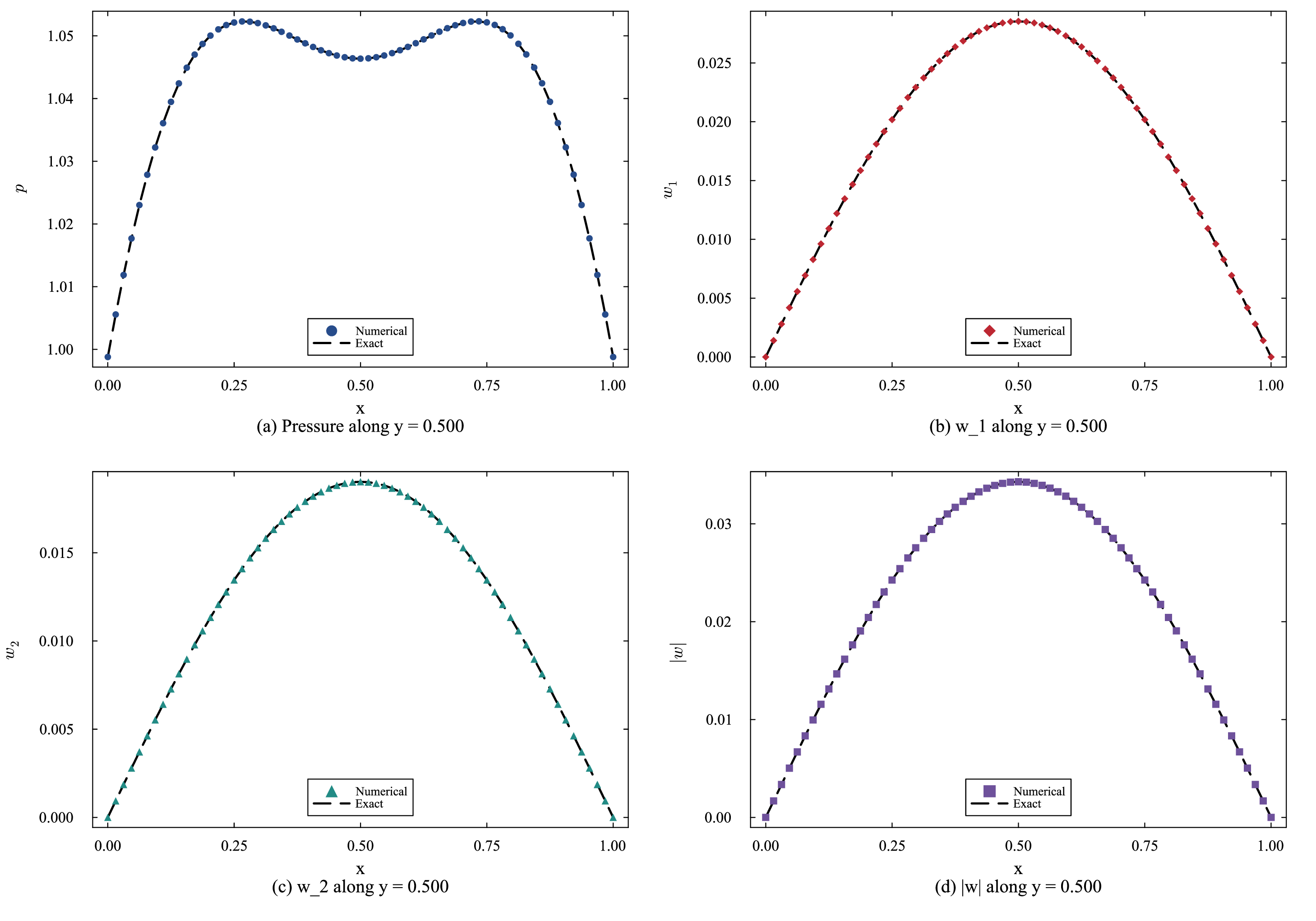}
		\end{minipage}
	}
	\centering
	\caption{Line cuts of the matrix pressure $p$, displacement components $w_1$ and $w_2$, and displacement magnitude $|\mathbf{w}|$ along $y=0.5$.}
	\label{y=0.5}
\end{figure}

\begin{figure}[htbp]
	\centering
	\begin{minipage}[b]{0.38\textwidth}
		\centering
		\includegraphics[width=\textwidth]{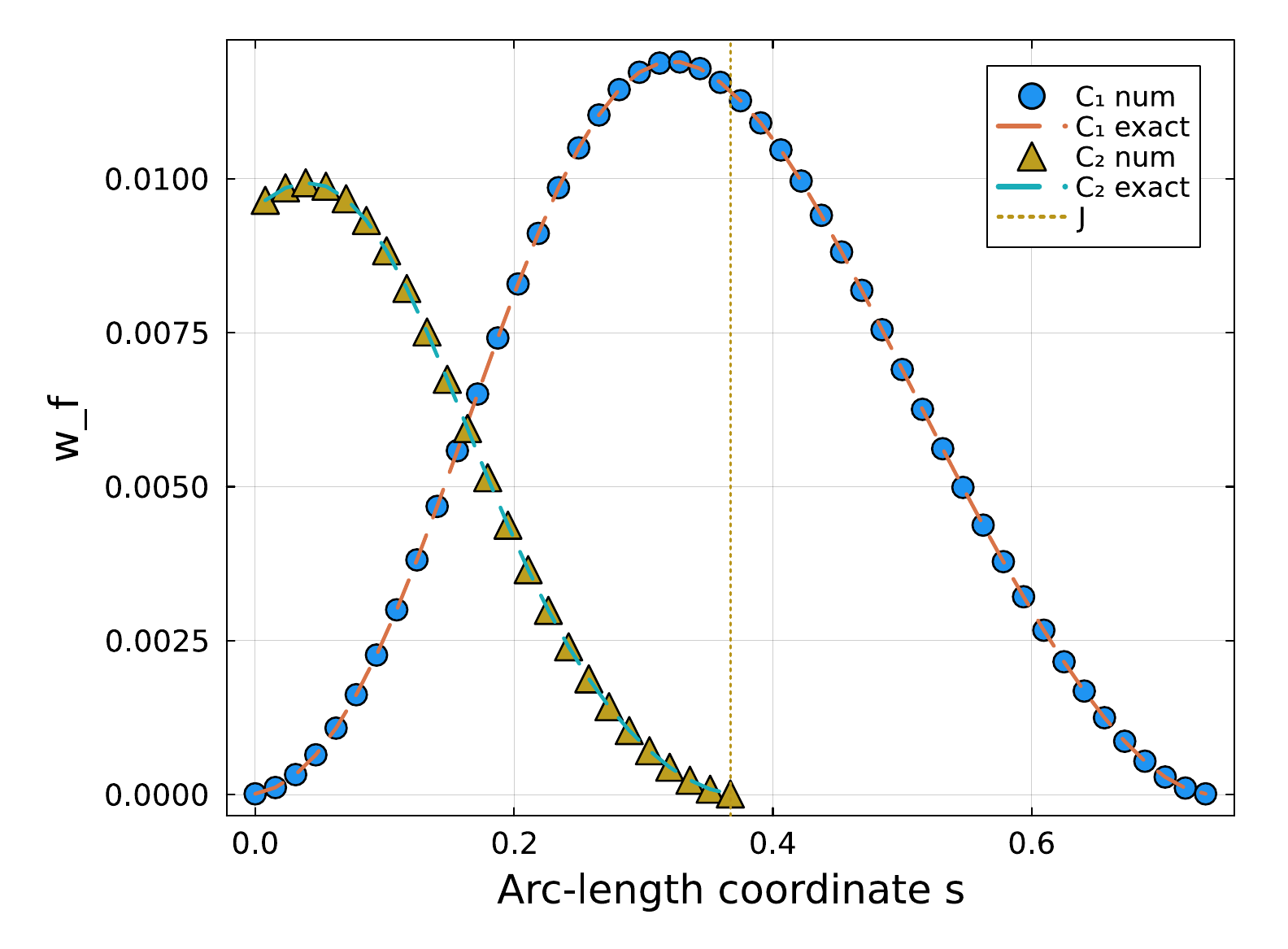}
		{\small (a) Fracture aperture.}
	\end{minipage}
	\hspace{0.03\textwidth}
	\begin{minipage}[b]{0.38\textwidth}
		\centering
		\includegraphics[width=\textwidth]{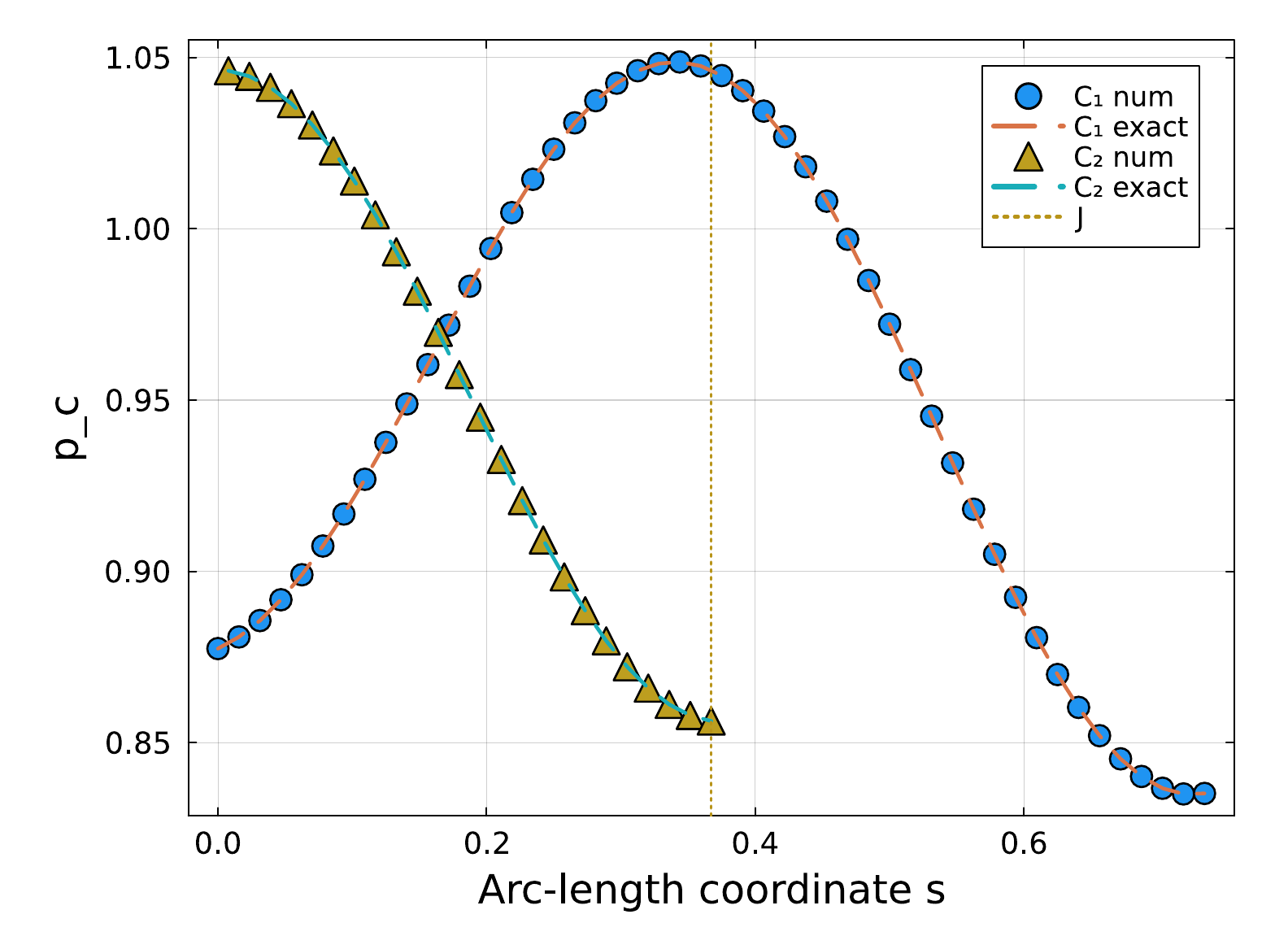}
		{\small (b) Fracture pressure.}
	\end{minipage}
	\caption{Fracture aperture and pressure along $\mathcal{C}_1$ and $\mathcal{C}_2$. }
	\label{Fracture-line}
\end{figure}

\begin{figure}[h!]
	\centering
	\subfigure{
		\includegraphics[width=0.48\linewidth]{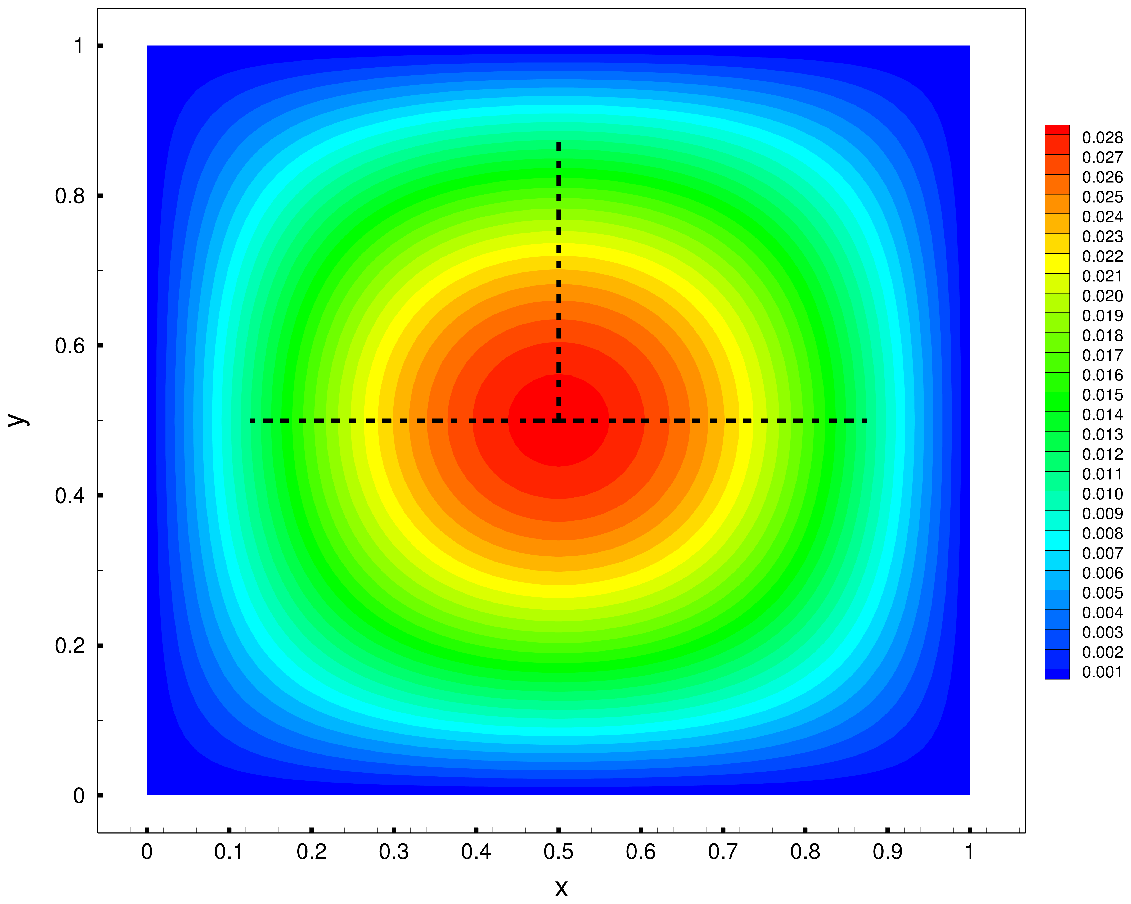}
	}
	\hfill
	\subfigure{
		\includegraphics[width=0.48\linewidth]{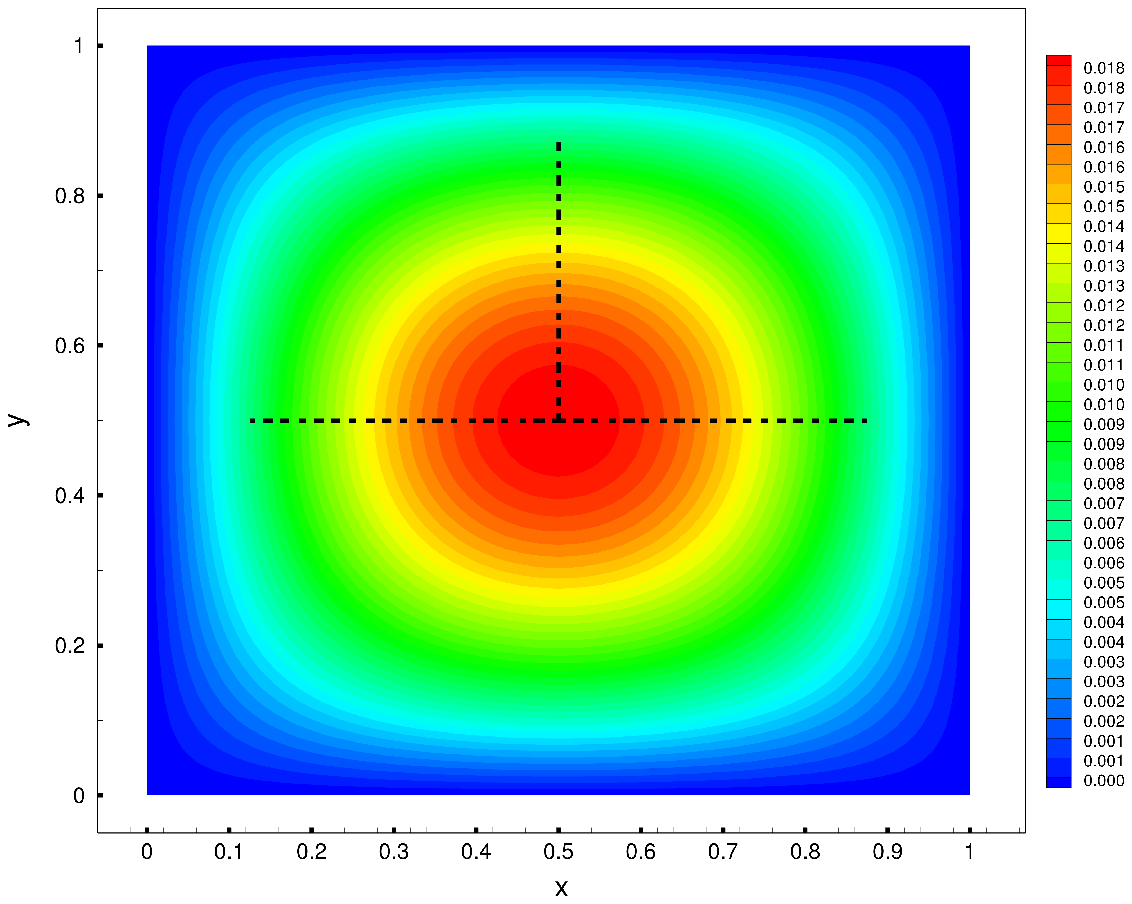}
	}
	\caption{Displacement component contours: (a) Horizontal displacement component \(w_1\); (b) Vertical displacement component \(w_2\).}
	\label{displacement_contour}
\end{figure}

\begin{figure}[h!]
	\centering
	\subfigure{
		\includegraphics[width=0.48\linewidth]{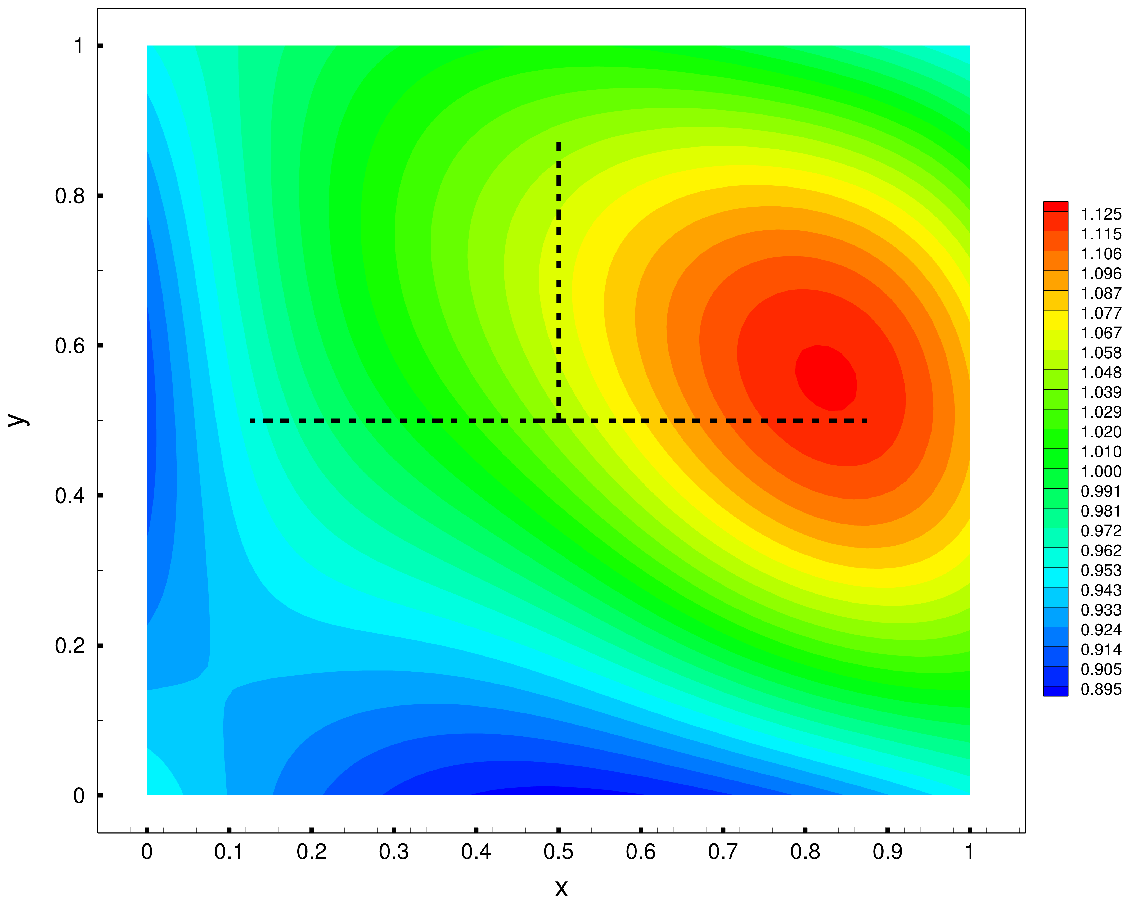}
	}
	\hfill
	\subfigure{
		\includegraphics[width=0.48\linewidth]{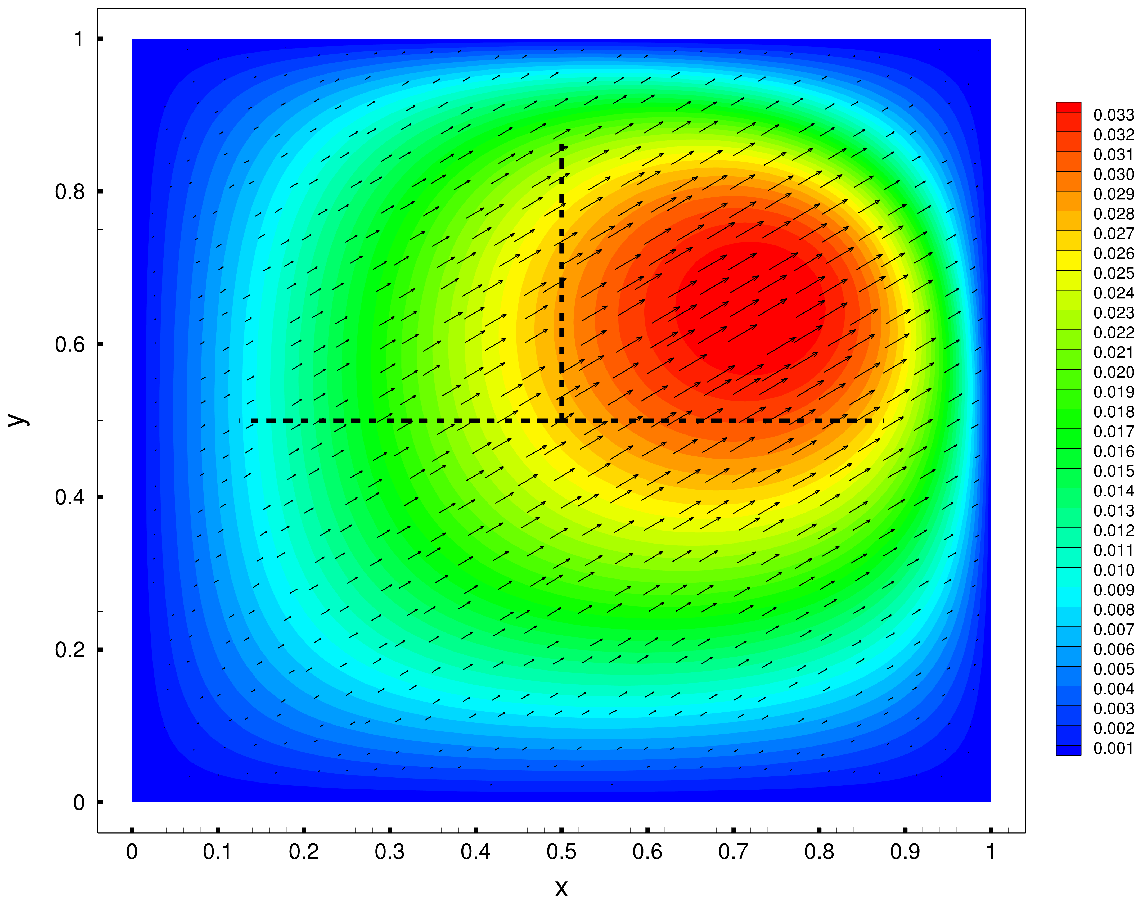}	}
	\caption{Numerical results of the total pressure and displacement fields: (a) Total pressure \(\xi\); (b) Deformed configuration with displacement vectors.}
\end{figure}
\begin{figure}[h!]
	\centering
	\subfigure{
		\includegraphics[width=0.48\linewidth]{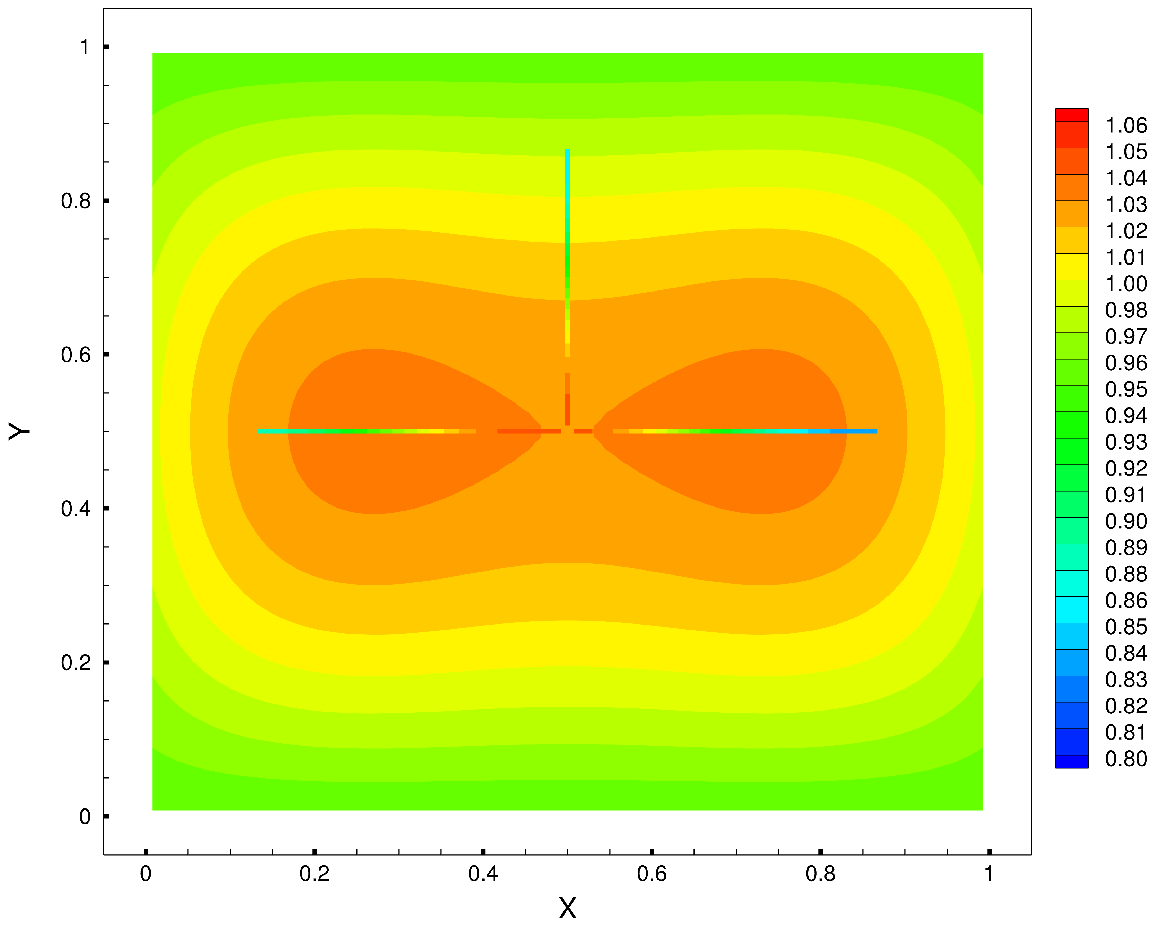}
	}
	\hfill
	\subfigure{
		\includegraphics[width=0.48\linewidth]{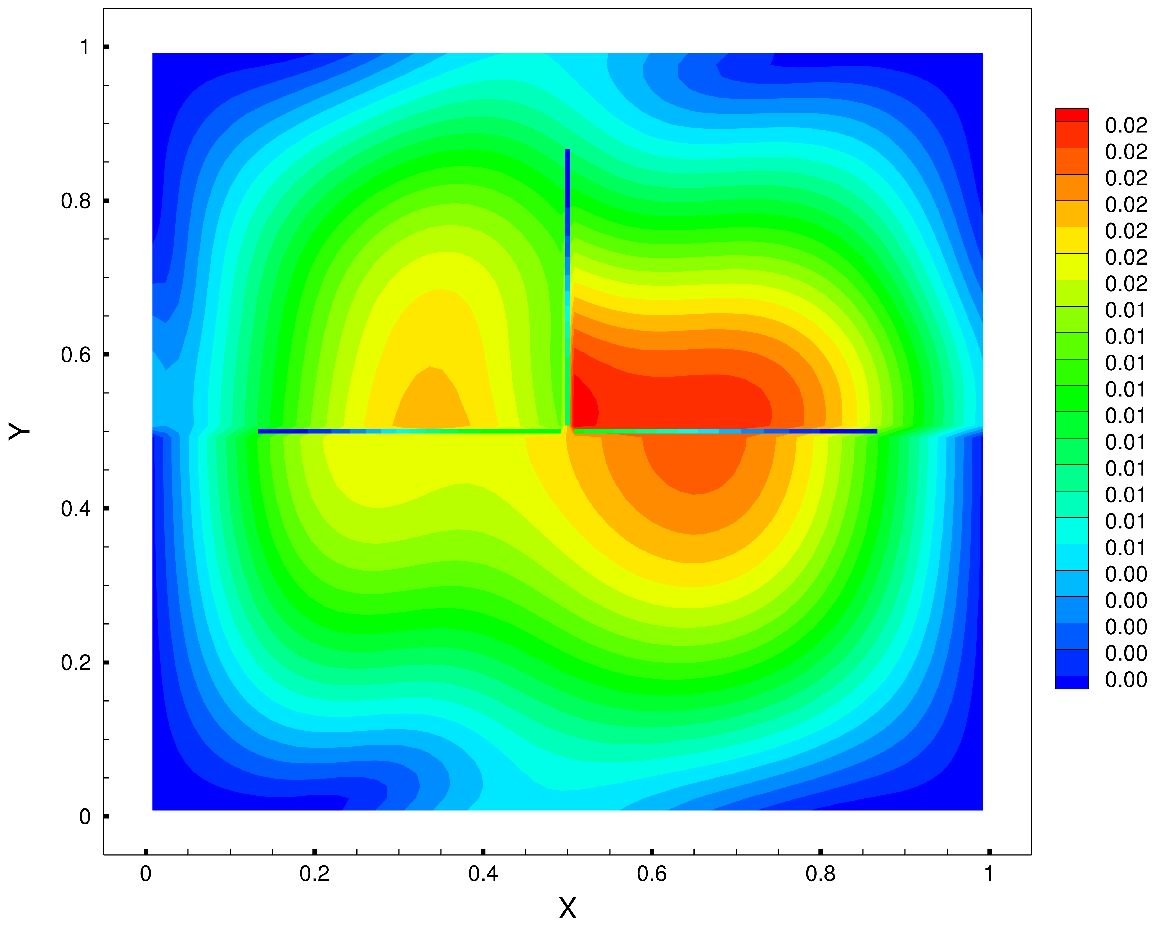}	}
	\caption{Numerical results of the pressure and displacement fields: (a) Pressure field ($p$ \& $p_c$); (b) Displacement magnitude \& aperture.}
	\label{pc_ap}
\end{figure}
Overall, the numerical results in Figures \ref{x=0.5}--\ref{pc_ap} provide a consistent verification of the proposed discretization for the coupled matrix-fracture poroelastic problem. The line-cut comparisons in Figures \ref{x=0.5} and \ref{y=0.5} show that the computed matrix pressure $p$, displacement components $w_1$ and $w_2$, and displacement magnitude $|\mathbf{w}|$ agree well with the exact solutions along the selected sections. Figure \ref{Fracture-line} shows close agreement between the numerical and exact fracture aperture and pressure. For $\mathcal{C}_1$, the arc-length coordinate is measured from the left outer tip to the right outer tip, so the junction $J$ is located at $s=0.375$, whereas the coordinate on $\mathcal{C}_2$ is measured from $J$ toward its outer tip.
The curves for $\mathcal{C}_2$ do not start from zero because its arc-length coordinate is measured from the junction $J$, where the aperture and pressure are generally nonzero. The contour plots in Figures \ref{displacement_contour}--\ref{pc_ap} provide a global visualization of the pressure and displacement fields, confirming that the numerical solution reproduces the expected spatial patterns and the matrix-fracture coupling behavior.
\section{Conclusions}
In this paper, we proposed a robust splitting scheme for a linear poroelastic model with fractures. For the continuous model, we established an energy dissipation law, which is consistent with the second law of thermodynamics. We then introduced a semi-discrete splitting scheme in time and proved its energy stability. Based on this time discretization, a fully discrete splitting method was developed by using Taylor--Hood finite elements for the displacement and total pressure, together with Lagrange finite elements for the matrix and fracture pressures. The corresponding optimal convergence estimates were derived. Compared with a fully coupled method, the proposed scheme is computationally more efficient since the subproblems are decoupled and no iterative coupling procedure is required at each time step. In addition, the four-field mixed-dimensional formulation, combined with stable finite element spaces, helps alleviate volumetric locking as $\lambda\to\infty$ and suppress pressure oscillations. Numerical results are reported to support the theoretical analysis of stability and convergence.

\vspace{1em}
\noindent\textbf{Data Availability}
\vspace{1em}

Data will be made available on request.

\vspace{1em}

\noindent\textbf{Declaration of competing interest}
\vspace{1em}

The authors declare that they have no known competing financial interests or personal relationships that could have appeared to influence the work reported in this paper.

\vspace{1em}
\noindent\textbf{Acknowledgment}
\vspace{1em}

The work of D. Tang and M. Feng was supported by the National Natural Science Foundation of China (Grant No. 11971337).

The work of L. He was supported by the Key Project of Natural Science Foundation of Sichuan Province (Grant No. 2026NSFSCZY0043) and the Natural Science Foundation of Sichuan Province (Grant No. 2026NSFSC0753).

The work of S. Sun was supported by the National Natural Science Foundation of China (Grant No. 12571466), the Fundamental Research Funds for the Central Universities, the Shanghai Magnolia Talent Fund (Innovation Talent Category) of Shanghai Municipal HR \& SS Bureau, and the Chang Jiang Scholars Program of the Ministry of Education of China.
\bibliographystyle{plain}
\bibliography{bibfile_All.bib}

\end{document}